\documentclass{amsart}

\usepackage{amsmath}
\usepackage{amssymb}
\usepackage{epic}
\usepackage{amsthm}
\usepackage{subfigure}
\usepackage{epsfig}
\usepackage{enumitem}
\usepackage{stmaryrd}

\usepackage{graphicx}

\newtheorem{thm}{Theorem}[section]
\newtheorem{lem}[thm]{Lemma}
\newtheorem{prop}[thm]{Proposition}
\newtheorem{cor}[thm]{Corollary}
\newtheorem{defn}[thm]{Definition}
\newtheorem{conj}[thm]{Conjecture}

\newcommand{\LSN}[2]{\sin(\theta_{#2}-\theta_{#1})}
 \newcommand{\obib}[1]{}
 
\newcommand{\axMonotone}{A4}
\newcommand{\axDegenerate}{A3}
\newcommand{\axCompare}{A1}
\newcommand{\axZero}{A2}

\newcommand{\Rin}[0]{\mathbf{Rin}(9)}
\newcommand{\Arr}{\mathbf{\mathcal{A}}}
\newcommand{\Brr}{\mathbf{\mathcal{B}}}
\newcommand{\Crr}{\mathbf{\mathcal{C}}}

\newcommand{\jle}{\preccurlyeq}
\newcommand{\sjle}{\prec}

\newcommand{\SN}[2]{S^{#2}_{#1}}
\newcommand{\BSN}[2]{\mathbf{{\boldsymbol{S}}^{#2}_{#1}}}

\newcommand{\PA}{A}
\newcommand{\NA}{-A}

\newcommand{\PPA}{A^{+}}
\newcommand{\ZA}{A^{0}}
\newcommand{\NNA}{A^{-}}

\newcommand{\Natural}{\mathbb{N}}
\newcommand{\Nat}{\mathbb{N}}
\newcommand{\Real}{\mathbb{R}}

\newcommand{\Mat}{\mathcal{M}}
\newcommand{\DualMat}{\mathcal{M^*}}
\newcommand{\Circ}{\mathcal{C}}
\newcommand{\CoC}{\mathcal{C^*}}

\newcommand{\ind}{\boldsymbol{1}}

\DeclareMathOperator{\rank}{rank}

\newcommand{\pow}{\mathfrak{p}}

\newcommand{\cupdot}{\ensuremath{\mathaccent\cdot\cup}}

\title{A New Proof of Pappus's Theorem}
\author{Jeremy J. Carroll}
\email{pinkboy@bluebottle.com}

\begin{document}


\date{25th April 2007}

\thanks{
Copyright \copyright~2007 Jeremy J. Carroll.
}
\thanks{
Copying this document is licensed under a 
Creative Commons Attribution 2.0 UK: England \& Wales License.
See http://creativecommons.org/licenses/by/2.0/uk/ for details.
}
\thanks{
The author invites correspondence related to this paper, including
discussion of the results and suggested corrections.
}

\begin{abstract}
Any stretching of Ringel's non-Pappus pseudoline arrangement
when projected into the Euclidean plane,
implicitly contains a particular arrangement
of nine triangles.
This arrangement has a complex constraint
involving the sines of its angles.
These constraints cannot be satisfied
by any projection of the initial arrangement.
This is sufficient to prove Pappus's theorem.
The derivation of the constraint is via
systems of inequalities arising
from the polar coordinates of the lines.
These systems are linear in $\mathbf{r}$ for any given $\boldsymbol{\theta}$,
and their solubility can be analysed in terms of the signs of determinants.
The evaluation of the determinants is via a normal form for sums
of products of sines, 
giving a powerful system of trigonometric identities.
The particular result is generalized to arrangements
derived from three edge connected totally cyclic directed graphs,
conjectured to be sufficient for a complete analysis of angle 
constraining
arrangements of lines, and thus a full response to Ringel's slope 
conjecture.
These methods are generally applicable to the realizability problem
for rank 3 oriented matroids.
\end{abstract}

\maketitle

\section{Introduction}
A more accurate, but less snappy, title for this paper might have been: ``New Approaches
to Angles and Arrangements of Lines and Pseudolines applied to Pappus's Theorem''.
This paper does contain a new proof of Pappus's theorem, but it is
fairly laborious and ugly. In particular, it is strangely asymmetric given the
beauty and symmetry of the theorem being proved. The reader only wishing to
be convinced of the truth of Pappus's theorem is best advised to go elsewhere.

The hope is that the reader will find:
\begin{itemize}
\item 
A new appreciation of the complex constraints between angles
in line arrangements, without regard to any distances in the arrangement.
\item
An awareness of the power of polar coordinates in addressing the
pseudoline stretching problem.
\item
New techniques for decomposing pseudoline arrangements into partial
arrangements by considering the orientations of 
only some of the triangles.
\item
A normal form for sums of products of sines, useful
for finding complex trigonometric identities.
\item
A concept of twisted graph, allowing
the derivation of angle constraining arrangements of lines
from three edge connected graphs.
\item
An understanding of the potential for solving problems set in
the projective plane by an analysis in the Euclidean plane.
\item
And, a new proof of Euclid's porism, 
usually known as Pappus's theorem.
\end{itemize}

\begin{figure}[htp]
  \begin{center}
    \subfigure[The main result (E)]{\label{f:main}\includegraphics[width=.45\textwidth]{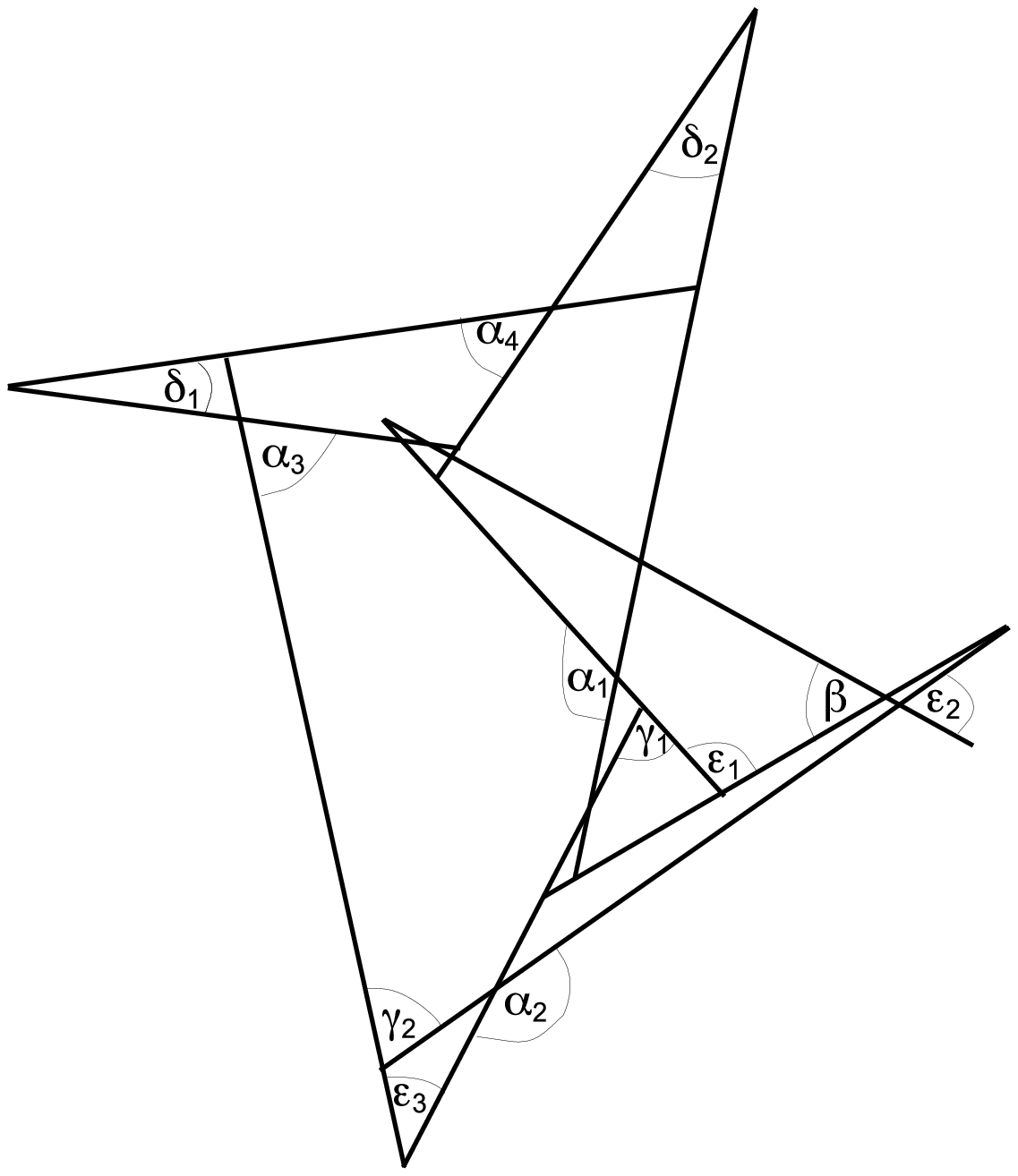}}
    \subfigure[$\Rin$ (from \cite{grunbaum:straight}) (P)]{\label{f:grunbaum}\includegraphics[width=.45\textwidth]{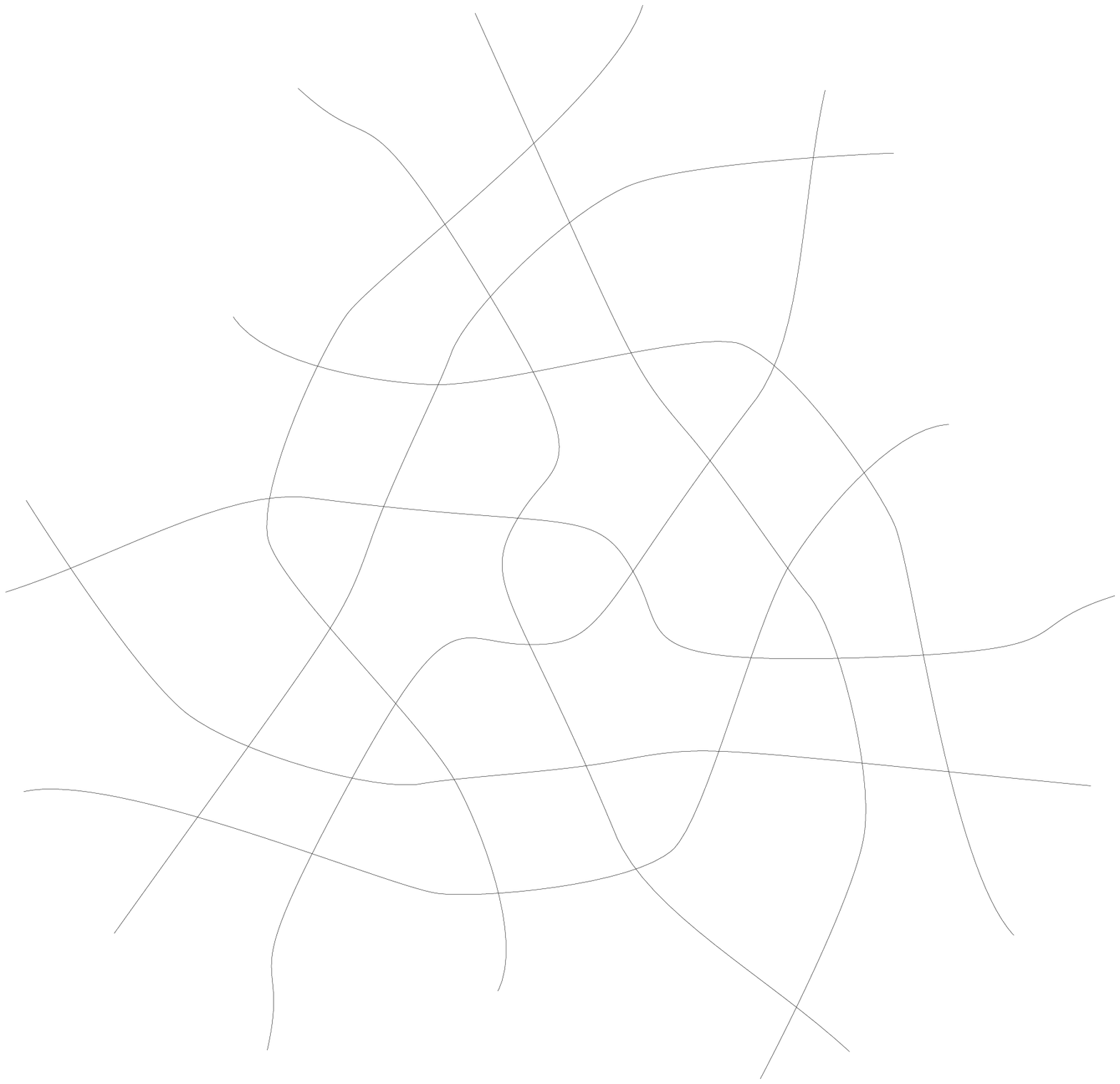}} 
  \end{center}
  \caption{}
  \end{figure}


\subsection{Main Result}
In figure~\ref{f:main}, we have:
\begin{multline}
\label{e:main}
\sin(\alpha_1)\sin(\alpha_2)\sin(\alpha_3)\sin(\alpha_4)\sin(\beta) \\
+ \sin(\beta)\sin(\gamma_1)\sin(\gamma_2)\sin(\delta_1)\sin(\delta_2) \\
> \sin(\delta_1)\sin(\delta_2)\sin(\epsilon_1)\sin(\epsilon_2)\sin(\epsilon_3)
\end{multline}

%

This result is sufficient to show that \obib{Ringel's} the 9-line non-Pappus pseudoline arrangement of
\cite{ringel:teilungen},
fig.~\ref{f:grunbaum}, cannot be stretched, i.e. drawn with straight lines.
Reversing Ringel's non-stretchability proof from Pappus, proves
Pappus from equation~(\ref{e:main}).
The main result is first proved fairly directly, 
from Motzkin's PhD thesis, and then as a special case of a general result
which gives similar conditions to an infinite class of diagrams, for example,
derivable from every cubic graph.

\subsection{Angles}
This paper studies the relationship between angles and the topology
of line arrangements in the Euclidean plane.
This presents a significant departure from previous approaches
to geometry, in which either the metric plays a central role, or
which abstract both from the metric and the protractor.
The relationships of interest are nontrivial relationships
constraining products of sines of the angles
which do not involve any of the distances. Some initial examples are shown in 
figures~\ref{f:circsaw3}, \ref{f:circsaw4} and~\ref{f:circsaw4ringel}. 
For each of these the following inequality holds:
\begin{equation}
\Pi_{i=1}^n \sin(\beta_i) > \Pi_{i=1}^n \sin(\alpha_i)
\end{equation}
This result is derived in \cite{carroll:sines}, by a simple application of the sine formula to 
each of the exscribed larger triangles.
\begin{figure}[htp]
\begin{center}
\begin{tabular}{{cc}}
\subfigure[n=3]{\label{f:circsaw3}\includegraphics[width=.4\textwidth]{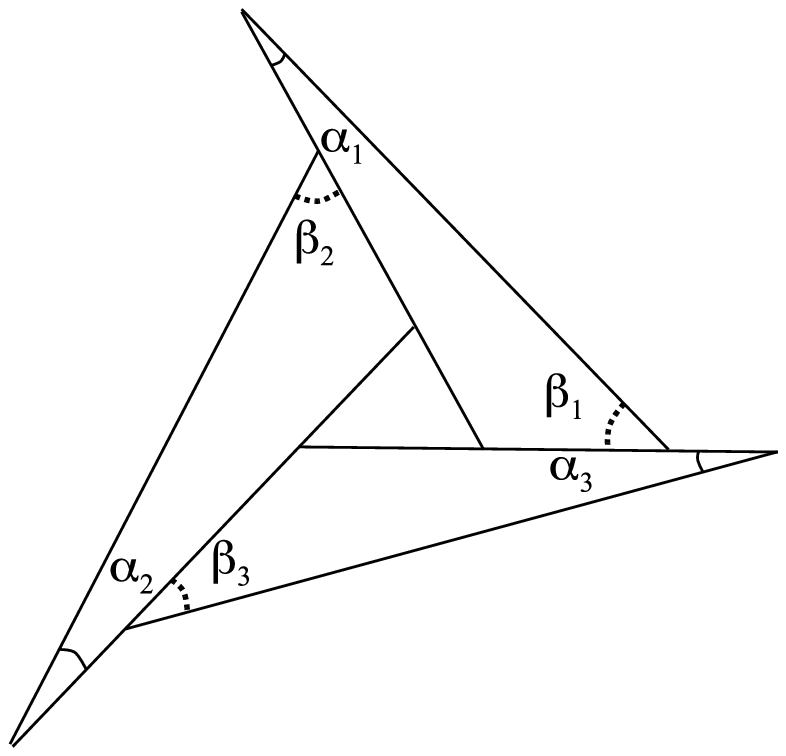}}&
\subfigure[n=4]{\label{f:circsaw4}\includegraphics[width=.4\textwidth]{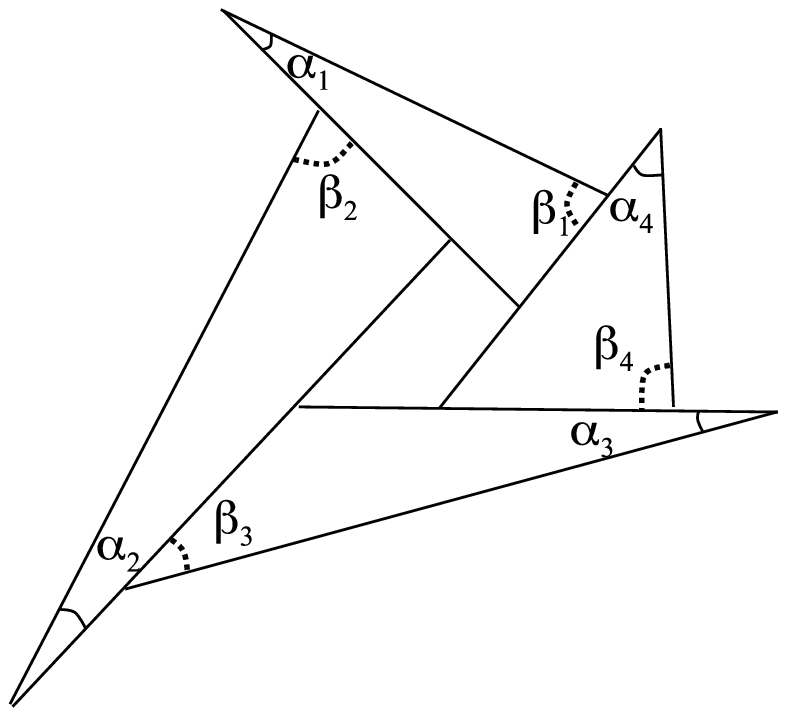}}\\
\subfigure[n=4]{\label{f:circsaw4ringel}\includegraphics[width=.4\textwidth]{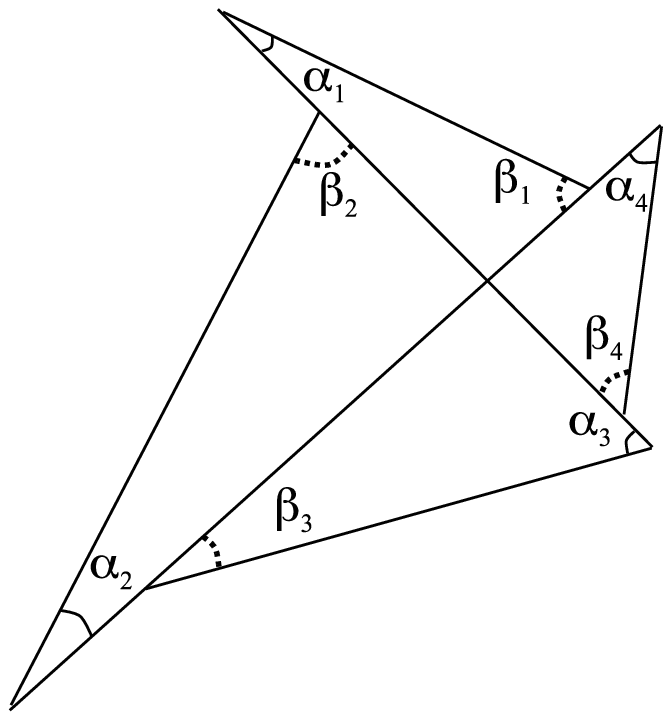}}&
\subfigure[Ceva]{\label{f:ceva}\includegraphics[width=.4\textwidth]{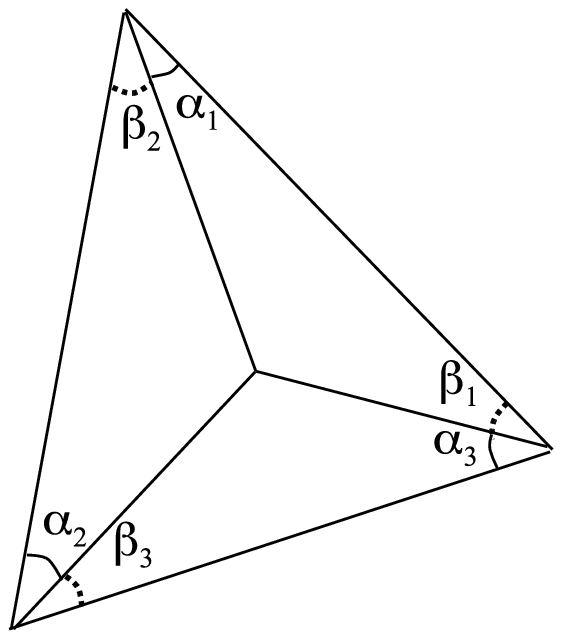}}
\end{tabular}
  \end{center}
  \caption{Circular Saws (E)}
  \label{f:circsaws}
\end{figure}
The first of these, when taken to the limit amounts to the trigonometric variation
of Ceva's theorem~\cite{ceva:lineis} \obib{, 1678, } that in figure~\ref{f:ceva}, 
$\Pi_{i=1}^3 \sin(\beta_i) = \Pi_{i=1}^3 \sin(\alpha_i)$.

The longterm goal of this work is to solve the pseudoline stretching problem: 
to provide a useable algorithm that, given a pseudoline arrangement, either
provides an equivalent line arrangement or provides a proof that none exists.

\section{Outline of Paper}
\label{s:outline}
To help the reader navigate this overly long paper,
we give a detailed outline.

The next section gives an overview of the broad context
of this study. 
Some specific conventions and notations are described
in section~\ref{s:notation}; while our specific approach
to polar coordinates of lines is given in section~\ref{s:polar}. 
The relationship between our use of polar coordinates 
and the chirotope of oriented matroids is given in
section~\ref{s:chirotope}, which should be omitted
by the reader unfamiliar with oriented matroids.

Results from the literature are given throughout the paper:
\begin{description}
\item[section~\ref{s:pseudolines}]
describes pseudoline arrangements, and the stretchability problem.
\item[section~\ref{s:motzkin}]
summarizes Motzkin's version of Carver's results
concerning system of strict linear inequalities.
\item[section~\ref{s:om}]
discusses the oriented matroids derived from directed graphs
and total orders.
\end{description}

The new contributions start with the statement
of the main theorem in section~\ref{s:maintheorem}.

This is used in section~\ref{s:proof} to prove
the nonstretchability of Ringel's non-Pappus
pseudoline arrangement. From this we derive
Pappus's theorem. This derivation is
included principally to justify the title of this paper.
The specialist reader may wish to skip it. The more general
reader may find that it relates the somewhat esoteric topic
of oriented matroid realizability with something more familiar.
Section~\ref{s:proof} starts with
a brief look at Pappus's own work, including,
notably, two of Pappus's own diagrams, again aimed at the more
general reader.

The first proof of the main theorem is given
in section~\ref{s:mainproof}.
This tedious proof involves the computation of the signs
of twenty two determinants of sub-matrices from a particular
nine by ten matrix, encapsulating fig.~~\ref{f:main}.

The computations in that proof depend 
on the manipulation of sums of products of sines.
Hence, section~\ref{s:normal} explores
such sums, giving a simple method of expanding them to normal
form, much simplifying section~\ref{s:mainproof}.
Unfortunately, the methods used to prove the correctness of the normal form
depend on multisets
and I was unable to find an appropriate treatment in the literature,
so that section~\ref{s:multisets} gives a quick overview,
extending the description from Wikipedia to permit a finite power multiset
operation. Thus section~\ref{s:normal} consists of two digressions,
and should probably be omitted at first reading, except the introductory
paragraphs (on the other hand, the digression is interesting in itself, and
a different reader may prefer to read only that section, and omit
the rest of the paper).

The most interesting new work concerns generalisations 
of the techniques used to prove the main theorem.
These are given in sections~\ref{s:twisted}, \ref{s:om2}
and~\ref{s:twistedsimplex}. The first of these introduces
the notion of a twisted graph, generalizing the
line arrangement studied in the main theorem. 
Section~\ref{s:twistedsimplex}
shows that results similar to the main theorem can be derived
for many twisted graphs, including all totally cyclic,
simply three edge connected, directed
cubic graphs. Conjectures of stronger results are made.

The final discussion, in section~\ref{s:future},
concerns
how these results
may be relevant to the pseudoline stretching problem,
giving (without proof), a cryptomorphic axiom system
for rank 3 acyclic oriented matroids, suited
for studying partial line arrangements in the Euclidean plane.
We suggest that every unrealizable rank 3 oriented matroid 
contains a line arrangement that is a counterexample
to Ringel's slope conjecture, thus showing that the proof
technique used to prove the unstretchability of the non-Pappus
arrangement is general.
We discuss future directions for this work.

The paper closes with a brief conclusion.

\section{Context and Related Work}

\subsection{Euclidean Geometry}

Euclidean geometry is the oldest area of mathematical study,
but is nowadays seen as essentially completed and not an area for research.
One of the claims of this paper is that angles have not been studied
adequately independently of the metric in Euclidean geometry (unlike the metric
independently of angles).
As well as the angular variant of Ceva's theorem,
there is a small amount of recent work on 
angles independent of distance, such as the study of the 
angles formed by $n$-points in the plane, specifically the
greatest least angle~\cite{jaudon:angles} and the least greatest angle~\cite{sendov:minimax}.

The ancients accounted for  linear constraints on angles, such as the
angle sum of a triangle, and the angles formed by a transversal of 
parallel lines. However, their emphasis was on results whose primary focus
is distance or area, such as Pythagoras' theorem.

Later, the introduction of Cartesian coordinates put a further emphasis
on distance, at the expense of angles. Trigonometry, of course,
does give angles a central role, but rarely to the exclusion of distance.

As modern geometry developed, either both angles and distances were retained
(e.g. hyperbolic geometry), or angles got abstracted away (projective geometry)  ,
or both distances and angles vanish in the abstraction.

\subsection{Pseudoline Arrangements}
Pseudolines were introduced by \obib{Levi in 1926~} \cite{levi:pseudolines},
who along with most authors (such as 
\obib{Ringel~} \cite{ringel:teilungen}), work in the projective plane.
A seminal paper is \obib{Gr\"unbaum's} \lq The Importance of Being Straight\rq~\cite{grunbaum:straight}.
A core problem in the study of pseudolines is stretchability:
given a pseudoline arrangement is there an equivalent line arrangement.

In this paper, because of the focus on angles,
we work primarily in the Euclidean plane.
This preference is shared with some
other authors such as \cite{felsner:number,
felsner:sweeps,
sharir:generalized,
agarwal:pseudoline,
shor:stretchability}.
The definition used by most of these is that a pseudoline is an $x$-monotone
curve in the Euclidean plane; and in an
arrangement of pseudolines every pair 
meet exactly once, at a point where they cross. This
definition commits to a Cartesian coordinate
system, whereas we work in polar coordinates. 
\obib{Shor~}\cite{shor:stretchability} works in the Euclidean plane and allows more
general pseudolines (``the image of a line under a homeomorphism of the plane")
and permits `parallel' pseudolines that do not meet.
In contrast, we follow
\obib{Felsner~} \cite{felsner:triangles}, and define pseudolines 
in the projective plane,
but work in the Euclidean plane.

We note that line arrangements (in the Euclidean plane) were a major focus of mathematics
for over a millennium. 

The pseudoline stretchability problem is known to be NP-hard~\cite{shor:stretchability}. Moreover, via the relationship
to oriented matroids, it is known to be polynomial time equivalent
to the existential theory of the reals, (i.e. multivariate polynomial programming)~\cite{mnev:universality}.

\subsection{Oriented Matroids}
The projective plane appears most forcefully in the correspondence between
pseudolines and oriented matroids given by \obib{Folkman and Lawrence's} 
the topological 
representation theorem of \cite{folkman:oriented}. In this, 
the problem of  pseudoline stretchability 
is equivalent to the problem of rank 3 oriented matroid realizability.
Every rank 3 oriented matroid can be represented by a pseudoline
arrangement. Realizable oriented matroids can be represented
by a line arrangement.
Other authors study oriented matroids via the chirotope,
which, in line arrangements,
corresponds naturally to determinants of
homogeneous coordinates (in the projective plane). 
Most progress on oriented matroid realizability has been made
in such terms. For example, Bokowski's algorithm for finding
biquadratic final polynomials can be applied to the oriented
matroid $\Rin$ to prove its nonrealizability. The main result
of this paper is equivalent. The invaluable standard reference
for oriented matroids is~\cite{bjorner:oriented}.

\subsection{Ringel's Slope Conjecture}
\cite{ringel:teilungen} conjectured that in an arrangement of lines in general position the slopes 
could be arbitrarily prescribed. This conjecture was disproved first by \cite{lasvergnas:order} 
using oriented matroid techniques over a 32 point dual construction. \cite{richter:topology} 
improved this to give a 6 line counterexample (fig.~\ref{f:circsaw4ringel}), still 
demonstrating the slope constraint using oriented matroid techniques. \cite{felsner:zonotopes} give a different proof of the counterexample using higher Bruhat orders. \cite{carroll:sines} demonstrates
the result using schoolbook geometry. The main theorem of this paper is another counter-example.
The theme of the more general analysis of this paper, 
sections~\ref{s:twisted} to~\ref{s:future}, 
is the search for all minimal counter-examples to
this conjecture.

\subsection{Venn Triangles}
This paper builds on my earlier work\footnote{
\cite{carroll:between, carroll:saws, carroll:drawing}
} which suffers from complete ignorance of the field.
This was motivated by a specific pseudoline stretching problem, relating
to drawing diagrams of 6-Venn triangles~\cite{carroll:venn}. While I produced
a pseudoline stretching program that stretched the diagrams 
of interest to me at that time,
I could not adequately explain why it worked. This paper is a move towards an explanation.
The key insight of my earlier work, that is not in the literature, is
that the use of polar coordinates allow the pseudoline stretching problem
to be divided into two separate phases: first determine 
the $\boldsymbol{\theta}$ coordinates of the lines, and second, 
use linear programming techniques to determine the $\mathbf{r}$ coordinates. 
The latter step is a solved problem, although
issues are presented by the 
linear program being over the reals rather than the rationals; 
and by the extensive use of strict inequalities.
Thus, my primary interest is the first problem
of what are the nonlinear constraints placed on angles by line arrangements,
such as those illustrated by figures~\ref{f:main} and~\ref{f:circsaws}.

\section{Notation and Conventions}
\label{s:notation}

\subsection{Geometry of Diagrams}
Some of the diagrams, e.g. fig.~\ref{f:main}, 
illustrate the Euclidean plane,
and the choice of the line at infinity is significant.
These are marked with an (E). 
Others, e.g. fig.~\ref{f:grunbaum}, illustrate the projective
plane, and no significance should be
read into the particularly choice of the projection
used for the illustration. These are marked with a (P).
 Figs.~\ref{f:carroll} and~\ref{f:carroll2} on page~\pageref{f:carroll},
 explicitly show the line at infinity,
as an oval around the diagram.

\subsection{Graphs, Directed Graphs}
We use simple graphs and directed graphs, both restricted
to be loop free and without parallel or antiparallel edges.
Fig.~\ref{f:graph} on page~\pageref{f:graph}
shows a simple graph and a directed graph, the latter also intended to have 
a Euclidean reading.

\subsection{Polar Coordinates}
\label{s:s:polar}
We use polar coordinates extensively. 
They always refer to lines rather than points. 
A pair of polar coordinates $(r_i,\theta_i)$ refer to the $i$th line,
with $r_i$ being the perpendicular distance from the origin to the
line, and $\theta_i$ is the angle that the perpendicular makes with 
the polar direction.

In diagrams, we always place the origin at the far lower edge of the plane,
in an unbounded region, which is not cut by any of the lines (or their extensions),
in the picture.
We represent the origin either: as in figure~\ref{f:triangle}, like:\\
\begin{center}
\includegraphics[width=.06\textwidth]{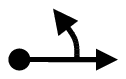}\\
\end{center}
to indicate its exact position; or, as in figure~\ref{f:straight}, like:\\
\begin{center}
\includegraphics[width=.15\textwidth]{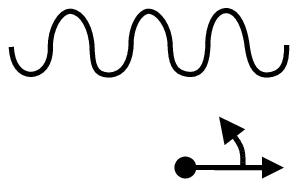}\\
\end{center}
to indicate that the origin lies directly below the indicated position,
sufficiently far to lie in an unbounded region.

This choice of positioning ensures that in the diagrams $r_i > 0$ 
and $0< \theta_i < 180$ for all lines $i$, and is formally
validated in lemma~\ref{l:origin}.

\subsection{Sines}
This paper deals extensively with the relationships between the sines
of different angles in line arrangements. In terms of the polar 
coordinates such sines are $\LSN{i}{j}$.
We abbreviate this as $\SN{i}{j}$. i.e.
\begin{equation}
\SN{i}{j}=\LSN{i}{j}
\end{equation}

\subsection{Partial and Total Orders}
We use $\sjle$ to represent partial orders, and $<$ to represent total orders.
The distinction is 
particularly pertinent in sections~\ref{s:twisted} to~\ref{s:twistedsimplex}.

\subsection{Matrices}
We make extensive use of matrices with a specific form. 
Each row has three non-zero entries.
If the non-zero columns are $i, j, k$ then the non-zero values are 
either: $\LSN{j}{k}, -\LSN{i}{k}, \LSN{i}{j}$
or  $-\LSN{j}{k}, \LSN{i}{k}, -\LSN{i}{j}$. In all cases the sines themselves are
positive. 
For example,
\begin{equation}
M = \left(\begin{smallmatrix}
&\SN{4}{6}&&-\SN{2}{6}&&\SN{2}{4}\\
\SN{2}{3}&-\SN{1}{3}&\SN{1}{2}&&&\\
-\SN{4}{5}&&&\SN{1}{5}&-\SN{1}{4}&\\
&&-\SN{5}{6}&&\SN{3}{6}&-\SN{3}{5}
\end{smallmatrix}
\right)
\end{equation}
abbreviates:
\begin{equation}
M = \left(\begin{smallmatrix}
0&\LSN{4}{6}&0&-\LSN{2}{6}&0&\LSN{2}{4}\\
\LSN{2}{3}&-\LSN{1}{3}&\LSN{1}{2}&0&0&0\\
-\LSN{4}{5}&0&0&\LSN{1}{5}&-\LSN{1}{4}&0\\
0&0&-\LSN{5}{6}&0&\LSN{3}{6}&-\LSN{3}{5}
\end{smallmatrix}
\right)
\end{equation}

For large such matrices, for convenience, we add explicit row and column labels, e.g.
\begin{equation}
M = \begin{pmatrix}
&(1)&(2)&(3)&(4)&(5)&(6)\\
(A)&&\SN{4}{6}&&-\SN{2}{6}&&\SN{2}{4}\\
(B)&\SN{2}{3}&-\SN{1}{3}&\SN{1}{2}&&&\\
(C)&-\SN{4}{5}&&&\SN{1}{5}&-\SN{1}{4}&\\
(D)&&&-\SN{5}{6}&&\SN{3}{6}&-\SN{3}{5}
\end{pmatrix}
\end{equation}
We represent submatrices formed by columns 
using the notation $M\left[2,3,6\right]$. Thus:
\begin{equation}
M\left[2,3,6\right] = \begin{pmatrix}
\SN{4}{6}&0&\SN{2}{4}\\
-\SN{1}{3}&\SN{1}{2}&0\\
0&0&0\\
0&-\SN{5}{6}&-\SN{3}{5}
\end{pmatrix}
\end{equation}

\section{Pseudolines}
\label{s:pseudolines}

Most authors, following Levi, work in the projective plane.
\cite{grunbaum:straight}:
\begin{quote}
an {\em arrangement of pseudolines} is a finite family [{\em with at least two member}]
of simple closed curves in the projective plane, such that every two curves have exactly
one point in common, each crossing the other at this point, while no point is common
to all the curves.
\end{quote}
A pseudoline $L$, like a projective line, is such
that $\mathbb{P}^2 \setminus L$ is connected.
A family of non-coincident lines in the projective plane satisfy this definition,
so that every {\em arrangement of lines} is an arrangement of pseudolines. 

Two pseudoline arrangements are {\em equivalent} if there is
a homeomorphism from one to the other.
A pseudoline arrangement is {\em stretchable} if it is equivalent
to a line arrangement.

Our use of polar coordinates commits to a Euclidean viewpoint,
and not the usual one of $x$-monotone lines. Following 
\cite{felsner:triangles} we use Levi's definition, 
but fix a line at infinity, with a homeomorphism
followed by a projection, and then work in the Euclidean plane.

\section{Polar Coordinates of Line Arrangements}
\label{s:polar}
\subsection{Polar Coordinates}
\begin{figure}[htp]
  \begin{center}
\begin{tabular}{{ccc}}
    \subfigure[A positive triangle]{\label{f:triangle}\scalebox{0.3}{\includegraphics{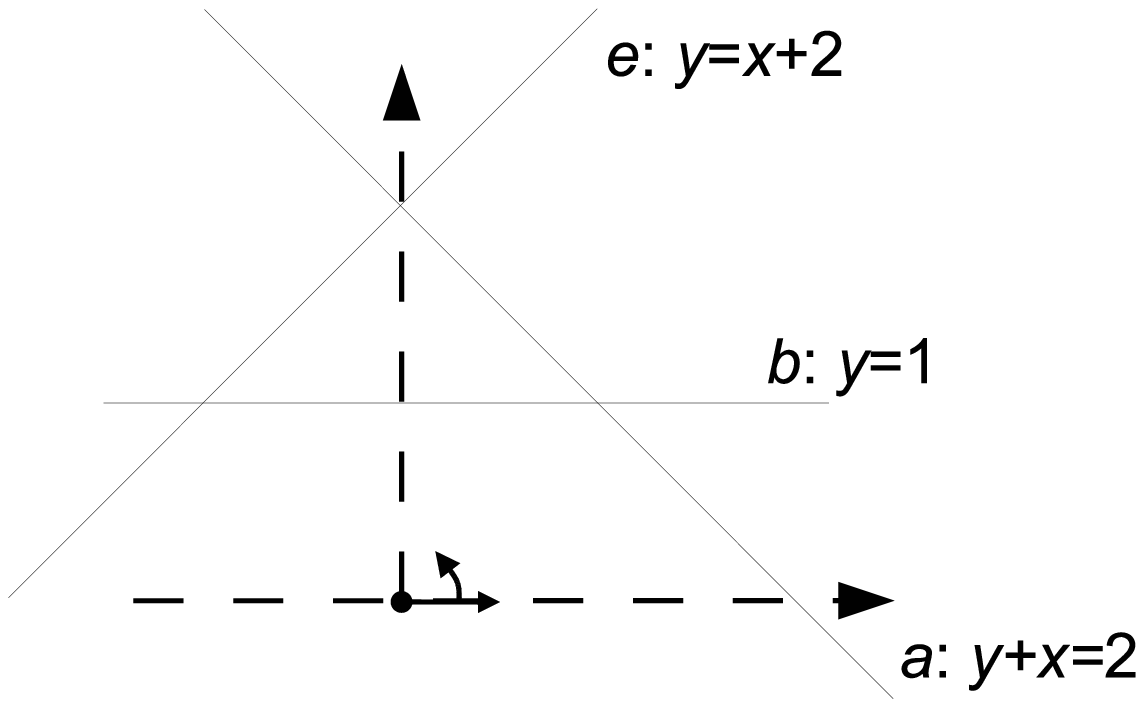}}} &
    \subfigure[Coincident lines]{\label{f:point3}\scalebox{0.3}{\includegraphics{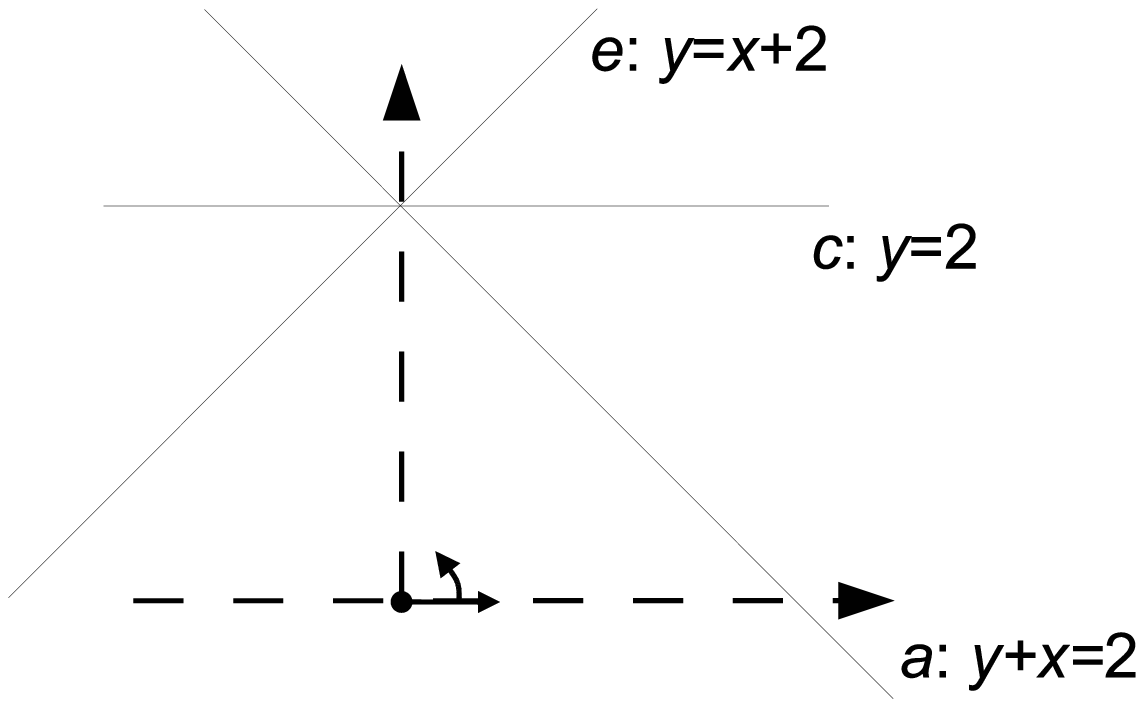}}}&
    \subfigure[A negative triangle]{\label{f:trianglex}\scalebox{0.3}{\includegraphics{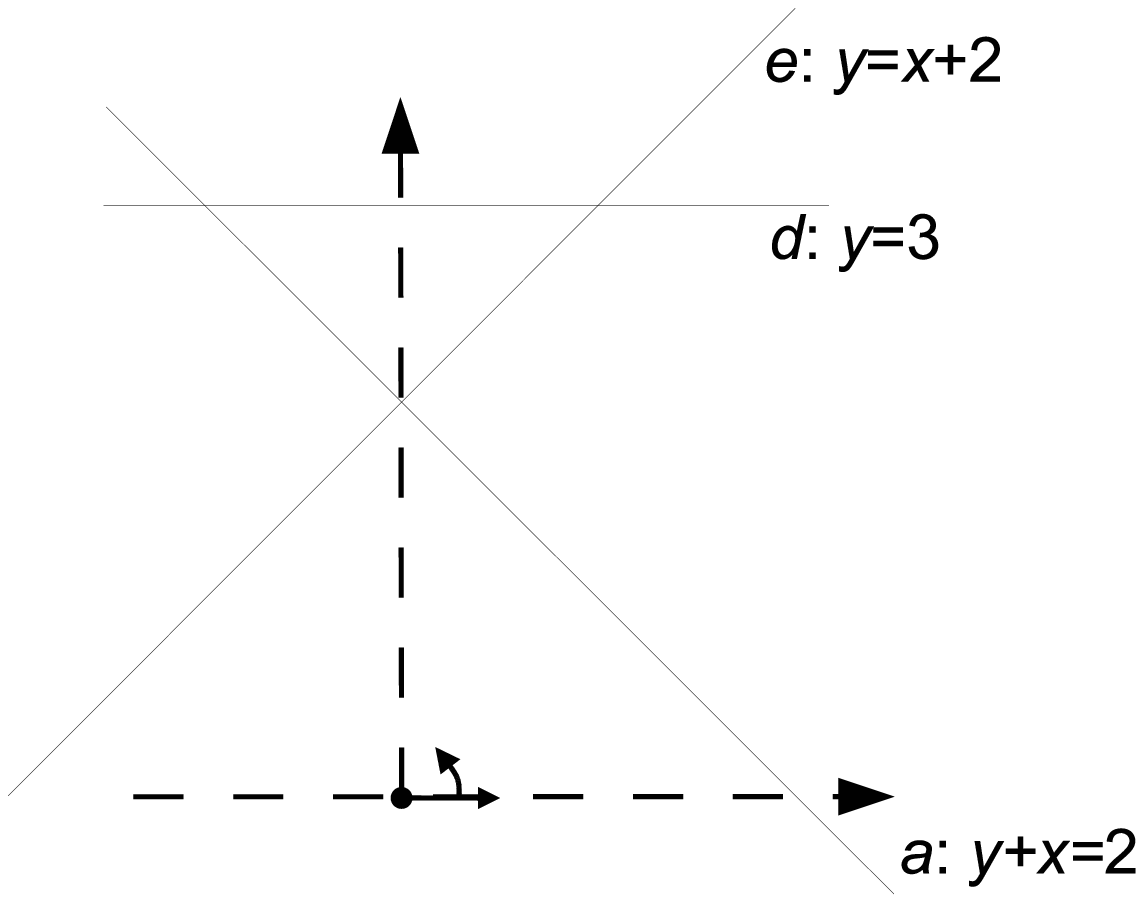}}} 
\end{tabular}
\end{center}
  \caption{Arrangements of three lines (E)}
  \label{f:lines}
\end{figure}
The polar coordinates of a line are given as the coordinates of the point
where a perpendicular to the origin can be drawn.
In figure~\ref{f:triangle}, for line (a) the point has cartesian coordinates $(1,1)$,
or polar coodinates $(\surd{2}, 45)$. Similarly, line (b) is given by $(1,90)$, and line (e)
by $(\surd{2},135)$.

For any three coincident lines as in fig.~\ref{f:point3},
if the lines have coordinates $(r_i,\theta_i)$, with $i$ from 1 to 3,
and $\theta_i$ increasing,
the following identity holds:
\begin{equation}
\label{e:coincident}
r_1\LSN{2}{3}-r_2\LSN{1}{3}+r_3\LSN{1}{2} = 0
\end{equation}

If the three lines form a {\em positively oriented} triangle, with the origin in 
an unbounded face with 
three edges, as in fig.~\ref{f:triangle}, then this identity becomes an inequality:
\begin{equation}
\label{e:ptriangle}
r_1\LSN{2}{3}-r_2\LSN{1}{3}+r_3\LSN{1}{2} > 0
\end{equation}

If the three lines form a {\em negatively oriented} triangle, with the origin in 
an unbounded face with 
two edges, as in fig.~\ref{f:trianglex}, then the inequality is reversed in sign:
\begin{equation}
\label{e:ntriangle}
-r_1\LSN{2}{3}+r_2\LSN{1}{3}-r_3\LSN{1}{2} > 0
\end{equation}

These results can be proved directly using schoolbook geometry.
\cite{carroll:between} uses trilinear coordinates \cite{coxeter:applications, plucker:system}.

We do not consider the case when the origin is inside a triangle,
preferring to only use coordinate systems with origins towards the edge of the plane, 
to ensure
that this does not happen. We also choose the coordinate system (particularly the location
of the origin), so that $0 < \theta < 180$ for all angles $\theta$ of interest.
In the next subsection, we derive both the above inequalities,
and formalize the choice of coordinate system in terms of oriented matroid theory.
This can be skipped by the uninterested reader.

\subsection{Chirotopes and the Choice of Polar Origin (optional)}
\label{s:chirotope}
Given a line arrangement in the Euclidean plane, we have argued above that the polar origin
can always be placed in such a way that all the lines have angles between 0 and 180.
In this section, we formalise the argument in terms of 
the chirotope of a rank 3 acyclic oriented matroid.
Given such a line arrangement, indexed by a set $X$, we can form a set
$E = X \cup \{ \omega \}$, and take homogenous coordinates for $E$, with 
$\omega$ representing
the line at infinity. We label some face adjacent to $\omega$ and not between parallel lines,
as the positive tope, and hence choose an acylic orientation for the matroid,
along with a particular realization.
We can then apply the following lemma, which shows how to compute polar 
coordinates with the desired property,
and hence locate a polar origin.

\begin{lem}
\label{l:origin}
For a rank 3 acyclic oriented matroid $\Mat$, with chirotope  
$\chi : E^3 \rightarrow \{ -1, 0, 1 \}$
with $E = X \cupdot \{ \omega \}$, with $\omega$, not being a coloop,
nor parallel or antiparallel,
to any other element,
and with distinct 
positive cocircuits $A$, $B$ neither containing $\omega$,
such that for all $a \in A \setminus B$ and $b \in B \setminus A$,
$\chi(\omega,a,b) = 1$,
and  with a realization
given by $\mathbf{v}_e = \left( \begin{smallmatrix}
x_e\\
y_e\\
z_e
\end{smallmatrix}\right) \in \Real^3$, for each $e \in E$,
then the $\mathbf{v}_e$ can be chosen such that:
\begin{enumerate}
\item
$\mathbf{v}_\omega = \left( \begin{smallmatrix}
0\\
0\\
1
\end{smallmatrix}\right)$
\item
$x_a^2 + y_a^2 = 1$ for all $a \in X$
\item 
There are $\theta_a \in [0,360)$, such that
$x_a = -cos(\theta_a)$, $y_a = -sin(\theta_a)$ for each $a \in X$.
\item 
There are $\theta_a \in [0,180)$, such that
$x_a = -cos(\theta_a)$, $y_a = -sin(\theta_a)$ for each $a \in X$.
\item 
There are $\theta_a \in (0,180)$, such that
$x_a = -cos(\theta_a)$, $y_a = -sin(\theta_a)$ for each $a \in X$.
\item
$z_a > 0$ for each $a \in X$
\item
There are 
$(r_a,\theta_a)_{a \in X} \in \Real^+ \times (0,180)$
such that equations (\ref{e:coincident}, \ref{e:ptriangle}, \ref{e:ntriangle})  hold,
depending on whether the corresponding value of $\chi$ is 0, 1 or -1, respectively.
\item
$(r_a,\theta_a)_{a \in X}$ are polar coordinates for the realization,
given an appropriate origin.
\end{enumerate}
\end{lem}
\begin{proof}
If this is not true, then
for one of the claims we can find a realization
$\mathbf{v}_e$
that satisfies the previous claims, and there is no
realization that satisfies both the previous claims
and the new claim. We will show this leads to a contradiction,
by constructing a realization $\mathbf{v^\prime}_e$.

At each stage, we do one of:
\begin{itemize} 
\item
take $\mathbf{v^\prime}_e = \mathbf{v}_e $ for all $e \in E$.
\item
Give a matrix $P$ with positive determinant, and take
$\mathbf{v^\prime}_e = P \mathbf{v}_e $ for all $e \in E$.
\item
Give $\lambda_e > 0$ for each $e \in E$, and take 
$\mathbf{v^\prime}_e = \lambda_e \mathbf{v}_e $
\end{itemize}
Each of these steps leaves the signs of the subdeterminants
unchanged, so that $\mathbf{v^\prime}_e$ is a realization
of $\Mat$.

We use a total order $\leq$ over $E\setminus \{ \omega \}$, defined by:
\begin{equation}
x \leq y \text{ if and only if } \chi(\omega,x,y) \geq 0
\end{equation}
This is transitive and reflexive since $\chi(\omega, x, y)$ is
the chirotope of the acyclic rank 2 oriented matroid, $\Mat \setminus \{ \omega \}$.

\begin{enumerate}
\item
From the two cocircuits we can find
$a,b \in X$, with $a < b$.
Take $P = \begin{pmatrix}
\mathbf{v}_a & \mathbf{v}_b  & \mathbf{v}_\omega
\end{pmatrix}^{-1}$.
\item
Take 
\begin{equation}
\lambda_e = \begin{cases}
1_e&e = \omega\\
(x_a^2 + y_a^2)^{-1/2} & \text{otherwise}
\end{cases}
\end{equation}
\item 
Unchanged.
\item
Take $a \in A \setminus B$, then for all $b \in X$, $a \leq b$.
Take 
\begin{equation}
P = \begin{pmatrix}
cos(\theta_a)&sin(\theta_a)&0\\
-sin(\theta_a)&cos(\theta_a)&0\\
0&0&1
\end{pmatrix}
\end{equation}
Since $\chi(\omega,a,b) = 1$, for  $b \in X$,
we have $\sin(\theta^\prime_b-\theta^\prime_a) > 0$, and $\theta^\prime_a = 0$,
this
ensures that $\theta^\prime_b < 180$.
\item 
For some sufficiently small $\epsilon$, take
\begin{equation}
P = \begin{pmatrix}
cos(\epsilon)&-sin(\epsilon)&0\\
sin(\epsilon)&cos(\epsilon)&0\\
0&0&1
\end{pmatrix}
\end{equation}
\item
For every $a \in X$, $y_a = - sin(\theta_a) < 0$.
So we can choose some $\mu$ sufficiently large such that,
 for all $a \in X$,  $\mu > z_a / y_a$
so that $z_a - \mu y_a > 0$.
Take
\begin{equation}
P = \begin{pmatrix}
1&0&0\\
0&1&-\mu\\
0&0&1
\end{pmatrix}
\end{equation}
\item
Take $r_a = z_a$.
\item
Nothing to prove.
\end{enumerate}
\end{proof}

While the conditions on this theorem seem a bit restrictive, they do not
hinder our purposes,
for a Euclidean line arrangement, the process of adding $\omega$ as the line
at infinity, ensures the constraints on $\omega$ hold; and we can then find an
acyclic orientation with an appropriate positive tope, such that
$X$ is a 
positive vector, and satisfying the conditions on the positive cocircuits $A$ and $B$.
The only resulting constraint is that we cannot place the origin between parallel lines,
which is obvious.

\section{Statement of Main Theorem}
\label{s:maintheorem}
Using the notion of positively and negatively oriented triangles, we
can now formally state the main theorem, which is illustrated
in figure~\ref{f:straight}. The labels correspond to the figures
in the next section: the numbers labelling the lines, the letters
labelling shaded regions.

\begin{figure}[htbp]
\begin{center}
\includegraphics[width=.7\textwidth]{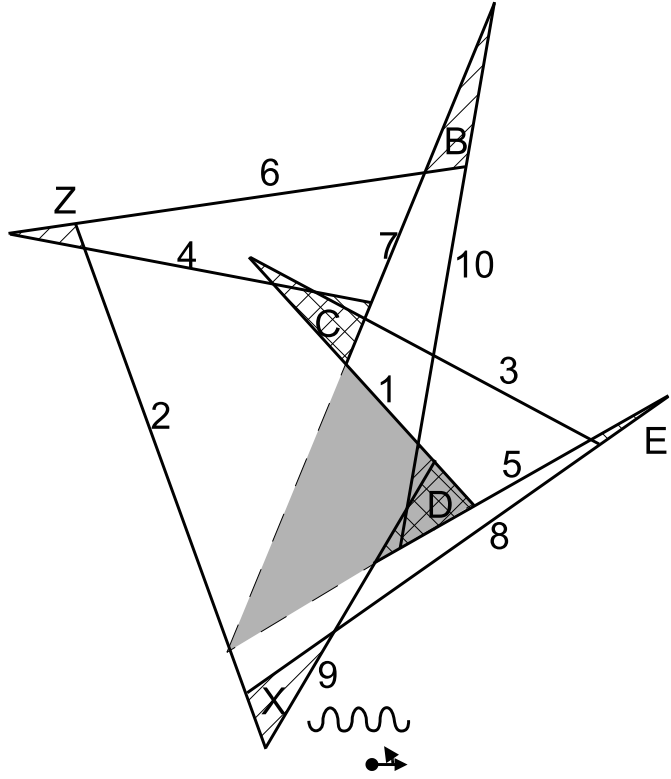}
\caption{Illustrating the main theorem (E)}
\label{f:straight}
\end{center}
\end{figure}

\begin{thm}
\label{t:main}
In any arrangement of 10 lines with polar coordinates 
$\{ (r_i,\theta_i): i = 1, \ldots 10\}$, in which the triangles 
$ (2,4,6)$,
$(1,5,9)$,
$(1,5,10)$,
$(1,5,7)$
are positively
oriented,
and the triangles
$ (1,3,7)$,
$ (1,4,7)$,
$(3,5,8)$,
$(2,8,9)$,
$ (6,7,10)$,
are negatively oriented,
we have:
\begin{multline}
\label{e:main2}
\LSN{8}{9}\LSN{1}{10}\LSN{2}{4}\LSN{3}{5}\LSN{6}{7} \\
+\LSN{4}{6}\LSN{1}{9}\LSN{7}{10}\LSN{3}{5}\LSN{2}{8} \\
- \LSN{4}{6}\LSN{3}{8}\LSN{7}{10}\LSN{2}{9}\LSN{1}{5}
> 0
\end{multline}
where, for each of the nine specified triangles, $(i,j,k)$,
we have $0 < \theta_i < \theta_j < \theta_k < 180$.
\end{thm}

Note, that except as specified, we do not require
$\theta_i < \theta_j$ when $i<j$. For example, $\theta_2$
may be less than or equal to $\theta_1$, as in fig.~\ref{f:straight}.

Two proofs are given, the first in section~\ref{s:mainproof},
is a direct computation specific to this figure. This 
approach is then generalised in section~\ref{s:twistedsimplex},
and this specific result is derived on page~\pageref{p:second}
from a more general theorem.

From the discussion in section~\ref{s:polar},
we see that such a figure can be drawn  if, and only if,
the following system of inequalities is soluble:
\begin{equation}
\label{e:linear}
\left(
\begin{smallmatrix}
 &          &-\SN{5}{8}& &\SN{3}{8}& & &-\SN{3}{5}& &\\ 
 &-\SN{8}{9}&          & & & & &\SN{2}{9}&-\SN{2}{8}&\\ 
 &\SN{4}{6} &          &-\SN{2}{6}& &\SN{2}{4}& & & & \\ 
 &          &          & & &-\SN{7}{10}&\SN{6}{10}& & &-\SN{6}{7}\\ 
-\SN{3}{7}& &\SN{1}{7}& & & &-\SN{1}{3}& & & \\ 
-\SN{4}{7}& & &\SN{1}{7}& & &-\SN{1}{4}& & & \\ 
\SN{5}{9}& & & &-\SN{1}{9}& & & &\SN{1}{5} & \\ 
\SN{5}{10}& & & &-\SN{1}{10}& & & & &\SN{1}{5}\\ 
\SN{5}{7}& & & &-\SN{1}{7}& & \SN{1}{5} &  & & 
\end{smallmatrix}
\right)
\boldsymbol{r} > \boldsymbol{0}
\end{equation}
The resulting  $\boldsymbol{r}$ gives the $r_i$ coordinates
of a drawing of the figure.

For fixed $\boldsymbol{\theta}$, this is a linear program in $\boldsymbol{r}$.
The solubility of linear programs is a
well-understood problem, and we spend section~\ref{s:motzkin}
reviewing some results from Motzkin's PhD thesis.

\section{Pappus's Theorem}
\label{s:proof}

\begin{figure}[htp]
  \begin{center}
\begin{tabular}{{cc}}
    \subfigure[Pappus folio 161v (E)]{\label{f:jones138}\includegraphics[width=.3\textwidth]{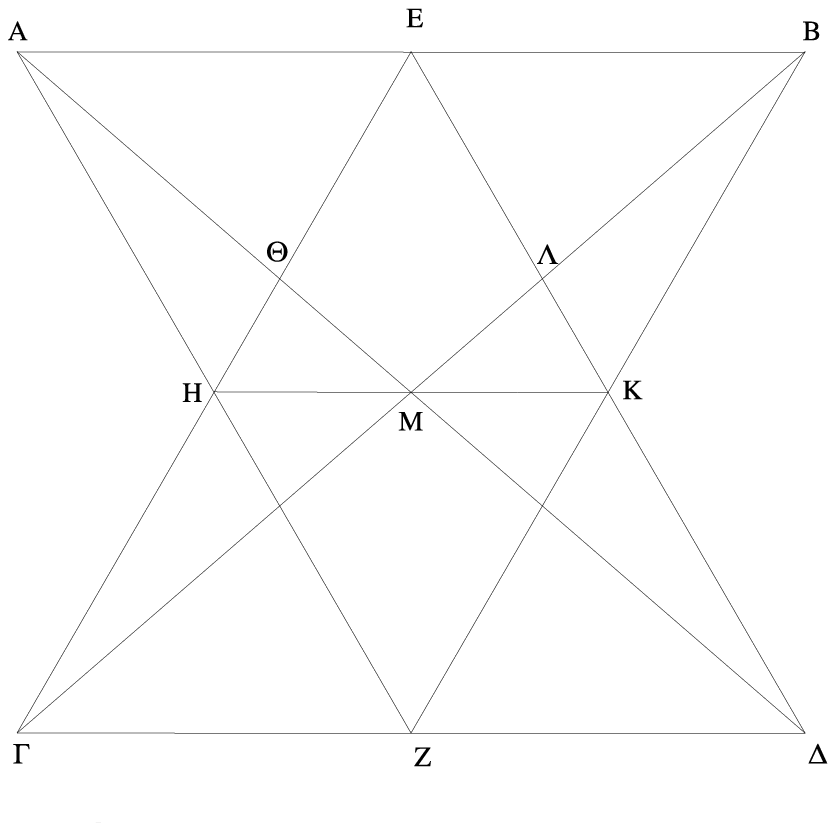}} &
      \subfigure[Pappus folio 162 \& Jones's fig. 139 (E)]{\label{f:jones139}\includegraphics[width=.45\textwidth]{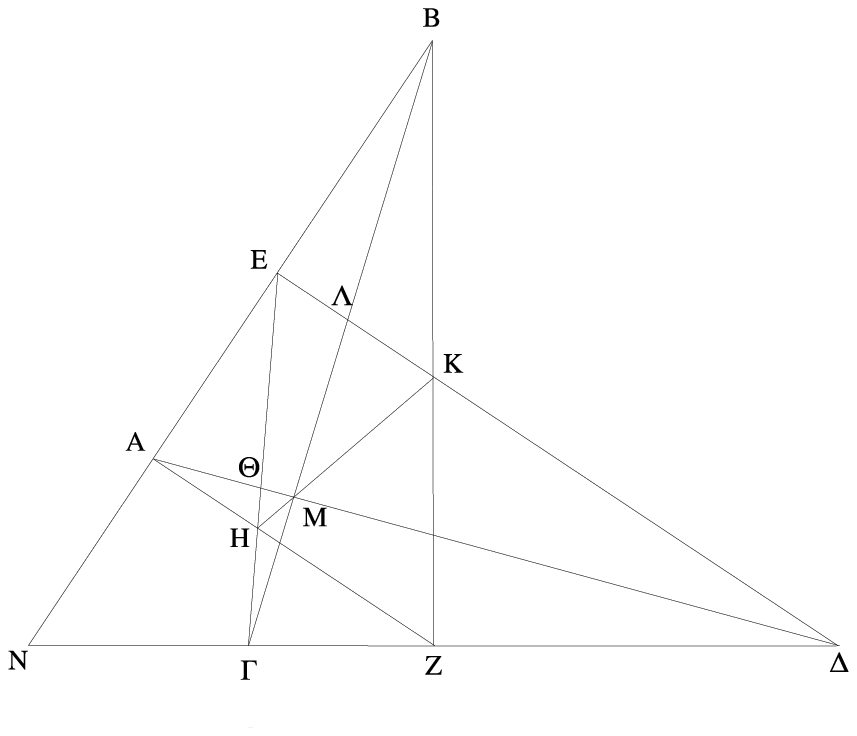}} 
      \end{tabular}
\end{center}
  \caption{Figures from \cite{pappus:synagogue}, \cite{pappus:jones1}  }
  \label{f:jones}
\end{figure}

In~\cite{pappus:synagogue}, Pappus of Alexandria proves
numerous lemmas concerning Euclid's Porisms~\cite{euclid:porisms}.
The combination of several, have become known as Pappus's Theorem.
A porism may have been a general statement linked to more specific examples,
in which case, Pappus's contribution of enumerating the cases and proving
each, would indeed merit the general attribution of the theorem to him.
Concerning the origin of the general statement,
even Pappus's attribution to Euclid may be insufficiently ancient:
at least some commentators view Euclid as a master compiler, rather
than a deep original thinker, which would suggest that
Euclid's lost Porisms, would in turn credit yet older work.

Pappus's statement of the lemmas, follows the convention
that the order of points on a line, and the definition of points
that are the intersections of lines is often left to the
reader's consulting of the drawing, see~\cite{pappus:jones1}. 
The two drawings for these lemmas
are taken from the earliest extant, tenth century, copy of~\cite{pappus:synagogue},
held in the Vatican library.
We've copied Jones' copies~\cite{pappus:jones2},
including his correction
to
fig.~\ref{f:jones139}  
of an error in the Vaticanus,
detailed on his page 624.
Jones notes in \cite{pappus:jones1} that Pappus's diagrams, following
the conventions of the time, have
a pronounced preference for symmetry and regularization. In particular,
line $HMK$ need not be horizontal in figure~\ref{f:jones138},
and none of the lines need be perpendicular in figure~\ref{f:jones139}.

Two of the relevant lemmas of Pappus are:

\begin{lem}
Figure~\ref{f:jones138}.
\cite{pappus:synagogue} (folio 161v in Vatican copy)
Now that these things have been proved, let it be required to
prove that, if $AB$ and $\Gamma\Delta$ are parallel,
and some straight lines $A\Delta$, $AZ$, $B\Gamma$, $BZ$
intersect them, and $E\Delta$ and $E\Gamma$ are joined,
it results that the (line) through $H$, $M$ and $K$ is straight.
\end{lem}

\begin{lem}
Figure~\ref{f:jones139}.
\cite{pappus:synagogue} (folios 161v, 162, in Vatican copy)
But now let $AB$ and $\Gamma\Delta$ not be parallel,
but let them intersect at $N$. That again the (line) 
through $H$, $M$ and $K$ is straight.
\end{lem}

Combined, with the other cases considered by Pappus,
 these form a single theorem, which we state in
the projective plane, with more modern sensibilities,
illustrated with the less regular figure~\ref{f:pappus},
which is labelled with Latin rather than Greek letters.

\begin{figure}[htp]
  \begin{center}
\begin{tabular}{{cc}}
    \subfigure[Pappus's Theorem]{\label{f:pappus}\includegraphics[width=.45\textwidth]{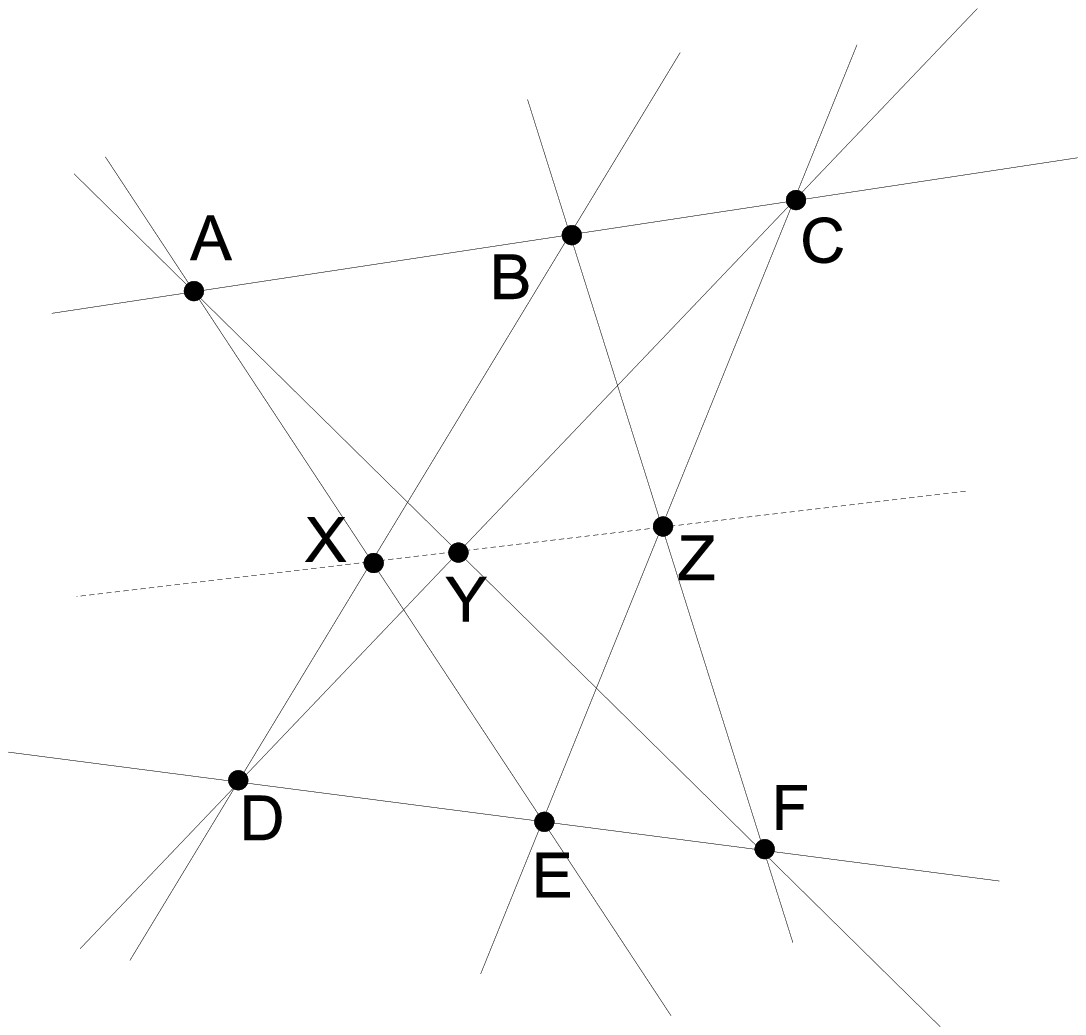}} &
   \subfigure[$\Rin$ labelled]{\label{f:grlabel}\includegraphics[width=.45\textwidth]{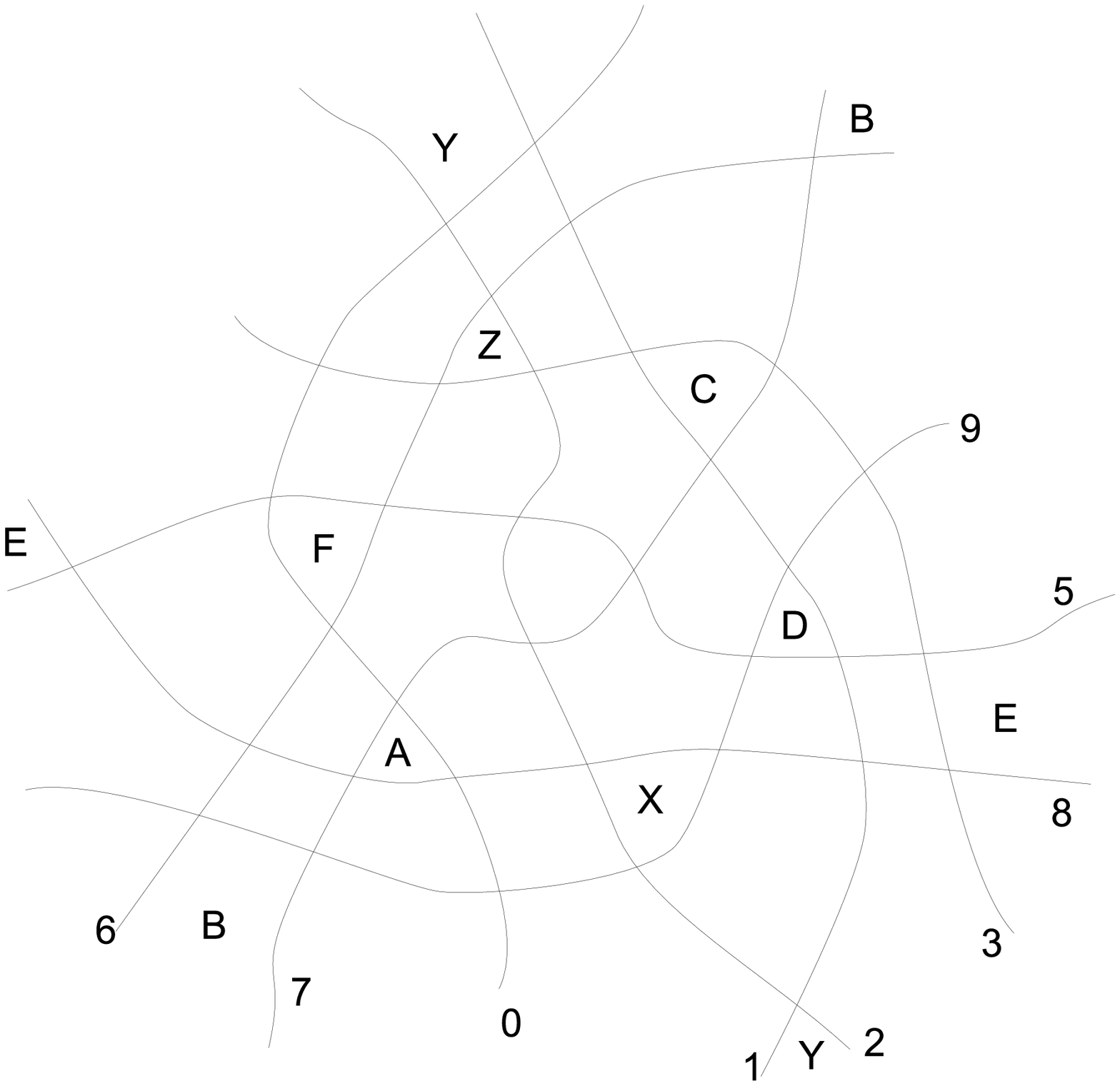}}\end{tabular}
\end{center}
  \caption{Pappus and non-Pappus (P)}
\end{figure}

\begin{thm}
\label{t:pappus}
In the projective plane,
if A, B, and C  are three points on one line, D, E, and F are three 
points on another line, and AE meets BD at X, 
AF meets CD at Y, 
and 
BF meets CE  at Z , then the three points X, Y, and 
Z are collinear. 
\end{thm}

\subsection{Ringel's non-Pappus Arrangement}

\begin{figure}[htp]
  \begin{center}
\begin{tabular}{{cc}}
    \subfigure[Projecting on 0]{\label{f:carroll}\includegraphics[width=.45\textwidth]{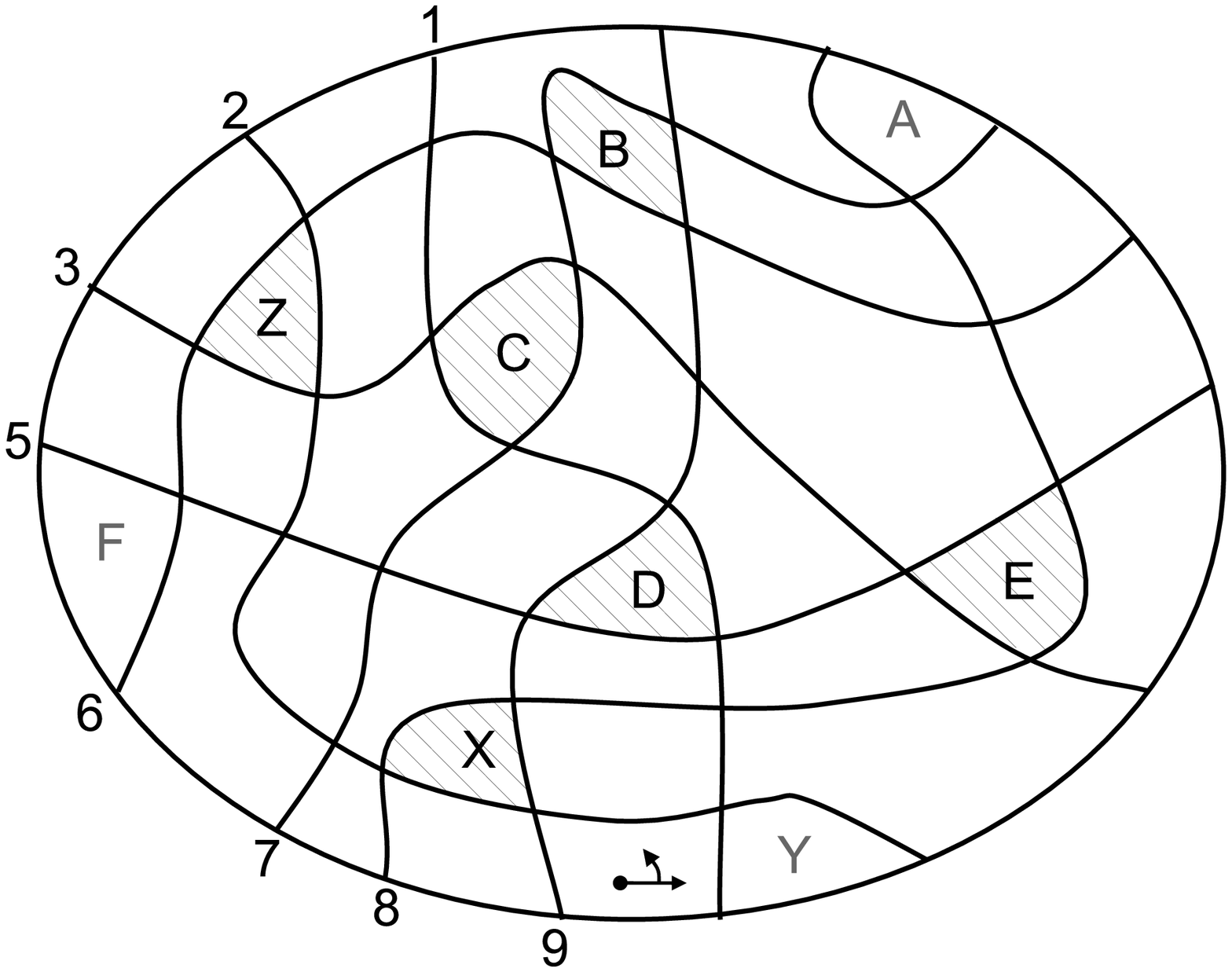}} &
    \subfigure[Adding $4 \parallel 3$ and $10 \parallel 9$ ]{\label{f:carroll2}
      \includegraphics[width=.45\textwidth]{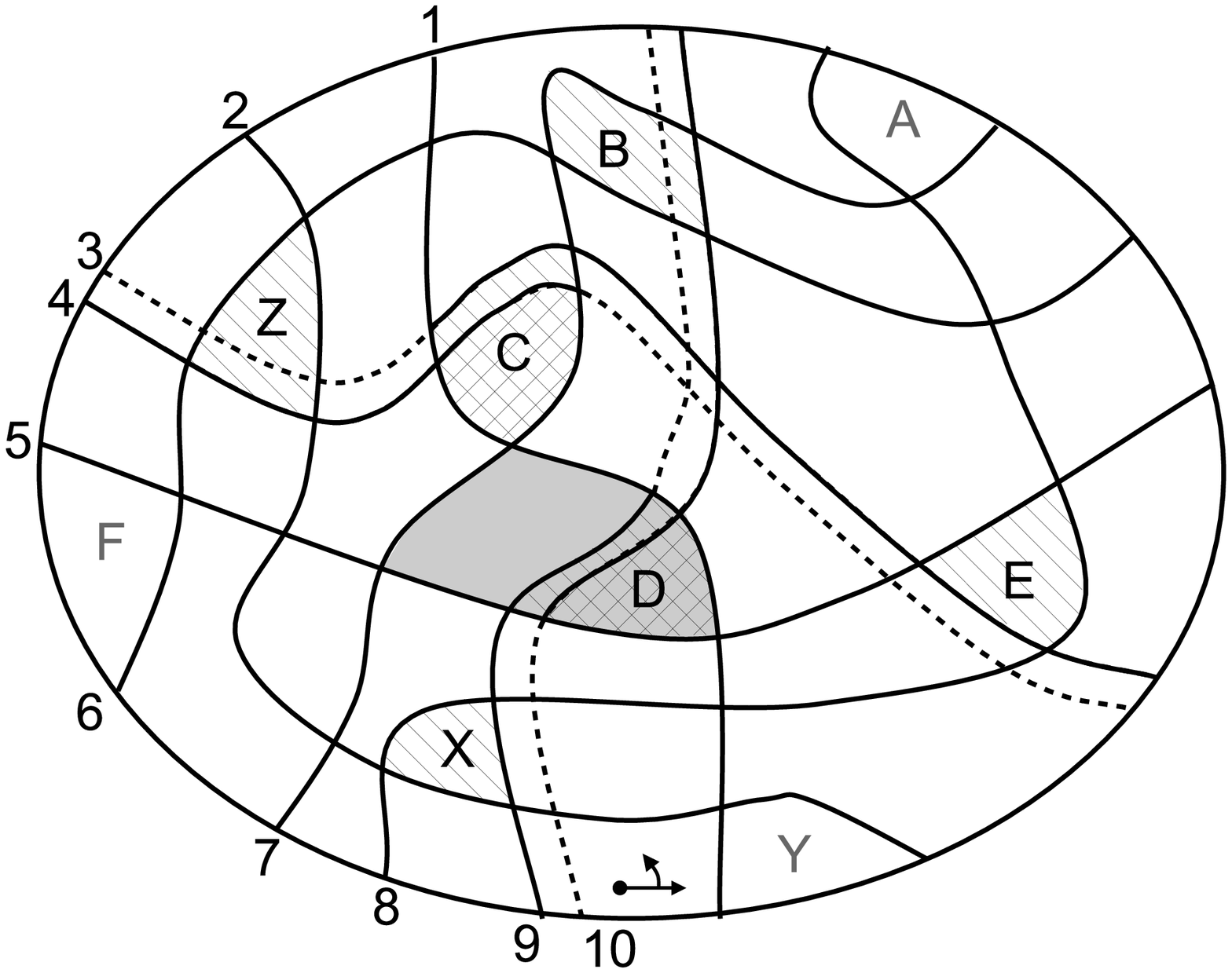}}
\end{tabular}
\end{center}
  \caption{Stretching $\Rin$ in the Euclidean plane}
 \label{f:stretching}
 \end{figure}

We start by proving that $\Rin$ is unstretchable.
Our proof below is new, and, I assert, interesting!
The result is well-known: \cite{ringel:teilungen} works from
Pappus, and \cite{bjorner:oriented} provide a final polynomial 
(see our equation~(\ref{e:final}) on page~\pageref{e:final}) which proves this independently of Pappus.

The remainder of the proof of Pappus, is an unsurprising reversal of
Ringel's argument.

The two proofs in this section
proceed by contradiction,
and are heavily illustrated. Thus, we need to
draw impossible illustrations. We do this by representing
hypothesised lines by actual pseudolines, hypothesised to
be straight.
The proofs argue to some extent `from the picture'. These
arguments fundamentally concern the relative ordering of various
points, on various lines, where the illustrations serve to capture the relative orderings,
and what is known about them. For such arguments, pseudoline arrangements suffice.
Indeed, in this paper, as in most of the literature, Ringel's non-Pappus
arrangement is not even defined, except by a picture. 
Formally, we could give a chirotope of
the oriented matroid, this would then relate to
equation~(\ref{e:coincident}) and its variations, to describe the significant
relationships between sets of three lines in each picture; however, we choose not to.

The labels in figures~\ref{f:pappus}, \ref{f:stretching},
\ref{f:pappus2}, \ref{f:pappus3} and~\ref{f:pappus4},
are all consistent with each other and figs~\ref{f:straight} 
and~\ref{f:graph}.
For example,  the triangle labelled $E$ in
fig.~\ref{f:straight} corresponds to the point labelled $E$
in figs~\ref{f:pappus} and~\ref{f:pappus3} and 
to the vertex labelled $E$ in fig.~\ref{f:graph}.

We use the following technical lemma:
\begin{prop}
\label{p:normalize}
For any real $\theta_a, \theta_b, \theta_c$ and $\theta_d$:
\begin{equation}
\label{e:normalize}
\LSN{a}{c}\LSN{b}{d}
= \LSN{a}{b}\LSN{c}{d} + \LSN{a}{d}\LSN{b}{c}
\end{equation}
\end{prop}
The proof is an exercise. We will only use this with
$a < b < c < d$. This allows us to expand any sum of products
of sines of angles in a line arrangement into a sum 
of products of non-overlapping angles (i.e. every pair of pairs is disjoint or nested).
Section~\ref{s:normal} is an in-depth study of  consequences of this lemma.

In the following proof, we project
one of the pseudolines of $\Rin$ to the line at infinity,
so that we have a  Euclidean arrangement of eight pseudolines.
This technique
was suggested by \obib{Lawrence~}\cite{lawrence:lopsided},
and drawn explicitly in figure 1 in~\cite{gioan:bases}.

\begin{lem}
\label{l:rin9}
$\Rin$ is unstretchable.
\end{lem}
\begin{proof}
Refer to fig.~\ref{f:grlabel}. 
Take the line 0, and project the figure onto the Euclidean plane,
with line 0 as the line at infinity, see fig.~\ref{f:carroll}.
If $\Rin$ is stretchable, then we can
take \ref{f:grlabel} as a line arrangement. The projection
is then also a line arrangement.
We choose a coordinate system with polar origin as indicated in the
face bounded by lines 0, 1, 2 and 9.
We draw a line 10 parallel to line 9, and a line 4 parallel to line 3;
in the simplest version, we draw these two parallel lines directly
on top of lines 9 and 3 respectively. In the illustration~\ref{f:carroll2},
we have drawn them slightly to one side, but leaving the relationship
with the other lines, and the origin, unchanged.
Fig.~\ref{f:carroll2}, being a line diagram (despite appearances),
then satisfies the conditions for theorem~\ref{t:main}.
Moreover:
\begin{itemize}
\item $\theta_3 = \theta_4$
\item $\theta_9 = \theta_{10}$
\item and $\theta_i < \theta_j$  and $\sin(\theta_j - \theta_i) > 0$
       for all $i < j$
\end{itemize}
Therefore, we have:
\begin{equation}
\label{e:pp-a}
\SN{8}{9}\SN{1}{9}\SN{2}{3}\SN{3}{5}\SN{6}{7} 
+\SN{3}{6}\SN{1}{9}\SN{7}{9}\SN{3}{5}\SN{2}{8} 
- \SN{3}{6}\BSN{3}{8}\SN{7}{9}\BSN{2}{9}\BSN{1}{5}
> 0
\end{equation}
from theorem~\ref{t:main}.

We apply proposition~\ref{p:normalize} twice to part of the last
term. We use bold terms to indicate the parts to be expanded.
\begin{align}
 &  \SN{3}{8}\BSN{1}{5}\BSN{2}{9} \\
 = & \BSN{3}{8}\SN{1}{9}\BSN{2}{5} + \SN{3}{8}\SN{1}{2}\SN{5}{9}\\
 = &\SN{1}{9}\SN{2}{8}\SN{3}{5} + \SN{1}{9}\SN{2}{3}\SN{5}{8} +\SN{3}{8}\SN{1}{2}\SN{5}{9}
\end{align}
Substituting this into~(\ref{e:pp-a}), and simplifying gives:
\begin{equation}
\label{e:pp-b}
\SN{8}{9}\SN{1}{9}\SN{2}{3}\SN{3}{5}\SN{6}{7}
- \BSN{3}{6}\BSN{7}{9}\SN{1}{9}\SN{2}{3}\BSN{5}{8} 
 -\SN{3}{6}\SN{7}{9}\SN{3}{8}\SN{1}{2}\SN{5}{9}
> 0
\end{equation}
Again, we  expand one part of the middle term:
\begin{align}
& \SN{3}{6}\BSN{7}{9}\BSN{5}{8} \\
= & \BSN{3}{6}\BSN{5}{7}\SN{8}{9} + \SN{3}{6}\SN{5}{9}\SN{7}{8} \\
= & \SN{3}{5}\SN{6}{7}\SN{8}{9} + \SN{3}{7}\SN{5}{6}\SN{8}{9} +
\SN{3}{6}\SN{5}{9}\SN{7}{8}
\end{align}
Substituting into equation (\ref{e:pp-b})
gives
\begin{equation}
\label{e:pp-c}
- \SN{1}{9}\SN{2}{3}\SN{3}{7}\SN{5}{6}\SN{8}{9}
- \SN{1}{9}\SN{2}{3}\SN{3}{6}\SN{5}{9}\SN{7}{8}
 -\SN{3}{6}\SN{7}{9}\SN{3}{8}\SN{1}{2}\SN{5}{9}
> 0
\end{equation}
which is a contradiction, proving the lemma.
\end{proof}
\begin{figure}[htp]
  \begin{center}
\begin{tabular}{{cc}}
    \subfigure[A counterexample (E)]{\label{f:pappus2}\includegraphics[width=.45\textwidth]{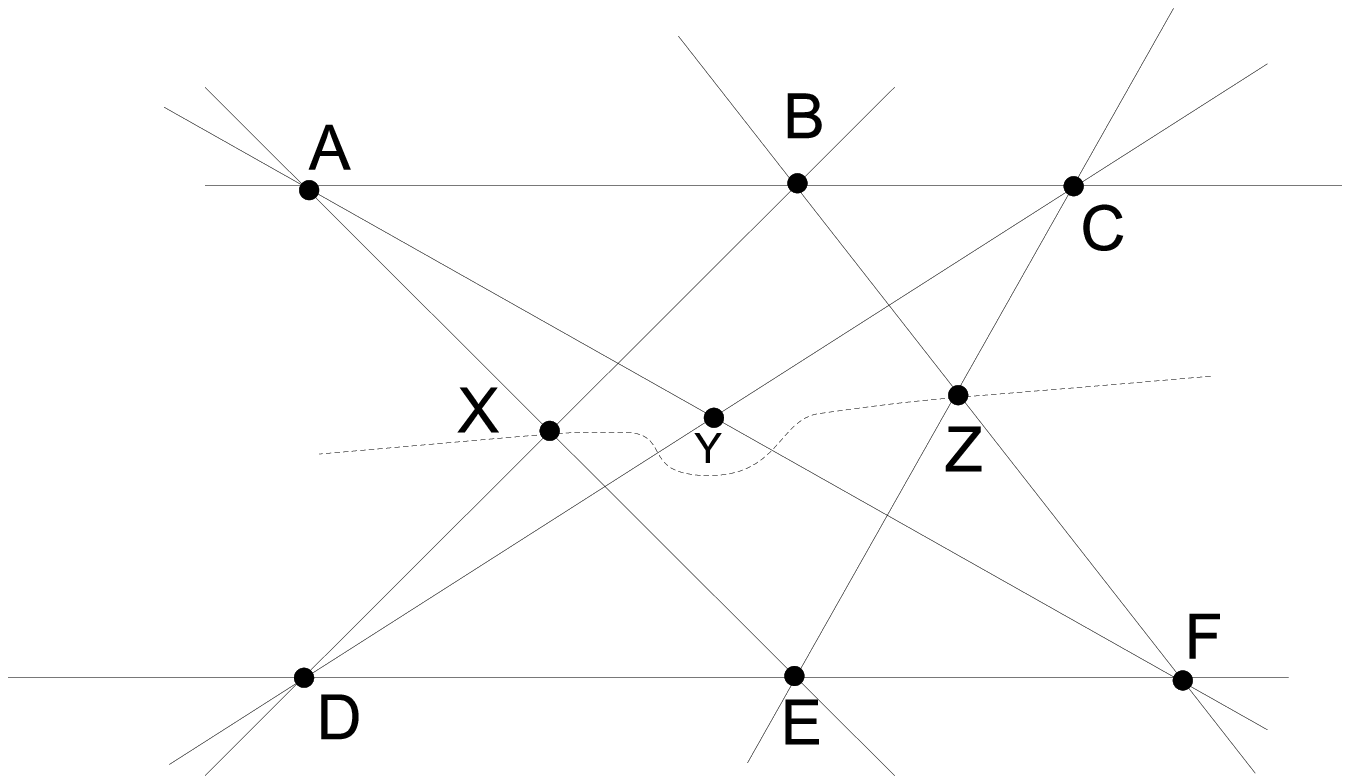}} &
      \subfigure[A 2nd counterexample (E)]{\label{f:pappus2X}\includegraphics[width=.45\textwidth]{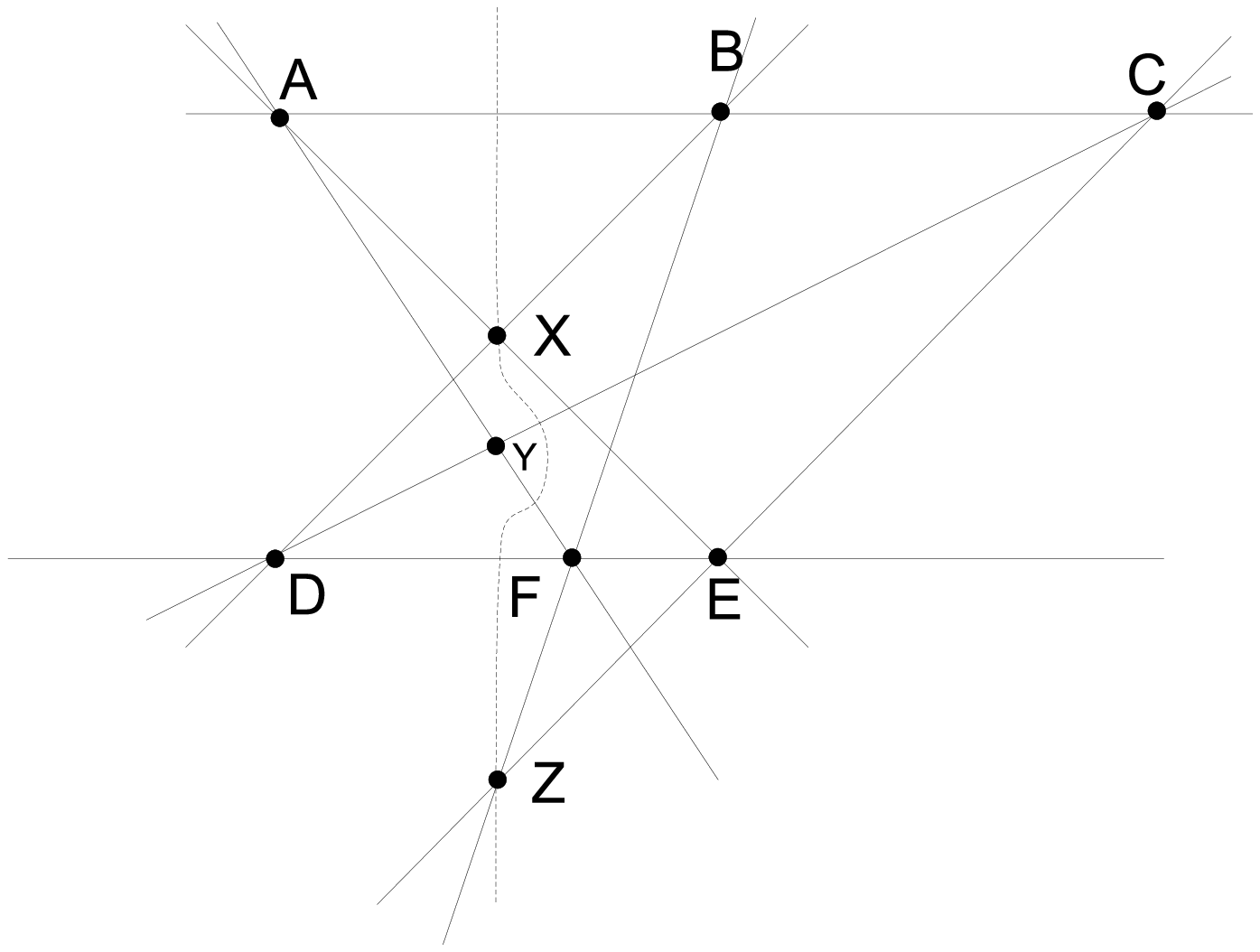}} \\
    \subfigure[Perturbing lines AF,CD,XY (E)]{\label{f:pappus3}
       \includegraphics[width=.45\textwidth]{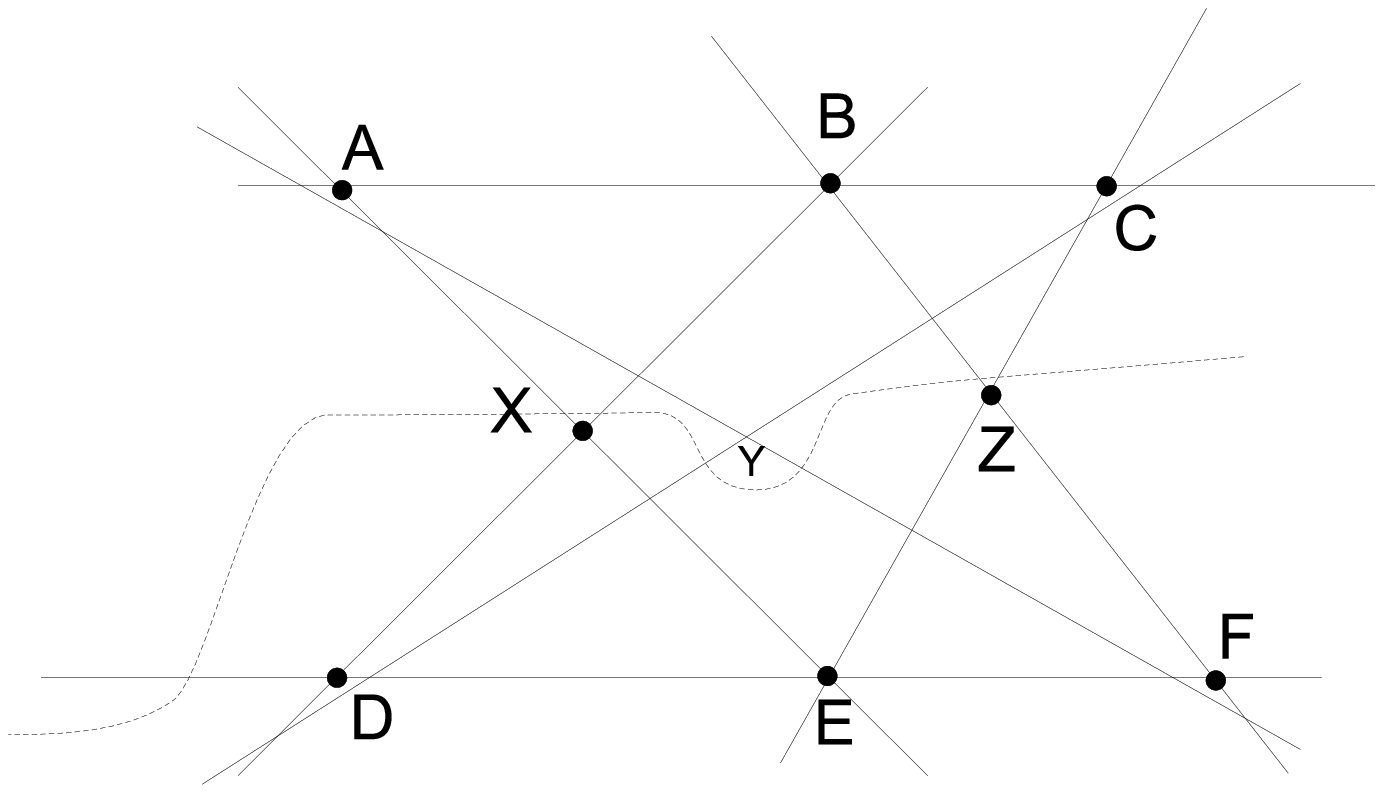}} &
    \subfigure[Perturbing lines AF,CD,XY  (E)]{\label{f:pappus3X}\includegraphics[width=.45\textwidth]{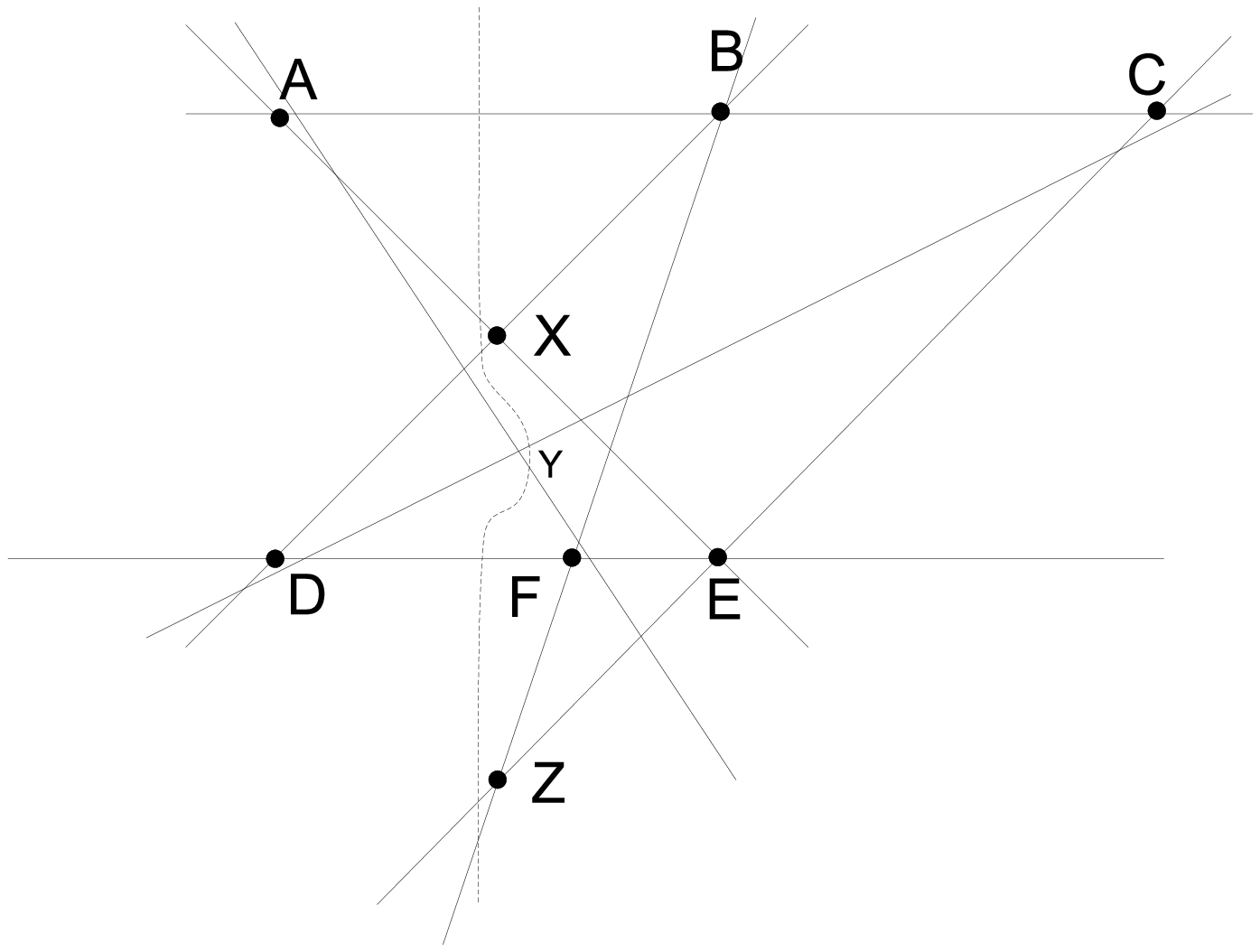}} \\
    \subfigure[Perturbing lines AC,DF (P,E)]{\label{f:pappus4}\includegraphics[width=.45\textwidth]{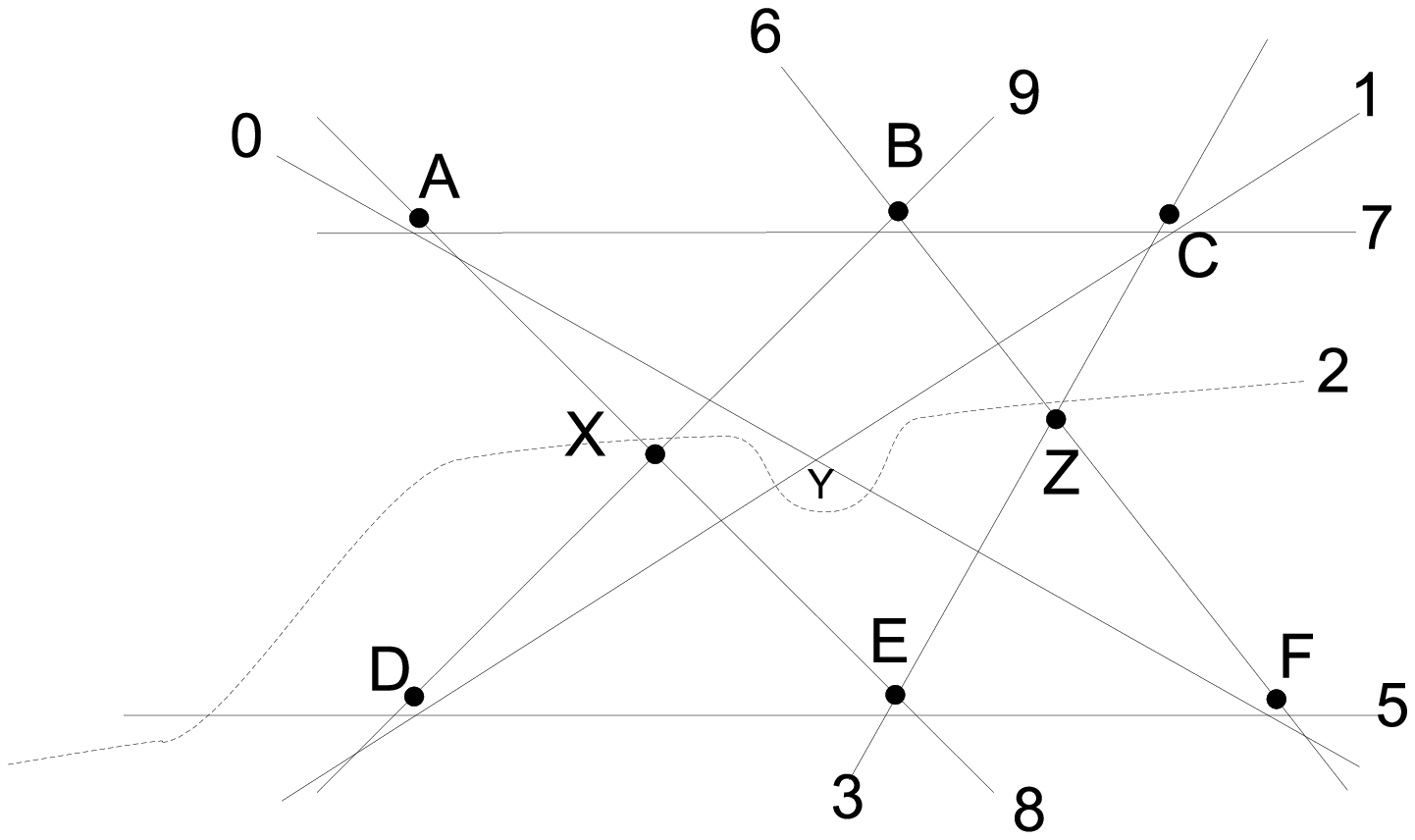}} &
        \subfigure[Perturbing lines AC,DF (P,E)]{\label{f:pappus4X}\includegraphics[width=.45\textwidth]{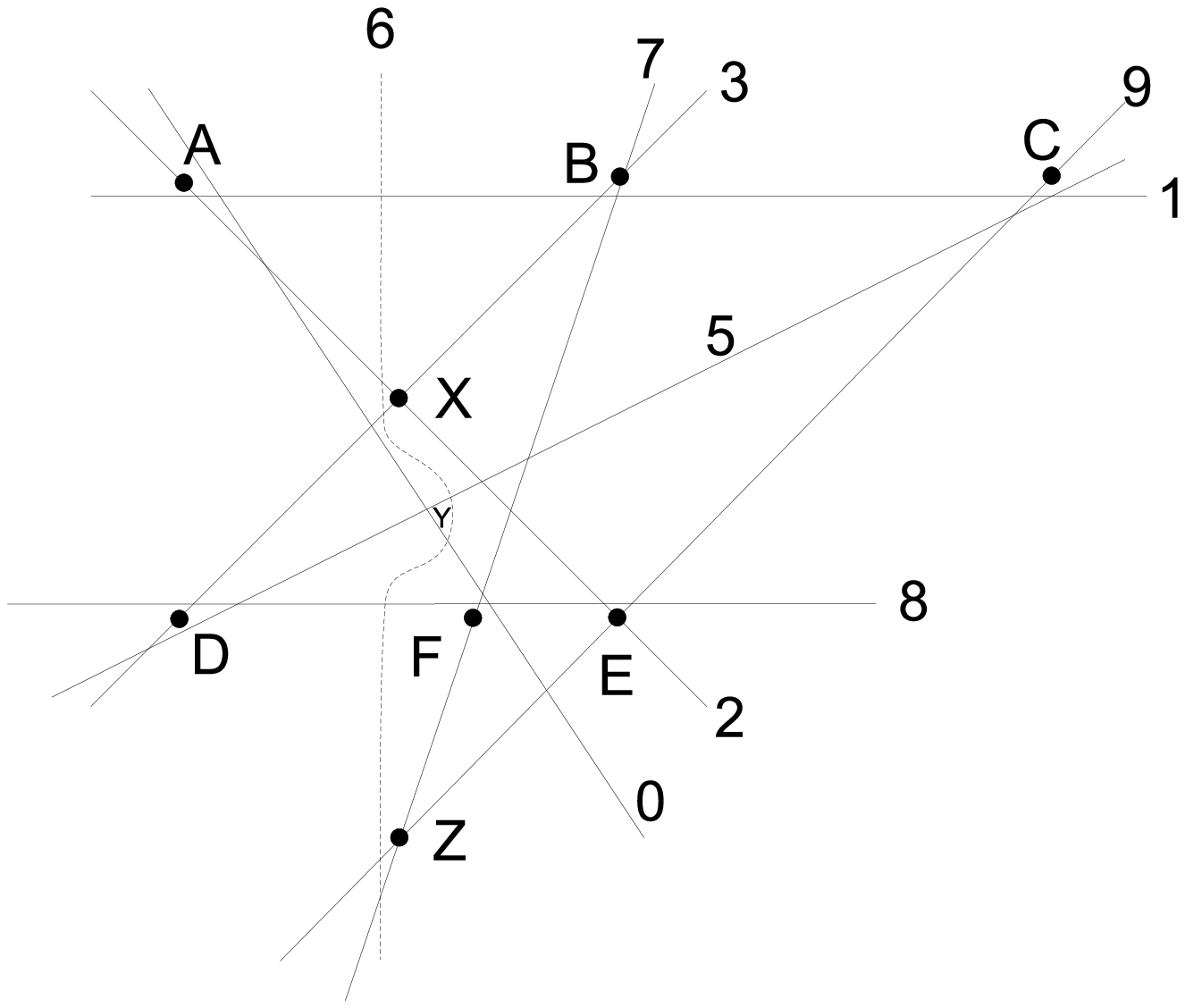}} 
\end{tabular}
\end{center}
  \caption{Illustrating the proof}
  \label{f:pappusproof}
\end{figure}

\subsection{Relationship with Pappus}

 \begin{proof}[Proof of Pappus's theorem~\ref{t:pappus}]
 Let N be the point of intersection of the two lines.
 
 If N, A, B, C, D, E and F are not all distinct,
 then the theorem is trivial.
 
 Take a counterexample to the theorem.
 
 By relabelling we can assume that A and C are both adjacent to N on one line,
 and that D is adjacent to N on the other. Since A, B, D, E are in general position, we
 can project them to the corners of a square,  fig.~\ref{f:pappus2}.
 Since C is adjacent to the point of intersection N, that lies at infinity,
 C is mapped to a point to the right of B as  illustrated. Since D is adjacent to N, 
 either F lies to
  the right of E as in fig.~\ref{f:pappus2}, or between D and E, as in fig.~\ref{f:pappus2X}.
 In the first case,
 X and Z are distinct because neither can be E and X lies on AE and Z lies on CE.
The line XZ is either parallel to DE, or not. If not, then it either intersects DE to the right of E
or to the left. If it intersects to the right, we can relabel the projective points 
by swapping A and C, and D and F,
and hence X and Z. 
So, in the first case, without loss of generality, 
we have XZ intersecting DE as illustrated,
or XZ is parallel to DE.
 
 By hypothesis Y does not lie on the line XZ.
 Without loss of generality we can assume that Y lies on the same side of XZ as A.
Since, if not,
 in the first of the two cases, 
 we can relabel the projective points swapping A and D, B and E, C and F, leaving X, Y and Z unchanged, but the newly labelled A lies in the same half plane of XZ as Y. 
 In the second,
 E is adjacent to N, and we can relabel the projective points, swapping A and D, B and F
 and C and E, so that the points X and Y are swapped too. 
 
Given either of these two diagrams, we can draw the corresponding 
diagram  fig.~\ref{f:pappus3} or~\ref{f:pappus3X} by: selecting a point P
within the triangle XYZ and drawing new lines parallel and close to XY, XZ and YZ, strictly
between the original line and P, while maintaining the incidence properties of
the original line, with the other lines in the diagram, except at the points that lie on it. 
In the first case, if XY is parallel to DE we can twist the new line through a small angle, maintaining all the incidence properties, 
except that it intersects DE to the left of E, thus arriving at fig.~\ref{f:pappus3}.

We then draw lines parallel but close to lines AC and DE.
To draw fig.~\ref{f:pappus4} from~\ref{f:pappus3}, we add new lines below the 
old lines, but close enough 
not to flip any triangles. 
To draw fig.~\ref{f:pappus4X} from~\ref{f:pappus3X}, we add new lines below AC and above DE.

We map the resulting line arrangement back into the projective plane.

So, to summarize, given an arrangement of nine lines in the projective plane,
a counterexample to Pappus's theorem,
we can project it onto either fig.~\ref{f:pappus2} or~\ref{f:pappus2X}, considered as recording the
relative incidence properties of the lines. By a sequence 
of operations, we can construct a different arrangement of nine lines, in the projective
plane, either fig.~\ref{f:pappus4} or~\ref{f:pappus4X}, again recording the arrangement of the lines.
However, the numbering of the lines in these last two figures, corresponds to the numbering
in fig.~\ref{f:grlabel}, with the same incidence properties, e.g. line 0 crosses lines 1, 2, 3, 5, 6, 7, 8, 9
in order; and line 8 crosses lines 1, 9, 2, 0, 7, 6, 5 and 3 in order; these statements are true of all
three diagrams. Hence they represent the same pseudoline arrangement. Thus we have
constructed a stretched version of $\Rin$, which is not possible: so  Pappus is proved.
 \end{proof}

\section{Linear Programming}
\label{s:motzkin}
In section~\ref{s:maintheorem}, we saw that the main theorem
amounts to a question of linear programming.

This section reviews, without proof, results from \cite{motzkin:phd},
\obib{Motzkin's} his seminal PhD thesis on systems of linear inequalities.
These results are all expressed in terms of vertical simplexes, and with a system of 
$n$ unknowns and $m$ inequalities being expressed by an $m$-by-$n$ matrix $M$, multiplying a column vector of size $n$.
Actually, we only care about such systems where the $n$ unknowns are the $r_i$ polar coordinates,
and each row of the matrix $M$, has three non-zero elements, the $i$-th being
$\pm\LSN{j}{k}$, the $j$-th being $\mp\LSN{i}{k}$ and the $k$-th being $\pm\LSN{i}{j}$, for some $i$, $j$ and $k$.

\begin{defn}
\label{d:simplex}
A matrix $M$ with $r+1$ rows and $r$ columns is a {\em simplex}
if, up to multiplication by a constant, there is precisely
one positive linear dependency between the rows.
\end{defn}

\begin{thm}
\label{t:simplex}
A matrix $M$ with $r+1$ rows and $r$ columns is a simplex if and only if
its $r+1$ $r$-by-$r$ subdeterminants alternate in sign.
\end{thm}

\begin{thm}
\label{t:mus}
Given an $m$-by-$n$ matrix $M$, the system $M x > 0$ is
a minimal insoluble system if and only if there is some 
 $m$-by-$m-1$ submatrix $S$ of $M$ such that:
\begin{itemize}
\item 
$S$ is a simplex.
\item
for all columns $c$ of $M$, the $m$-by-$m$ matrix
$S c$ has determinant zero.
\end{itemize}
\end{thm}

The spirit of these results is anticipated by \obib{Carver~}\cite{carver:systems}; Motzkin's statements
clarify that it is sufficient to examine determinants of submatrices, which Carver mentions in passing.

Carver's key theorem is:
\begin{thm}
A necessary and sufficient condition that a given system $S$ be inconsistent is that there
exist a set of $m+1$ constants $k_1, k_2, \cdots k_{m+1}$, such that
\begin{equation}
\sum_{i=1}^m k_i L_i(x) + k_{m+1} \equiv 0,
\end{equation}
at least one of the $k$'s being positive, and none of them being negative.
\end{thm}
In the terms of this paper, $k_{m+1}$ is always $0$,
since our inequalities always compare with $0$.
Thus, we can rearticulate this as:
\begin{thm}
\label{t:carver}
Given an $m$-by-$n$ matrix $M$, the system $M x > 0$ is
insoluble if and only if there is a non-negative, non-zero
linear
dependency between the rows of $M$.
\end{thm}

\section{A Normal Form for Products of Sines}
\label{s:normal}

In the previous section we saw that the solubility issues
for the linear program introduced in section~\ref{s:maintheorem}
are to be addressed by looking at determinants.
The matrix in question has non-zero entries of the from
$\LSN{i}{j}$. Hence, we are going to consider sums of products
of many such terms.
We have already seen one fairly laborious application
of the identity~(\ref{e:normalize}) to equation~(\ref{e:pp-a})
to derive~(\ref{e:pp-c}). The reader may fear that
this process will be repeated.

Fear not!

We will show that repeated application of~(\ref{e:normalize}),
with $a < b < c < d$ until it can no longer be applied, leads 
to a normal form. Despite the many choices faced
while making such a derivation, the process terminates,
and always at the same answer.
Then, with the remainder of the paper, whenever we need to
show a trigonometric identity, like that used in 
the proof of lemma~\ref{l:rin9}, we will simply say,
{\em by normalization}.
The suspicious reader, will need, like the author, to
write a simple computer program to perform the computation.
\begin{thm}
\label{t:normal}
Given a formal expression being a sum of products of sines
of differences between unknown angles, repeated expansion
using equation~(\ref{e:normalize}) with $a < b < c < d$,
always terminates at a uniquely determined normal form.
\end{thm}
Apart from this result, this fairly long section is unused elsewhere 
in the paper.
It is suggested that on first reading, you skip to page~\pageref{s:mainproof}.

\subsection{Multisets}
\label{s:multisets}
We wish to represent an expression such as $\LSN{2}{4}\LSN{1}{3}$
as a set of pairs $\{ \{2,4\}, \{1,3\} \}$. However, if we use a set
then the different expressions $\LSN{1}{3}\LSN{1}{3}$ and $\LSN{1}{3}$
would be represented as the same set $\{ \{ 1, 3 \} \}$.
Therefore we will work with multisets. Unfortunately, I have failed to
find an appropriate paper introducing multisets, so, in this subsection, I will
give a quick introduction based on the Web page at Wikipedia\footnote{
http://en.wikipedia.org/wiki/Multiset, as of 24th March 2007. 
My extensions include using $\aleph_0$ as a multiplicity and the  $  \uparrow, \downarrow$
and $\pow$ operators.
}

We will use the set
\begin{equation}
\Nat^+ = \Nat \cup \{ \infty \}
\end{equation}
for counting. Formally, by $ \infty$ we mean
$\aleph_0$, and the arithmetic we are using is cardinal arithmetic so that, for example:
\begin{align}
2 - 3 &= 0 \\
\infty + 1 &= 1 + \infty = \infty\\
5 &< \infty
\end{align}
We will also have the convention that
\begin{equation}
\infty - \infty = 0
\end{equation}

Given a fixed set $X$, then a multiset $A$ is formally defined by an indicator function 
$\ind_A : X \rightarrow \Nat^+$, which gives the multiplicities of the elements
of $X$. We define the usual set operators, 
$\in, \cap, \cup, \setminus, \subset, |\cdot|$ over multisets.
We also define three multiset specific operators $ \uplus, \uparrow, \downarrow$.
\begin{align}
x \in A & \: \: \textrm{when } \ind_A(x) \geq 1 \\
\ind_{A \cap B}(x) &= \min \{ \ind_A(x), \ind_B(x) \} \\ 
\ind_{A \cup B}(x) &= \max \{ \ind_A(x), \ind_B(x) \} \\ 
\ind_{A \setminus B} &=  \ind_A -  \ind_B \\ 
\ind_{A \uplus B} &=  \ind_A +  \ind_B \\
A \subset B &   \: \: \textrm{when for all } x \in X, \; \ind_A(x) \leq \ind_B(x) \\
|A| &= \Sigma_{x \in X} \ind_A(x) \\
\ind_{\downarrow A}(x) &= \begin{cases}
0 & \textrm{when }\ind_A(x) = 0 \\
1 & \textrm{otherwise}
\end{cases} \\
\ind_{\uparrow A}(x) &= \begin{cases}
0 & \textrm{when }\ind_A(x) = 0 \\
\infty & \textrm{otherwise}
\end{cases} 
\end{align}

The sum and product operators, $\sum$, $\prod$, are defined
in terms of the standard ones, using multiplicity, i.e.
\begin{align}
\sum_{a \in A} f(a) &= \sum_{x \in X} \ind_A(x) f(x) \\
\prod_{ a \in A} f(a) &= \prod_{x \in X} f(x) ^ {\ind_A(x)}
\end{align}

We see in these two expressions that the expression $a \in A$
is in some circumstances understood as 
itself having a multiplicity. This is particular significant in definitions
of multisets in terms of other multisets, e.g.
\begin{align}
A &= \{ 2, 3, 3,4,4,4 \} \\
B &= \{ a+3 : a \in A \} \\
B &= \{ 5, 6, 6, 7,7,7 \}
\end{align}
so that the multiplicities in $A$ carry across to the multiplicities in $B$.
A second example:
\begin{align}
A &= \{ 2, 3, 3, 4,4,4 \} \\
C &= \{ a \mod 2 :  a \in A \} \\
C &= \{ 0,0,0,0, 1,1 \}
\end{align}
Formally, given a multiset $A$, a predicate $P(x)$ assigning truth values
to each $x \in A$ and a partial function $f: X \rightharpoonup X$ defined on all
$x \in A$ with $P(x)$,
we can construct the multiset $\{ f(x) : x \in A, P(x) \}$,
with indicator function defined:
\begin{equation}
\ind_{\{ f(x) : x \in A, P(x) \}}(y) =
\sum_{\{x\in X: \ind_A(x) > 0, P(x), y = f(x)\}} \ind_A(x)
\end{equation}
where the sum is over a set (not a multiset).
If more than one occurrence of the
multiset membership operator 
occurs
in such a set definition then the second
is introduced with words (like `such that') indicating 
that it is to be read as a true/false predicate, ignoring multiplicities.

In contrast the
multiset subset operator is simply a predicate, with a true or false value.

The most complex multiset operator we use is $\pow$ for 
finite powermultisets. 
The finite powermultiset $\pow A$ of a multiset
$A$ contains precisely each of the finite multisets 
that are subsets of $A$, and each 
has an infinite multiplicity in $\pow A$:
\begin{equation}
\pow A = \uparrow \{ B \in X : B \subset A, |B| < \infty \}
\end{equation}
Notice that this differs from the normal set definition of powerset.
We will make $X$  large enough  so that  $B \in X$ is not restrictive .

In that definition we see that any subset of $X$ can be considered 
as a multiset, 
whose indicator function
takes values in $\{ 0, 1 \}$. Likewise, any multiset whose indicator function
only takes values $0$ or $1$, can be considered as a set. In particular,
for any multiset $A$, $\downarrow A$ can be considered a set.

A final multiset operator is $\biguplus$, 
which additively combines all members of a multiset of multisets.
This can be defined by:
\begin{equation}
\ind_{\biguplus A} = \sum_{a \in A} \ind_a
\end{equation}

The base set $X$ can usually be chosen large enough to contain everything of interest
for a particular discussion, and hence can be ignored.
Formally, for this section we will take $X$ as the smallest set
containing both $\Nat$ and $\pow \uparrow X$.
i.e. for any $x \notin X$,
$\pow \uparrow X$ when  defined over a ground of $X \cup \{ x \}$
has the same elements as when defined
over $X$.

Thus $X$ contains many multisets. 
Since $\pow$ introduces only {\em finite}
multisets, this is adequately limited to avoid paradox,
and could, with just a little bit more effort, be fully
formalized within ZF.
We make no further reference to $X$.

\subsection{On Pairs of Integers}
The set $\Nat^{(2)}$ is the set of all pairs of natural numbers.
If $\{a,b\} \in \Nat^{(2)}$ then $a \neq b$.

We use a function $\tau$ to map members of $\Nat^{(2)}$ to formal expressions
over a vector $\boldsymbol{\theta}$, corresponding to the sine function. i.e.
\begin{equation}
\tau( \{ a, b \} ) = \begin{cases}
\LSN{a}{b} & a < b \\
\LSN{b}{a} & b < a
\end{cases}
\end{equation}

Technically, the range of $\tau$ is a free algebra.
Given values for $\boldsymbol{\theta} \in \Real ^\Nat$, 
we can evaluate $\tau(\alpha)$, by
substituting in the values for $\boldsymbol{\theta}$. 
We write $\tau(\alpha)(\boldsymbol{\theta})$
for this value.

The expressions of interest are those such as in equation~(\ref{e:main2}). We will
separate out the positive and negative terms, so that we have two expressions, 
each being the sum of products of sines of differences of pairs of angles.

To express products of sines, we will use finite submultisets of 
$\uparrow \Nat^{(2)}$, i.e. any member of $P$
\begin{equation}
P = \pow \uparrow \Nat^{(2)}
\end{equation}
We extend the definition of $\tau$ for $p \in P$, with
\begin{equation}
\tau( p ) = \prod_{ x \in p } \tau(x)
\end{equation}
For example:
\begin{equation}
\tau( \{ \{ 2, 4 \}, \{2, 4 \}, \{ 1, 3 \} \} ) = \LSN{2}{4}\LSN{2}{4}\LSN{1}{3}
\end{equation}

To express sums of products of sines, we use finite submultisets of $P$,
i.e: any member of $S$
\begin{equation}
S = \pow P
\end{equation}
We similarly extend $\tau$ to $S$, to give the following
definition of $\tau$ on $\Nat^2 \cup P \cup S$:
\begin{equation}
\tau(x) = \begin{cases}
\LSN{a}{b} & x = \{ a, b \} \in \Nat^{(2)}, a < b \\
\prod_{ y \in x } \tau(y) & \textrm{when } x \in P \\
\sum_{ y \in x } \tau(y) & \textrm{when } x \in S 
\end{cases}
\end{equation}

We then define an equivalence relationship over $S$ by:
\begin{equation}
\alpha \equiv \beta \textrm{ if and only if } \tau(\alpha)(\boldsymbol{\theta}) 
= \tau(\beta)(\boldsymbol{\theta}) 
\textrm{ for all } \boldsymbol{\theta} \in \Real ^\Nat
\end{equation}

Trigonometric identities, such as 
in equation~(\ref{e:d1_d5_d7}), can then be verified by gathering together
the positive and negative terms, to give two members of $S$
and using the combinatoric methods of this section, to show that
they are equivalent.

We are interested in the applicability of formula~(\ref{e:normalize})
to members of $P$ and $S$. We say
\begin{defn}
A pair of pairs $\{ x, y \} \in P$ is
{\em expandable}, if there are $a,b,c,d \in \Nat$, with
\begin{align}
a &< b < c < d \\
x &= \{ a, c \} \\
y &= \{ b, d \}
\end{align}
\end{defn}

We also define the multiset $B \subset S$ (resp. $B_0 \subset P$)
of atomic elements $\alpha$
of $S$ (resp. $P$) such that formula~(\ref{e:normalize}) is not applicable
to $\tau(\alpha)$:
\begin{align}
B_0 &=\{ p \in P : \textrm{There is no expandable } q \subset p
\}  \\
B  &= \pow B_0 
\end{align}

In contrast, we can expand any $\alpha \in S \setminus B$
corresponding to an application of~(\ref{e:normalize})
to $\tau(\alpha)$.

\begin{prop}
\label{p:expansion}
For any $\alpha \in S \setminus B$, we can {\em expand} 
$\alpha$ to some $\beta \in S$, by taking
$p \in \alpha$, with
 $\{ a, b \}, \{ c, d \} \in p$,
such that $a < b < c <  d$,
and $q, r \in P$,
such that:
\begin{align}
q &= \{ \{ a , c \}, \{ b, d \} \} \cup p \setminus \{ \{ a, b \}, \{ c, d \} \} \\
r &= \{ \{ a , d \}, \{ b, c \} \} \cup p \setminus \{ \{ a, b \}, \{ c, d \} \} \\
\beta  &= \{ q, r \} \uplus \alpha \setminus \{ p \} 
\end{align}
\end{prop}
In such a case, we write
$\alpha \mapsto_1 \beta$.

We use $\mapsto$ as the transitive closure of $\mapsto_1$.

Inductively, from proposition~\ref{p:normalize}, we have:
\begin{prop}
\label{p:equivalence}
If $\alpha \mapsto \beta$ then $\alpha \equiv \beta$.
\end{prop}

\begin{lem}
\label{l:expansion}
For each $\alpha \in S$,
there exists at least one $\beta \in B$
with $\alpha \mapsto \beta$.
\end{lem}
\begin{proof}
Given such an $\alpha$, if $\alpha \in B$, then we are done.
Otherwise, there is some $p \in \alpha$ which can be expanded
to $q$ and $r$ as in proposition~\ref{p:expansion}.

We can do this repeatedly to arrive at some $\beta \in B$.
We need to prove termination of such a derivation.

We do so with the size function:
\begin{align}
s: & \downarrow S \rightarrow \Natural \times \Natural \\
\label{e:defn-p-size}
s_1(\alpha) & = \max_{p \in \alpha \setminus B } \prod_{ \{ a, b \} \in p }| a - b | \\
s_2(\alpha) &= | \{ p \in{ \alpha \setminus B }  : \prod_{ \{ a, b \} \in p } | a - b | = s_1(\alpha) \} | \\
s(\alpha) &= ( s_1(\alpha), s_2(\alpha) )
\end{align}

A $p$ in (\ref{e:defn-p-size}) on which the maximum is realized,
can be expanded using proposition~\ref{p:expansion}. 
This gives an $\alpha'$, with $\alpha \mapsto_1 \alpha'$
and either $s_1(\alpha') < s_1(\alpha)$
or  $s_1(\alpha') = s_1(\alpha)$ and $s_2(\alpha') < s_2(\alpha)$.

Induction then proves the result.
\end{proof}

We prove uniqueness in several steps.
We use two values computed from any $\alpha \in S$.
$E_\alpha$ is the multiset formed from the 
numbers that appear
in any pair in any product in $\alpha$,
each with the multiplicity it has in the product in which
it appears most often,
and $n_\alpha$
is the greatest of these.
i.e.
\begin{align}
E_\alpha &= \bigcup \left\{ \biguplus p : p \in \alpha \right\} \\
n_\alpha &= \max E_\alpha
\end{align}

Now, each expansion step of proposition~\ref{p:expansion}
leaves the multipliticies in $\biguplus q$ and 
$\biguplus r$ the same as the multiplicities in $\biguplus p$,
so that if $\alpha \mapsto_1 \beta$ we have $E_\alpha = E_\beta$.
Inductively, we have:
\begin{prop}
If $\alpha \mapsto \beta$ then $E_\alpha = E_\beta$.
\end{prop}

The same observation leads to the following definition, 
and proposition.
\begin{defn}
An $\alpha \in S$ is {\em regular}, if, for every $p \in \alpha$
$E_p = E_\alpha$.
\end{defn}

\begin{prop}
If $\alpha$ is regular, and $\alpha \mapsto \beta$ then $\beta$
is regular.
\end{prop}

The following definitions and lemmas
provide an inductive step for proving uniqueness.

\begin{defn}
Given $p \in P$, with $\{ n_p-1, n_p \} \notin p$
then the contraction $p^*$ of $p$ is:
\begin{equation}
p^* = \left\{ \{ a, b \} : \{ a, b \} \in p, a, b < n_p 
\right\} \uplus 
\left\{ 
\{ a, n_p-1 \} : 
\{ a, n_p \} \in p \right\}
\end{equation}
\end{defn}

\begin{defn}
Given $\alpha \in S$, the contraction $\alpha ^ *$ is given by:
\begin{equation}
\alpha^* = \left\{ p^* : 
p \in \alpha, n_p = n_\alpha, \{ n_\alpha -1, n_\alpha\} \notin p \right\}
\uplus
\left\{
p : p \in \alpha, n_p < n_\alpha
\right\} 
\end{equation}
\end{defn}

\begin{lem}
\label{l:inductiveStep}
If $p, q \in B_0$,
with $E_p = E_q$,
$\{ n_p-1, n_p \} \notin p \cup q$ 
and
$p^* = q^*$,
then $p = q$.
\end{lem}
\begin{proof}
Consider the multiset:
\begin{equation}
M = \{ a : \{ a, n_{p^*}\} \in p^* \}
\end{equation}
$|M|$ is the same as the number of occurrences
of $n_{p^*}$ in $p^*$, which by construction is
the same as the sum of the number of occurrences  in $p$
of $n_p$ and of $n_p - 1$, i.e.
\begin{equation}
|M| = \ind_{E_p}(n_p) + \ind_{E_p}(n_p - 1) 
\end{equation}
Since $p$ is not expandable, the pairs in $p$ giving rise
to $M$ must be nested. So that the first $\ind_{E_p}(n_p)$
members of $M$ are paired with $n_p$ in $p$, and the remaining
$\ind_{E_p}(n_p-1)$ members of $M$ are paired with $n_p-1$.
Thus, we can find $M_1, M_2$ with
\begin{align}
M &= M_1 \uplus M_2 \\
|M_1| &= \ind_{E_p}(n_p) \\
|M_2| &= \ind_{E_p}(n_p-1) \\
a \in M_1, & b \in M_2 \Rightarrow a \leq b
\end{align}
\begin{multline}
p =
\left\{ \{ a, b \} \in p^* : a, b < n_p - 1 \right\} \\
\uplus
\left\{ \{ a, n_p \} : a \in M_1 \right\}
\uplus
\left\{ \{ b, n_p-1 \} : b \in M_2 \right\}
\end{multline}
Since $\ind_{E_p} = \ind_{E_q}$ we find an identical
formula for $q$, so that $p = q$.
\end{proof}

\begin{lem}
If $\beta, \gamma \in B$ are regular 
with $E_\beta = E_\gamma$
and
$\beta \equiv \gamma$
then $\beta = \gamma$.
\end{lem}
\begin{proof}
Suppose not.
Then we can find such a regular counterexample with least $n_\beta$,
and secondarily, with $\sum_{p \in \beta}|p|$ as small as possible, 
and $\beta \neq \gamma$.

We can divide $\beta$ 
and $\gamma$ into those products that involve $\{ n_\beta-1, n_\beta \}$
and those that don't:
\begin{align}
\beta_1 &= \{ p : p \in \beta  \textrm{ such that }\{ n_\beta-1, n_\beta \} \in p \}
\\
\beta_2 &= \{ p : p \in \beta  \textrm{ such that }\{ n_\beta-1, n_\beta \} \notin p \}
\\
\beta &= \beta_1 \uplus \beta_2 \\
\gamma_1 &= \{ p : p \in \gamma  \textrm{ such that }\{ n_\beta-1, n_\beta \} \in p \}
\\
\gamma_2 &= \{ p : p \in \gamma  \textrm{ such that }\{ n_\beta-1, n_\beta \} \notin p \}
\\
\gamma &= \gamma_1 \uplus \gamma_2 
\end{align}

If  $\beta_1$ and $\beta_2$ are both empty, then $E_\beta$ is empty, and so is $E_\gamma$
and $\beta = \emptyset = \gamma$, and this was not a counterexample.

Otherwise
consider any $\boldsymbol{\theta} \in \Real ^ \Nat$, with
$\theta_{n_\beta} = \theta_{n_\beta-1}$, then:
\begin{align}
\tau(\beta_1)(\boldsymbol{\theta}) &= 0 = \tau(\gamma_1)(\boldsymbol{\theta}) \\
\tau(\beta_2)(\boldsymbol{\theta}) &= \tau(\beta_2^*)(\boldsymbol{\theta}) \\
\tau(\gamma_2)(\boldsymbol{\theta}) &= \tau(\gamma_2^*)(\boldsymbol{\theta})
\end{align}
i.e. every term in $\tau(\beta_1)$ and $\tau(\gamma_1)$ contains
a factor $\sin(\theta_{n_\beta}-\theta_{n_\beta-1})$ and so they vanish,
whereas, the evaluation of $\beta_2$ and $\beta_2^*$ is the same,
since they differ only by replacing all $\theta_{n_\beta}$ with
$\theta_{n_\beta-1}$ which have the same value.

Thus:
\begin{equation}
\tau(\beta_2^*)(\boldsymbol{\theta}) = \tau(\gamma_2^*)(\boldsymbol{\theta})
\end{equation}
and so $\beta_2^* \equiv \gamma_2^*$. 
Since $n_{\beta_2^*} < n_\beta$
by the minimality of $\beta$ we have that $\beta_2^* = \gamma_2^*$,
and hence that $\beta_2 = \gamma_2$, by the previous lemma.
As a consequence, $\beta_2 \equiv \gamma_2$,
and so $\beta_1 \equiv \gamma_1$, and in addition 
 $\beta_1 \neq \gamma_1$.

Thus, by minimality of $\sum_{p \in \beta}|p|$, we have
$\sum_{p \in \beta_1}|p| = \sum_{p \in \beta}|p|$, and so $\beta_2$ is empty.
But, consider:
\begin{align}
\beta' &= \{ p \setminus \{ \{ n_\beta-1, n_\beta \}\}: p \in \beta_1  \} \\
\gamma' &= \{ p \setminus \{ \{ n_\beta-1, n_\beta \}\}: p \in \gamma_1  \} 
\end{align}
We have $\tau(\beta)=\sin(\theta_{n_\beta}-\theta_{n_\beta-1})\tau(\beta')$
and $\tau(\gamma)=\sin(\theta_{n_\beta}-\theta_{n_\beta-1})\tau(\gamma')$,
so that (noting the continuity of $\tau(\beta')$ and $\tau(\gamma')$) for the
case $\theta_{n_\beta}=\theta_{n_\beta-1}$, we have $\beta' \equiv \gamma'$.
However, $\sum_{p \in \beta'}|p| < \sum_{p \in \beta}|p|$, and
so $\beta' = \gamma'$, and hence $\beta = \gamma$.
\end{proof}

\begin{lem}
\label{l:unique}
If $\alpha \in S$ and $\beta, \gamma \in B$
with $\alpha \mapsto \beta$
and $\alpha \mapsto \gamma$ 
then $\beta = \gamma$.
\end{lem}
\begin{proof}
If $|\alpha|=1$ then $\beta$ and $\gamma$ are regular,
and satisfy the conditions for the previous lemma, so that
$\beta = \gamma$.

Otherwise, for each member of $\alpha$, we have a unique expansion,
as just proved. The process of expanding each member is separate and
independent, since it works on one product $p \in \alpha$ at a time, without reference
to other members of $\alpha$. Thus we find a unique expansion in $B$ of $\alpha$
as the join of the unique expansions in $B$ of the members of $\alpha$.
\end{proof}

\begin{proof}[Proof of theorem~\ref{t:normal}]
This follows from lemmas~\ref{l:expansion} and~\ref{l:unique}.
\end{proof}

\section{First Proof of Main Theorem}
\label{s:mainproof}
We now give the first, very direct, proof of the main theorem.
In the following sections, we will give a more illuminating
and general proof.

This section can be skipped in its entirety; it is tedious and mechanical.
It's value is two fold: first, it illustrates how the general techniques
of the next sections apply in practice; second, it shows that,
once we have found the appropriate matrix, and simplex, that
the rest of the process can be automated.

The tedious computation of this section, should be compared and contrasted
with the equally tedious computation needed to verify, from first principles, the final polynomial 
for $\Rin$:
\begin{equation}
\label{e:final}
\begin{split}
[246][184][175][437][197] &+ [129][184][175][437][467] + \\
[138][194][247][175][467] &+ [156][184][247][437][197] + \\
[345][184][247][176][197] &+ [489][247][175][176][143] + \\
[597][247][184][176][143] &+ [678][247][175][194][143] + \\
[237][194][184][175][467] &= 0
\end{split}
\end{equation}
where, for indeterminates, $x_i, y_i, z_i$,
\begin{equation}
[i j k] = \begin{vmatrix}
x_i & x_j & x_ k \\
y_i & y_j & y_ k \\
z_i & z_j & z_ k
\end{vmatrix}
\end{equation}
This is taken from \cite{bjorner:oriented} p 349, note that the numbering of the lines
is different from ours.

At heart, we may conjecture that these two computations are cryptomorphic, and 
hence, the tedium of this section is unsurprising.

We have already argued that the conditions of
theorem~\ref{t:main} amount to requiring that 
the system~(\ref{e:system},\ref{e:linear-again}) is soluble.

Throughout this section we will use $M_9$ for the 
matrix from~(\ref{e:linear}), i.e.
we are considering the system
\begin{equation}
\label{e:system}
M_9 \boldsymbol{r} > \boldsymbol{0}
\end{equation}
where
\begin{equation}
\label{e:linear-again}
M_9 =
\left(
\begin{smallmatrix}
&(1)&(2)&(3)&(4)&(5)&(6)&(7)&(8)&(9)&(10)\\
(A)& &          &-\BSN{5}{8}& &\BSN{3}{8}& & &-\BSN{3}{5}& &\\ 
(B)& &-\BSN{8}{9}&          & & & & &\BSN{2}{9}&-\BSN{2}{8}&\\ 
(C)& &\BSN{4}{6} &          &-\BSN{2}{6}& &\BSN{2}{4}& & & & \\ 
(D)& &          &          & & &-\BSN{7}{10}&\SN{6}{10}& & &-\BSN{6}{7}\\ 
(E)&-\SN{3}{7}& &\BSN{1}{7}& & & &-\SN{1}{3}& & & \\ 
(F)&-\SN{4}{7}& & &\BSN{1}{7}& & &-\SN{1}{4}& & & \\ 
(G)&\SN{5}{9}& & & &-\BSN{1}{9}& & & &\BSN{1}{5} & \\ 
(H)&\SN{5}{10}& & & &-\BSN{1}{10}& & & & &\BSN{1}{5}\\ 
(I)&\SN{5}{7}& & & &-\BSN{1}{7}& &\SN{1}{5} &   & &
\end{smallmatrix}
\right)
\end{equation}

The 9 by 8 submatrix formed from columns 2-6 and 8-10, shown in bold, 
above, is referred to as $S_9$.

We also consider extensively the submatrix $M_8$ formed from the first eight rows:
\begin{equation}
\label{e:m8}
M_8 =
\left(
\begin{smallmatrix}
&(1)&(2)&(3)&(4)&(5)&(6)&(7)&(8)&(9)&(10)\\
(A)& &          &-\BSN{5}{8}& &\SN{3}{8}& & &-\BSN{3}{5}& &\\ 
(B)& &-\BSN{8}{9}&          & & & & &\BSN{2}{9}&-\BSN{2}{8}&\\ 
(C)& &\BSN{4}{6} &          &-\BSN{2}{6}& &\BSN{2}{4}& & & & \\ 
(D)& &          &          & & &-\BSN{7}{10}&\SN{6}{10}& & &-\BSN{6}{7}\\ 
(E)&-\SN{3}{7}& &\BSN{1}{7}& & & &-\SN{1}{3}& & & \\ 
(F)&-\SN{4}{7}& & &\BSN{1}{7}& & &-\SN{1}{4}& & & \\ 
(G)&\SN{5}{9}& & & &-\SN{1}{9}& & & &\BSN{1}{5} & \\ 
(H)&\SN{5}{10}& & & &-\SN{1}{10}& & & & &\BSN{1}{5}
\end{smallmatrix}
\right)
\end{equation}

The 8 by 7 submatrix formed from columns 2, 3, 4, 6, 8, 9 and~10, shown in bold
above, is referred to as $S_8$.

\begin{lem}
\label{l:simplex8}
With the conditions on $\boldsymbol{\theta}$ of
theorem~\ref{t:main},
the 8 by 7 matrix
$S_8$
is a simplex.
\end{lem}
\begin{proof}
The eight 7 by 7 subdeterminants are:
\begin{equation}
\label{e:subdets}
\begin{pmatrix}
-\SN{4}{6}\SN{1}{7}\SN{1}{7}\SN{7}{10}\SN{2}{9}\SN{1}{5}\SN{1}{5}\\
+\SN{4}{6}\SN{1}{7}\SN{1}{7}\SN{7}{10}\SN{3}{5}\SN{1}{5}\SN{1}{5}\\
-\SN{8}{9}\SN{1}{7}\SN{1}{7}\SN{7}{10}\SN{3}{5}\SN{1}{5}\SN{1}{5}\\
+\SN{8}{9}\SN{1}{7}\SN{1}{7}\SN{2}{4}\SN{3}{5}\SN{1}{5}\SN{1}{5}\\
-\SN{4}{6}\SN{5}{8}\SN{1}{7}\SN{7}{10}\SN{2}{9}\SN{1}{5}\SN{1}{5}\\
+\SN{8}{9}\SN{1}{7}\SN{2}{6}\SN{7}{10}\SN{3}{5}\SN{1}{5}\SN{1}{5}\\
-\SN{4}{6}\SN{1}{7}\SN{1}{7}\SN{7}{10}\SN{3}{5}\SN{2}{8}\SN{1}{5}\\
+\SN{8}{9}\SN{1}{7}\SN{1}{7}\SN{2}{4}\SN{3}{5}\SN{1}{5}\SN{6}{7}
\end{pmatrix}
\end{equation}
Since the constraints on the angles in the statement of the main
theorem require that $\SN{i}{j} > 0$ for all pairs $i, j$
appearing in these subdeterminants, the signs of the subdeterminant
alternate, and theorem~\ref{t:simplex} applies.
\end{proof}

\begin{lem}
\label{l:nequals}
With the conditions on $\boldsymbol{\theta}$ of theorem~\ref{t:main},
if
\begin{equation}
\SN{8}{9}\SN{1}{10}\SN{2}{4}\SN{3}{5}\SN{6}{7} 
+\SN{4}{6}\SN{1}{9}\SN{7}{10}\SN{3}{5}\SN{2}{8} 
- \SN{4}{6}\SN{3}{8}\SN{7}{10}\SN{2}{9}\SN{1}{5}
= 0
\end{equation}
then the system~(\ref{e:system}) is insoluble.
\end{lem}
\begin{proof}
From above, 
$S_8$
is a simplex.
We can compute the three determinants required by
theorem~\ref{t:mus}, as follows:
\begin{equation}\begin{split}
d_1 = & | M\left[A-H;1,2,3,4,6,8,9,10\right] | = \\
&-\SN{5}{10}\SN{8}{9}\SN{1}{7}\SN{1}{7}\SN{2}{4}\SN{3}{5}\SN{1}{5}\SN{6}{7} \\
&-\SN{5}{9}\SN{4}{6}\SN{1}{7}\SN{1}{7}\SN{7}{10}\SN{3}{5}\SN{2}{8}\SN{1}{5} \\
&+\SN{4}{7}\SN{8}{9}\SN{1}{7}\SN{2}{6}\SN{7}{10}\SN{3}{5}\SN{1}{5}\SN{1}{5} \\
&+\SN{3}{7}\SN{4}{6}\SN{5}{8}\SN{1}{7}\SN{7}{10}\SN{2}{9}\SN{1}{5}\SN{1}{5}
\end{split}\end{equation}
\begin{equation}\begin{split}
\label{e:d5}
d_5 = & | M\left[A-H;2,3,4,5,6,8,9,10\right] | = \\
&+\SN{4}{6}\SN{1}{7}\SN{1}{7}\SN{3}{8}\SN{7}{10}\SN{2}{9}\SN{1}{5}\SN{1}{5} \\
&-\SN{4}{6}\SN{1}{7}\SN{1}{7}\SN{1}{9}\SN{7}{10}\SN{3}{5}\SN{2}{8}\SN{1}{5} \\
&-\SN{8}{9}\SN{1}{7}\SN{1}{7}\SN{1}{10}\SN{2}{4}\SN{3}{5}\SN{1}{5}\SN{6}{7}
\end{split}\end{equation}
\begin{equation}\begin{split}
d_7 =& | M\left[A-H;2,3,4,6,7,8,9,10\right] | = \\
&+\SN{4}{6}\SN{5}{8}\SN{1}{7}\SN{7}{10}\SN{1}{3}\SN{2}{9}\SN{1}{5}\SN{1}{5} \\
&-\SN{8}{9}\SN{1}{7}\SN{1}{7}\SN{2}{4}\SN{6}{10}\SN{3}{5}\SN{1}{5}\SN{1}{5} \\
&+\SN{8}{9}\SN{1}{7}\SN{2}{6}\SN{7}{10}\SN{1}{4}\SN{3}{5}\SN{1}{5}\SN{1}{5}
\end{split}\end{equation}
By normalization (theorem~\ref{t:normal}) we can show that:
\begin{equation}
\label{e:d1_d5_d7}
\SN{1}{7}\SN{1}{5}d_1 
=  \SN{1}{5}\SN{5}{7}d_5 
=  \SN{1}{7}\SN{5}{7} d_7
\end{equation}
Given the premise of the lemma, $d_5 = 0$, and hence so are
$d_1$ and $d_7$. Thus by theorem~\ref{t:mus} 
the system:
\begin{equation}
M_8\boldsymbol{r} > \boldsymbol{0}
\end{equation}
is insoluble, and hence so is~(\ref{e:system}).
\end{proof}

\begin{lem}
With the conditions on $\boldsymbol{\theta}$ of
theorem~\ref{t:main},
if
\begin{equation}
\SN{8}{9}\SN{1}{10}\SN{2}{4}\SN{3}{5}\SN{6}{7} 
+\SN{4}{6}\SN{1}{9}\SN{7}{10}\SN{3}{5}\SN{2}{8} 
- \SN{4}{6}\SN{3}{8}\SN{7}{10}\SN{2}{9}\SN{1}{5}
< 0
\end{equation}
then
the 9 by 8 matrix
$S_9$
is a simplex.
\end{lem}
\begin{proof}
The first eight of the nine 8 by 8 subdeterminants
are the same as in equation~\ref{e:subdets} multiplied
by $-\SN{1}{7}$. The first is positive, the eighth is negative.

The ninth subdeterminant is
\begin{equation}
+\SN{4}{6}\SN{1}{7}\SN{1}{7}\SN{3}{8}\SN{7}{10}\SN{2}{9}\SN{1}{5}\SN{1}{5} 
-\SN{4}{6}\SN{1}{7}\SN{1}{7}\SN{1}{9}\SN{7}{10}\SN{3}{5}\SN{2}{8}\SN{1}{5}
-\SN{8}{9}\SN{1}{7}\SN{1}{7}\SN{1}{10}\SN{2}{4}\SN{3}{5}\SN{1}{5}\SN{6}{7}
\end{equation}
which is $-\SN{1}{7}\SN{1}{7}\SN{1}{5}$ times the
negative value in the premise of the lemma.
So the nine values alternate in sign, and theorem~\ref{t:simplex} applies.
\end{proof}

\begin{lem}
\label{l:nnegative}
With the conditions on $\boldsymbol{\theta}$ of theorem~\ref{t:main},
if
\begin{equation}
\SN{8}{9}\SN{1}{10}\SN{2}{4}\SN{3}{5}\SN{6}{7} 
+\SN{4}{6}\SN{1}{9}\SN{7}{10}\SN{3}{5}\SN{2}{8} 
- \SN{4}{6}\SN{3}{8}\SN{7}{10}\SN{2}{9}\SN{1}{5}
< 0
\end{equation}
then the system~(\ref{e:system}) is insoluble.
\end{lem}
\begin{proof}
From above, 
$S_9$
is a simplex.
We can compute the two determinants required by
theorem~\ref{t:mus}, as follows:
\begin{equation}\begin{split}
d_{1,5} = & |M\left[A-I;1-6,8-10\right] |= \\
&-\SN{5}{7}\SN{8}{9}\SN{1}{7}\SN{1}{7}\SN{1}{10}\SN{2}{4}\SN{3}{5}\SN{1}{5}\SN{6}{7}\\
&-\SN{5}{7}\SN{4}{6}\SN{1}{7}\SN{1}{7}\SN{1}{9}\SN{7}{10}\SN{3}{5}\SN{2}{8}\SN{1}{5}\\
&+\SN{5}{7}\SN{4}{6}\SN{1}{7}\SN{1}{7}\SN{3}{8}\SN{7}{10}\SN{2}{9}\SN{1}{5}\SN{1}{5}\\
&+\SN{5}{10}\SN{8}{9}\SN{1}{7}\SN{1}{7}\SN{1}{7}\SN{2}{4}\SN{3}{5}\SN{1}{5}\SN{6}{7}\\
&+\SN{5}{9}\SN{4}{6}\SN{1}{7}\SN{1}{7}\SN{1}{7}\SN{7}{10}\SN{3}{5}\SN{2}{8}\SN{1}{5}\\
&-\SN{4}{7}\SN{8}{9}\SN{1}{7}\SN{2}{6}\SN{1}{7}\SN{7}{10}\SN{3}{5}\SN{1}{5}\SN{1}{5}\\
&-\SN{3}{7}\SN{4}{6}\SN{5}{8}\SN{1}{7}\SN{1}{7}\SN{7}{10}\SN{2}{9}\SN{1}{5}\SN{1}{5}
\end{split}\end{equation}
\begin{equation}\begin{split}
d_{5,7} = & |M\left[A-I;2-10\right] |= \\
&+\SN{4}{6}\SN{5}{8}\SN{1}{7}\SN{1}{7}\SN{7}{10}\SN{1}{3}\SN{2}{9}\SN{1}{5}\SN{1}{5}\\
&-\SN{4}{6}\SN{1}{7}\SN{1}{7}\SN{3}{8}\SN{7}{10}\SN{1}{5}\SN{2}{9}\SN{1}{5}\SN{1}{5}\\
&+\SN{4}{6}\SN{1}{7}\SN{1}{7}\SN{1}{9}\SN{7}{10}\SN{1}{5}\SN{3}{5}\SN{2}{8}\SN{1}{5}\\
&+\SN{8}{9}\SN{1}{7}\SN{1}{7}\SN{1}{10}\SN{2}{4}\SN{1}{5}\SN{3}{5}\SN{1}{5}\SN{6}{7}\\
&-\SN{8}{9}\SN{1}{7}\SN{1}{7}\SN{1}{7}\SN{2}{4}\SN{6}{10}\SN{3}{5}\SN{1}{5}\SN{1}{5}\\
&+\SN{8}{9}\SN{1}{7}\SN{2}{6}\SN{1}{7}\SN{7}{10}\SN{1}{4}\SN{3}{5}\SN{1}{5}\SN{1}{5}
\end{split}\end{equation}
By normalization (theorem~\ref{t:normal}) these are both zero. 
Thus by theorem~\ref{t:mus} 
the system~(\ref{e:system}) is insoluble.
\end{proof}

The main theorem is thus proved by combining lemmas~\ref{l:nequals} and~\ref{l:nnegative}.

\section{Oriented Matroids}
\label{s:om}
In the remainder of the paper, we assume familiarity with oriented matroids,
particularly with results from~\cite{bjorner:oriented}.

This section proves one result concerning oriented matroids of directed graphs,
which we will use in the next section.

Prior to that, we briefly review two classes of oriented matroids: 
acyclic uniform rank 2 oriented matroids,
and those derived from a directed graph.

We will be interested in strong maps between such oriented matroids,
and we briefly review these.

Oriented matroid theory makes extensive use of signed sets.
A signed set $A$ is a disjoint pair $(A^+,A^-)$. Its ground set
$\underline{A} = A^+ \cupdot A^-$. 
Its opposite $-A = (A^-,A^+)$.
For every $x$ (in some base set),
$A$ acts as a function to $\{ +1, 0, - 1\}$ defined by:
\begin{equation}
A(x) = \begin{cases}
+1 & x \in A^+ \\
-1 & x \in A^- \\
0 & \textrm{otherwise}
\end{cases}
\end{equation}

\subsection{Oriented matroids from a total order}

Given a finite set $E$, totally ordered by $<$,
we can construct a uniform rank 2 oriented matroid\footnote{
See pages~285-287 of~\cite{bjorner:oriented} for discussion of all rank 2
oriented matroids.
} 
$\Mat(<)$.
The circuit space $\Circ(<)$ signs each of the three element
subsets of $E$ (which are the circuits of the uniform rank 2 matroid on $E$).
\begin{equation}
\label{e:circuits}
\Circ(<) = \left\{ ( \{ e_1, e_3 \}, \{ e_2 \} ), 
( \{ e_2 \} , \{ e_1, e_3 \} ) : e_1 < e_2 < e_3 \right\}
\end{equation}
Since, none of these is positive, $\Mat(<)$ is acyclic,
and $\DualMat(<)$ is totally cyclic (see page 123 of~\cite{bjorner:oriented}, proposition 3.4.8).
The cocircuits can be given explicitly:
\begin{equation}
\CoC(<) = \left\{ ( \{ e' : e' < e \}, \{ e'' : e'' > e \} ), 
( \{ e'' : e'' > e \}, \{ e' : e' < e \} ) :  e \in E \right\}
\end{equation}
 
 \subsection{Oriented matroids from a directed graph}
 
 Given a directed graph $\vec{G}=(V,\vec{E})$, with underlying 
 connected simple
 graph $G=(V,E)$ (noting that we have restricted ourselves to
 graphs without loops or parallel or anti-parallel edges\footnote{
 Neither these restrictions, nor the restriction to connected $G$, are needed for these definitions, but
 simplify them: allowing us to identify an edge in $\vec{G}$ with an edge in $G$.
 }), then, we can construct an oriented matroid $\Mat(\vec{G})$ , on the edge set of $\vec{G}$ 
 in the following fashion,
 (see page 2 of~\cite{bjorner:oriented}).
 
 The set of circuits of the underlying matroid are simply the cycles of $G$.
 Each corresponds to two opposite signed sets, by following the cycle in $\vec{G}$, either
 `clockwise' or `anticlockwise',
 to get directed pairs of vertices, each being an edge, or an inverted edge.
The inverted edges in this cycle are negatively signed, and the remaining edges
in the cycle are
positively signed. This gives us the set of signed circuits $\Circ(\vec{G})$.

The cocircuits are similarly defined using minimal cuts of $G$.
A minimal cut is a cut dividing $G$ into two components. The edges
are signed depending on their direction in $\vec{G}$. More formally,
given a connected induced subgraph of $G$
with vertices $A \subset V$, such that the subgraph induced by
$V \setminus A$ is also connected, then:
\begin{equation}
C^*_A = (
\{ (u,v) \in \vec{E} : u \in A, v \in V \setminus A \},
\{ (u,v) \in \vec{E} : v \in A, u \in V \setminus A \},
) 
\end{equation}
and then
\begin{equation}
\CoC(\vec{G}) = \left\{ C^*_A
: A, V \setminus A \textrm{ connected in } G \right\}
\end{equation}

We will use later the specific cocircuits which isolate a vertex,
for $v \in V$:
\begin{equation}
\label{e:cocircuit}
C^*_v = C^*_{\{ v \} } = (
\{ (u,v) \in \vec{E} : u \in V \},
\{ (v,u) \in \vec{E} : u \in V \},
)
\end{equation}

We use the oriented matroid notion of totally cyclic:
\begin{defn}
A directed graph $\vec{G}$ is totally cyclic, if every edge is contained
in a directed cycle.
\end{defn}

In this case $\Mat(\vec{G})$ is also totally cyclic.

Since we are restricting ourselves to loop-free and parallel-free graphs, $\Mat(\vec{G})$
is simple (i.e. has no loops or parallel elements). Moreover, if $G$ is three edge connected,
then  $\DualMat(\vec{G})$ is also simple, since a coloop would form a cocircuit of size one,
and a pair of coparallel elements would form a cocircuit of size two, either of which would
disconnect $G$, by the construction of $\CoC(\vec{G})$. Thus:

\begin{prop}
\label{p:simpledual}
If $\vec{G}$ is a directed graph, with an underlying three edge connected 
simple graph, then $\DualMat(\vec{G})$ is simple.
\end{prop}

\subsection{Strong Maps}

Strong maps are discussed in \cite{bjorner:oriented}, section 7.7, in particular,
proposition 7.7.1 and definition 7.7.2 on page~319. We combine these as:
\begin{defn}
Given two oriented matroids $\Mat_1, \Mat_2$ on the same ground set $E$,
then there is a {\em strong map} from $\Mat_1$ to $\Mat_2$, and we write
$\Mat_1 \longrightarrow \Mat_2$ if either of these equivalent conditions hold:
\begin{itemize}
\item Every cocircuit of $\Mat_2$ is a covector of $\Mat_1$
\item Every circuit of $\Mat_1$ is a vector of $\Mat_2$
\end{itemize}
\end{defn}

\subsection{Simple acyclic oriented matroids}
We have seen that each finite total order gives a simple acyclic oriented matroid
$\Mat(<)$.
For each simple acyclic oriented matroid $\Mat'$ with ground set $E$,
we can always find a total order $<$ of $E$ such that there is a strong map
from $\Mat'$ to $\Mat(<)$.

We use the following technical lemma, concerning orthogonality (see page 115 of~\cite{bjorner:oriented}).

\begin{lem}
If $X, Y_1, Y_2$ are signed sets over $E$, with
$X \perp Y_1$ and $X \perp Y_2$, and $Y_1, Y_2$ being conformal,
then $X \perp Y_1  \circ Y_2$.
\end{lem}
\begin{proof}
If $\underline{X} \cap \underline{Y_1 \circ Y_2}$ is empty,
then there is nothing to prove. Otherwise, at least one of
$\underline{X} \cap \underline{Y_1}$ 
and
$\underline{X} \cap \underline{Y_2}$ is non-empty.
If the former, then there are $e,f \in \underline{X} \cap \underline{Y_1}$ with
\begin{equation}
X(e)(Y_1 \circ Y_2)(e) =  X(e)Y_1(e) = - X(f)Y_1(f) = - X(f)(Y_1 \circ Y_2)(f)
\end{equation}
If not then there are $e, f \in \underline{X} \cap \underline{Y_2}$
with $e,f \notin  \underline{Y_1}$, so that:
\begin{equation}
X(e)(Y_1 \circ Y_2)(e) =  X(e)Y_2(e) = - X(f)Y_2(f) = - X(f)(Y_1 \circ Y_2)(f)
\end{equation}
\end{proof}
\begin{cor}
The vectors of an oriented matroid are orthogonal to its covectors.
\end{cor}

If $\Mat'$ is a simple acyclic oriented matroid, we can use the topological 
representation theorem,
to find a representation of $\Mat'$ as a repetition free arrangement
of pseudo-spheres, (see page~234 of~\cite{bjorner:oriented}).
Since $\Mat'$ is acyclic, this arrangement has
a positive face $F$, labelled $(E,\emptyset)$ and its opposite $-F$.
We can draw a jordan curve 
joining a point inside $F$ to a point inside $-F$,
which crosses each pseudo-sphere at distinct points. Each pseudo-sphere
corresponds to an $e \in E$, and the order in which 
the curve crosses the pseudo-spheres induces an order $<$ on $E$.

Formally, we use the tope graph to show the following result. 
This is explained in section 4.2 of~\cite{bjorner:oriented}; note that
their results are stated for {\em simple} oriented matroids.
The proposition we use (their 4.2.3) says that,
in a simple oriented matroid, the distance
between two topes $X, Y$ in the tope graph is the size of their separation set,
$S(X,Y) = \{ e \in E : X_e = -Y_e \neq 0 \}$.

\begin{lem}
\label{l:simpletotal}
If $\Mat'$ is a simple acyclic oriented matroid on $E$, then there is
a total order $<$ of $E$, such that
$\Mat' \longrightarrow \Mat(<)$.
\end{lem}
\begin{proof}
Label the edges of the tope graph 
with the unique 
$e \in E$ that has sign zero in 
the subtope joining the two topes (note this is where simplicity is used).
Say that a path crosses $e$ if an edge in the path is labelled with $e$.

Since $\Mat'$ is acyclic, $F = (E,\emptyset)$ is a tope of $\Mat'$.

Using their proposition 4.2.3, we find that the distance in the tope graph
from $F$ to $-F$ is $|E|$. Take a path $P$ of this length, joining $F$ and $-F$
in the tope graph.

By induction, using the same proposition, 
 $P$ crosses
each $e$ exactly once.

Define $<$ as corresponding to the order of the labels along
$P$. Then the subtopes along the path are precisely:
\begin{equation}
\left\{ ( \{ e' : e' > e \}, \{ e'' : e'' < e \} ) :  e \in E \right\}
\end{equation}
which along with their opposites are all of the cocircuits of $\Mat(<)$,
which gives the strong map.
\end{proof}
\begin{cor}
\label{c:simpletotal}
If $\vec{G} = (V,\vec{E})$ is a totally cyclic, simply three edge connected, directed graph
then there is some total order $<$ of $\vec{E}$ such that
$\DualMat(\vec{G}) \longrightarrow \Mat(<)$
\end{cor}
\begin{proof}
$\DualMat(\vec{G})$ is simple and acyclic.
\end{proof}

\section{Twisted Graphs}
\label{s:twisted}

Consider figure~\ref{f:graph}.
\begin{figure}[htbp]
\begin{center}
\begin{tabular}{{ccc}}
    \subfigure[Undirected (G)]{\label{f:ugraph}\scalebox{0.4}{\includegraphics{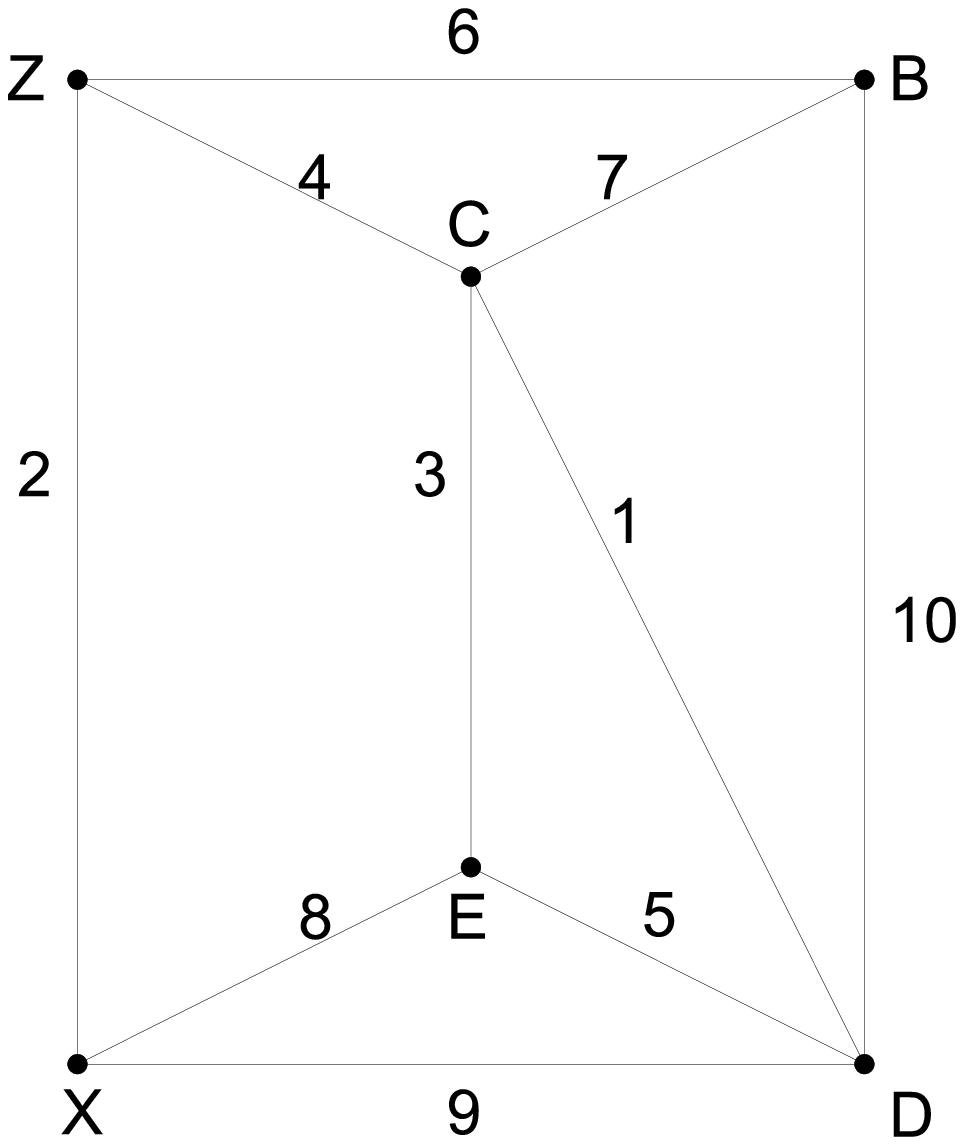}}} &
    $\:\:$
    &
    \subfigure[Directed (G,E)]{\label{f:digraph}\scalebox{0.4}{\includegraphics{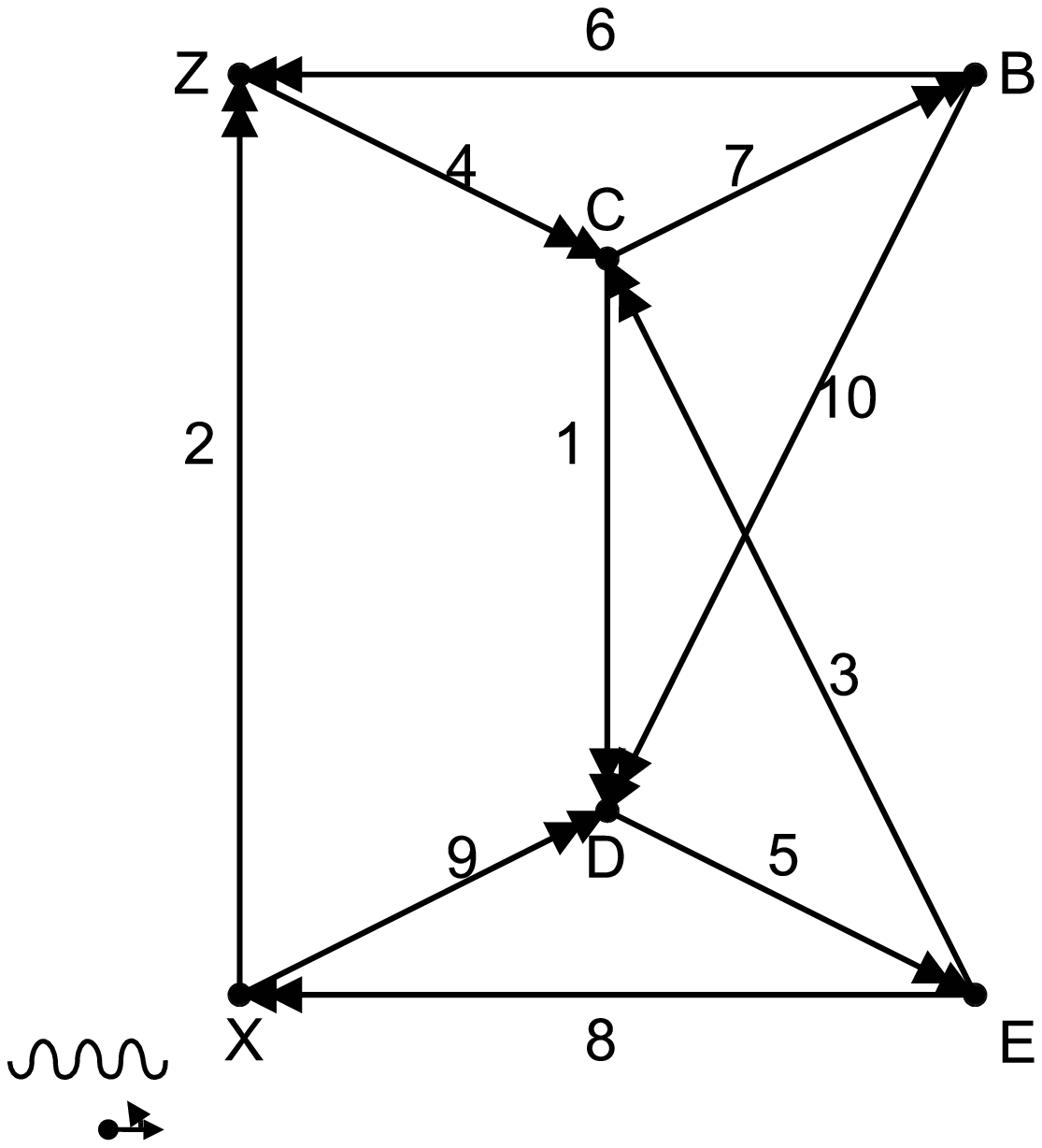}}}
\end{tabular}
\caption{The graph underlying fig.~\ref{f:straight}}
\label{f:graph}
\end{center}
\end{figure}

Two different drawings of the same (labelled) graph are shown. The
first is an undirected version of the second.
Moreover, 
since we are excluding
loops and parallel or anti-parallel edges,
we can identify each undirected edge in the first 
with a directed edge in the second.
The second can be converted into fig.~\ref{f:straight} (on page~\pageref{f:straight}), 
using the following recipe.
View the picture as a line arrangement in the Euclidean plane. Notice the choice
of polar origin.
Twist each edge of the graph, just a little, by moving the pointy head further away from the
origin, while keeping the other end fixed.
Add a pinch of imagination, and we have figure~\ref{f:straight}.

Let's formalise that recipe, keeping the pinch of imagination down to a minimum.

We start by observing that the crucial properties of figure~\ref{f:straight} that we care about
are listed in theorem~\ref{t:main}. These are the orientations of nine triangles,
and the constraints on the ordering of the lines. 
These constraints describe a partial order, $\sjle$.
For now, we will ignore the triangle $\{ 1, 5, 7 \}$ since it clearly plays a special role (for example,
it plays no part in
the proof of lemma~\ref{l:nequals}).
Each of the remaining eight triangles corresponds to one
of the six vertices in the graph. Vertices $C$ and $D$, having degree
4, have two corresponding triangles.

Given any total order $<$ over $E$ extending $\sjle$, then the triangles
can be expressed as a subset $A$ of the $\Circ(<)$.

In order to express this, we extend the idea of the circuits corresponding
to an order in the following fashion:
\begin{defn}
\label{d:partialcircuits}
Given a partial order $\sjle$ over $E$, the circuits $\Circ(\sjle)$
are those circuits that are circuits of $\Mat(<)$ for every total
order $<$ of $E$ which extends $\sjle$.
\end{defn}
i.e.
\begin{equation}
\Circ(\sjle) = \bigcap_{ < \supseteq \sjle } \Circ(<)
\end{equation}
For a total order $<$, definition~\ref{d:partialcircuits} and equation~(\ref{e:circuits})
are consistent.

\begin{defn}
\label{d:twistedgraph}
A twisted graph $T = (\vec{G}, G, \sjle, A)$
is a directed graph $\vec{G} = (V,\vec{E})$, 
with $\sjle$ partially ordering $E$, such that:
\begin{enumerate}
\item
 $G = (V,E)$ is the underlying simple graph of $\vec{G}$
 \item
 $G$ is
three edge connected
\item
$A \subset \Circ(\sjle)$
\item
for every total order $<$ of $E$ extending $\sjle$
there is a strong map $\DualMat(\vec{G}) \rightarrow \Mat(<)$.
\item
$A$ can be partitioned into $\{ A_v : v \in V \}$
with $C^*_v$ (as in equation~(\ref{e:cocircuit})) being the conformal composition over
$A_v$
for each $v$.
\end{enumerate}
\end{defn}


Essentially, a twisted graph is formed from the signed incidence matrix 
of $\vec{G}$, by taking every row with more than three entries, and splitting
it into rows of $A$ (considered as a matrix), 
each having three entries. Every row has the form $( -, +, - )$ or $( -, +, -)$ with
the rest of the entries being zero.

For example, we take the incidence matrix of fig.~\ref{f:digraph}:
\begin{equation}
\left(
\begin{smallmatrix}
&(1)&(2)&(3)&(4)&(5)&(6)&(7)&(8)&(9)&(10)\\
(E)& &          &-& &+& & &-& &\\ 
(X)& &-&          & & & & &+&-&\\ 
(Z)& &+ &          &-& &+& & & & \\ 
(B)& &          &          & & &-&+& & &-\\ 
(C)&-& &+&+& & &-& & & \\ 
(D)&+& & & &-& & & &+ &+  
\end{smallmatrix}
\right)
\end{equation}
and split the last two rows, because they have more than three entries, to give:
\begin{equation}
\left(
\begin{smallmatrix}
&(1)&(2)&(3)&(4)&(5)&(6)&(7)&(8)&(9)&(10)\\
(E)& &          &-& &+& & &-& &\\ 
(X)& &-&          & & & & &+&-&\\ 
(Z)& &+ &          &-& &+& & & & \\ 
(B)& &          &          & & &-&+& & &-\\ 
(C)&-& &+& & & &-& & & \\ 
(C')&-& & &+& & &-& & & \\ 
(D)&+& & & &-& & & &+ & \\ 
(D')&+& & & &-& & & & &+
\end{smallmatrix}
\right)
\end{equation}
which, with the insertion of appropriate $\SN{i}{j}$ values is $M_8$
from equation~(\ref{e:m8}).

While we don't exploit such a view of a twisted graph, 
for explanatory purposes,
we express it more formally as:
\begin{prop}
\label{p:twisted}
If $T = (\vec{G}, G, \sjle, A)$ is a twisted graph then there is a function
$v: A \rightarrow V$ such that:
\begin{enumerate}
\item
For each $a \in A$, there is a single vertex $v(a) \in V$
incident with each of the edges in $\underline{a}$.
\item
For each $e \in E$, there is some $a \in A$, with $a(e)=1$,
and $e = (u,v(a))$, for some $u \in V$
\item
For each $e \in E$, there is some $a \in A$, with $a(e)=-1$,
and $e = (v(a),u)$, for some $u \in V$
\item
For each each $e \in E$, there are $a, a' \in A$, with 
$a(e)=1$ and  $a(e)=-1$, and $e = (v(a'),v(a))$.
\item
Given $a_1, a_2 \in A$, and some $e \in E$ 
for which $a_1(e) = a_2(e) \neq 0$ then  $v(a_1)=v(a_2)$.
\item
If $a_1, a_2 \in A$ are conformal, and $\underline{a_1} \cap \underline{a_2} \neq \emptyset$
then $v(a_1)=v(a_2)$.
\end{enumerate}
\end{prop}
\begin{proof}
By construction, each $a \in A$, is in some $A_v$. Since
$C^*_v$ is the conformal composition of $A_v$, we have for each edge
$e \in \underline{a}$ that $a(e) = C^*_v(e)$. 
Thus, from~(\ref{e:cocircuit}), for each edge $e \in a$, 
there is a $u \in V$, such that
either $e = (u,v)$ and $a(e)=C^*_v(e)=1$ or
$e=(v,u)$ and $a(e)=C^*_v(e)=-1$.
So $v(a)=v$ is incident with each edge in $a$. If there was some other
vertex also incident with each edge in $a$, then we would have three
parallel edges, which we have excluded.

For the second point, we note that if $e=(u,v)$, then $C^*_v(e)=1$,
and there is some $A_v \subset A$, such that the conformal composition
over $A_v$ is $C^*_v$, so there is at least one $a \in A_v$ with $a(e) = C^*_v(e)$.
By the previous paragraph, for this $a$, we have $v(a) = v$.

The third point is similar.

The fourth point is a rephrasing of the second and third points.

For the fifth point, $e = (v(a'),v(a))$ as in the fourth point, and for the case
$v(a_1)=1$ we can, without loss of generality, take $a = a_1$.
From the first point, $v(a_2)$ is incident with $e$, and $a_2 \neq a'$ because
$a_2(e) \neq a'(e)$. Thus $a_2 = a_1$.

The sixth point follows from the fifth by
taking any $e \in \underline{a_1} \cap \underline{a_2}$.
\end{proof}

We explore the notion of twisted graphs, with the aim
of deriving minimal insoluble systems from them, corresponding
to constraints on line arrangements, or the realizations of rank 3
oriented matroids.

Two preliminary definitions:
\begin{defn}
A directed graph $\vec{G}$ can be {\em twisted}
if there is a twisted graph $(\vec{G}, G, \sjle, A)$,
for some $A$, and $\sjle$.
\end{defn}

\begin{defn}
A simple graph $G$ can be twisted
if there is some orientation $\vec{G}$ of $G$, such that
$\vec{G}$ can be twisted.
\end{defn}

We start by showing that there are infinitely many twisted graphs.
Given the requirement for there to be the strong map,
we wish to find a totally cyclic orientation of $G$.

\begin{prop}
\label{p:orientability}
Given a three-edge connected simple graph $G$, then there
is a totally cyclic directed graph $\vec{G}$ with underlying graph
$G$.
\end{prop}
\begin{proof}
Arbitrarily direct edge each of $G$ to get $\vec{G_0}$.
Take the oriented matroid $\Mat = \Mat(\vec{G_0})$.
Take an arbitrary maximal vector $v$ of $\Mat$. Reorient $\Mat$ so that
$v$ is positive. Apply the same reorientation to $\vec{G_0}$ to
get $\vec{G}$. 
\end{proof}

\begin{lem}
\label{l:twistability}
Every simply  three-edge connected, totally cyclic, directed graph $\vec{G}$, can be twisted.
\end{lem}
\begin{proof}
From corollary~\ref{c:simpletotal} we find a total order
$<$ over $E$, with the strong map $\DualMat(\vec{G}) \longrightarrow \Mat(<)$.
For every $v \in V$, $C^*_v$ is a signed minimal cut of  $\vec{G}$,
and hence a cocircuit of $\Mat(\vec{G})$, i.e. a circuit of $\DualMat(\vec{G})$.
By the strong map, $C_v$ is a vector of $\Mat(<)$, and hence a conformal
composition of some set of circuits $A_v$ of $\Mat(<)$.
We put $A = \bigcup_{v \in V}A_v$, and we have a twisted graph, with a total order.
\end{proof}
\begin{cor}
Every three-edge connected, simple graph $G$ can be twisted.
\end{cor}

The definition of twisted graph permits adding more triangles 
to $A$ than is strictly needed (for some vertices incident
with more than four edges).
So the following definition is helpful:
\begin{defn}
A twisted graph $T= (\vec{G}, G, \sjle, A)$ is  {\em irredundant}
if:
\begin{itemize}
\item
for all proper subsets $A' \subsetneq A$, then
$T'= (\vec{G}, G, \sjle, A)$ is not a twisted graph.
\item
for all partial orders $\sjle^{\prime} \subsetneq \sjle$, then
$T'= (\vec{G}, G, \sjle^{\prime}, A)$ is not a twisted graph.
\end{itemize}
\end{defn}

\subsection{Why the strong maps?}
Definition~\ref{d:twistedgraph} requires the existence of at least one strong map.
However,  the results in {\em this} paper do not depend on
that condition, nor even the potentially weaker condition that
$\vec{G}$ be totally cyclic.
\begin{figure}[htbp]
\begin{center}
\scalebox{0.4}{\includegraphics{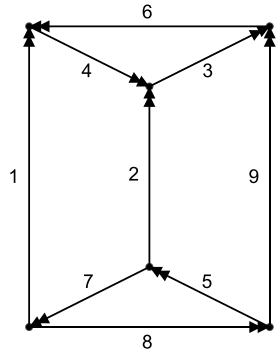}}
\caption{Should this be a twisted graph?}
\label{f:badTwisted}
\end{center}
\end{figure}

The motivation for this condition is to prevent graphs like in figure~\ref{f:badTwisted},
from being twisted.
The numbering in that figure shows a potential ordering. This satisfies the conclusions of
proposition~\ref{p:twisted}, and definition~\ref{d:twistedgraph} without the strong map
constraint. The positive cut of size three prevents there from being a strong map.

Currently, I think the various conjectures I make in the subsequent sections
are more plausible with the strong map condition. This is the motivation.

\subsection{Matrices, Realisations and Twisted Graphs}

Of course, our interest in twisted graphs is because of their relationship
to matrices that express conditions on line arrangements.
A twisted graph is realized by a line arrangement $(\mathcal{L}_e)_{e \in E}$,
with a polar frame of reference,
if for every $a =( \{ e, f \}, \{ g \} ) \in A$, the triangle $\{ L_e, L_f, L_g \}$ is positively oriented,
and likewise for every $( \{ e \}, \{ f, g \} )\in A$ there is a negatively oriented triangle.

To express this in the fashion of the previous sections, given a twisted graph $T$,
we consider a vector of
angles $\boldsymbol{\theta}$ indexed by $E$, and construct
a $|A|$ by $|E|$ matrix $\Sigma = \Sigma(T)$ with entries being either $0$ or $\sin(\theta_f - \theta_e)$
or $-\sin(\theta_f - \theta_e)$ with $e, f \in E$ and $e \sjle f$.

The actual entries are given, for each $a \in A$ and $e \in E$ by:
\begin{equation}
\Sigma_{a,e} = \begin{cases}
0 & e \notin a \\
a(e)\sin(\theta_e''-\theta_e') & \underline{a} = \{ e, e', e'' \}, e' \sjle e'' 
\end{cases}
\end{equation}

So as we have seen, there is a $T$ corresponding to figure~\ref{f:digraph}, 
with $\Sigma(T)$ being a matrix consisting of
the rows of $M_8$, equation (\ref{e:m8}), in some order.

This matrix can be understood either as a matrix of formal expressions over $\boldsymbol{\theta}$
or as a matrix over $\Real$, given specific values of $\boldsymbol{\theta}$. So that
we can compute signs of determinants, we restrict the values of interest, to those
that respect $\sjle$.

\begin{defn}
\label{d:respect}
A vector of real numbers $\boldsymbol{\theta} \in \Real^E$ {\em respects} $\sjle$, when:
\begin{itemize}
\item 
For all $e \in E$, $0 < \theta_e < 180$
\item
For all $e, f \in E$, if $e \sjle f$ then $\theta_e < \theta_f$.
\end{itemize}
\end{defn}

Then a solution in $\boldsymbol{\theta}, \boldsymbol{r} \in \Real^E$,
with $\boldsymbol{\theta}$ respecting $\sjle$,
for the system:
\begin{equation}
\Sigma(T) \boldsymbol{r} > 0
\end{equation}
provides polar coordinates of a line arrangement realizing $T$.
If the solution in  $\boldsymbol{r} $ is not positive, then 
the arrangement can still be drawn,
interpreting negative values for any $r_e$ as in the opposite direction to 
positive values.
Moving the polar origin will then give a solution that is all positive. 
More formally, 
we can convert the first set of coordinates into homogeneous coordinates and then
use lemma~\ref{l:origin} to find true polar coordinates.
So, without loss of generality, we can take $\boldsymbol{r} > 0$.

We have seen in lemma~\ref{l:simplex8}, that for one twisted graph $T$,
and for some $\boldsymbol{\theta}$, that $\Sigma(T)$ is a minimal insoluble
system. We ask whether this is generally the case.
\begin{conj}
\label{c:unsolvable}
For every twisted graph $T$, there is some $\boldsymbol{\theta}$
respecting $\sjle$,
for which $\Sigma(T) \boldsymbol{r} > 0$ is not soluble.
\end{conj}

To simplify statements of such conjectures and theorems, we introduce the 
term {\em constraining}:
\begin{defn}
A twisted graph is {\em constraining} if there is some $\boldsymbol{\theta}$
respecting $\sjle$,
for which $\Sigma(T) \boldsymbol{r} > 0$ is not soluble, and for
which $\Sigma(T) \boldsymbol{r} = 0$ has a nontrivial solution.
\end{defn}

Conjecture \ref{c:unsolvable}  can be strengthened to
\begin{conj}
Every irredundant $\sjle$-minimal twisted graph $T=(\vec{G}, G, \sjle, A)$
is constraining.
\end{conj}
where $\sjle$-minimal is defined as:
\begin{defn}
A twisted graph $T=(\vec{G}, G, \sjle, A)$ is $\sjle$-minimal
if for each $e \in E$, there is no $A'$ such that $T=(\vec{G}\setminus e, G\setminus e, \sjle \setminus e, A')$ is a twisted graph.
\end{defn}
If these conjectures are true, then we would 
also have (by lemma~\ref{l:twistability}), the weaker:
\begin{conj}
For every minimal simply three edge connected, totally cyclic
directed graph $\vec{G}$, there is a total order $<$,
and a constraining twisted graph, $T=(\vec{G}, G, <, A)$.
\end{conj}
and the, weaker still
\begin{conj}
For every minimal three edge connected simple
graph $G$, there is a directed graph $\vec{G}$, total order $<$,
and a constraining twisted graph, $T=(\vec{G}, G, <, A)$.
\end{conj}

We will find an initial class of constraining twisted graphs,
including
all three edge connected cubic graphs.

We will see that the somewhat surprising proofs
of lemmas~\ref{l:nequals} and~\ref{l:nnegative} are general, 
and not sporadic facts about
the non-Pappus arrangement.

\section{A Representation of $\Mat(<)$ and $\Mat(\vec{G})$}
\label{s:om2}
We continue, our study of twisted graphs, by a simultaneous
representation of both the oriented matroids of interest,
within the vector space in $\Real^n$, where $n = |E|$.
We identify $E$ with the numbers $1 \ldots n$, such that
$e_i < e_j$ if and only if $i < j$.
We take $\boldsymbol{\theta}$ to be an increasing sequence
of values in $(0,180)^n$, rather than a sequence of formal variables:
although we will once or twice, consider the effect of modifying some
of these values, we are, in the main, considering $\boldsymbol{\theta}$  
as a given vector.

Each circuit and cocircuit of $\Mat(<)$ is represented as a non-zero point
in $\Real^n$.
Each vector and covector are represented by
the convex cone  formed from 
non-negative combinations of the rays corresponding to the 
conforming circuits or cocircuits. Thus, a circuit, when considered as
a vector, is represented by a ray from the origin through the point corresponding
to the circuit.

The strong map from  $\DualMat(\vec{G})$ to  $\Mat(<)$,
then allows each of the cocircuits of $\Mat(\vec{G})$ to be
represented by the cone representing it when considered
as a vector of $\Mat(<)$.

For circuits of the form $+C_{i,j,k}=(\{i,k\},\{j\})$, with $1 \leq i<j<k \leq n$,
the corresponding  point is given by:
\begin{align}
\widehat{+C_{i,j,k}} &= \begin{pmatrix}
x_1 & x_2 & \cdots & x_n
\end{pmatrix}\\
x_h &=  \begin{cases}
\LSN{j}{k} & h = i \\
-\LSN{i}{k} & h = j \\
\LSN{i}{j} & h = k \\
0 & \text{otherwise}
\end{cases} 
\end{align}

We represent the opposite circuits by the opposite points.

For cocircuits of the form $+C^*_i = ( \{ j: 1 \leq j < i \}, \{j : n \geq j > i\})$,
with $1 \leq i \leq n$, the corresponding point is given by
\begin{equation}
\widehat{+C^*_i} = \begin{pmatrix}
\SN{1}{i}  & \SN{2}{i} & \cdots & \SN{n}{i}
\end{pmatrix} 
\end{equation}
Again, the opposites are represented
by the opposite points.

Immediately from proposition~\ref{p:normalize}, we have the following:
\begin{prop}
For any circuit $C \in \Circ(<)$ and any cocircuit $C^* \in \CoC(<)$,
$\hat{C} \perp \hat{C^*}$.
\end{prop}

We can represent 
any set $X$ of circuits or cocircuits as a non-negative sum of 
the representations of its members:
\begin{equation}
\hat{X} = \left\{ \sum_{x \in X}\lambda_x \hat{x} : \lambda_x \geq 0 \right\}
\end{equation}
so that if $X$ is symmetric (i.e. $x \in X$ if and only if $-x \in X$),
$\hat{X}$ is a vector space.

The basic structure of the representation is then 
captured in the following proposition:
\begin{prop}
\label{p:lessthan}
With $\hat{}$ as defined above:
\begin{enumerate}
\item
$\widehat{\Circ(<)} \perp \widehat{\CoC(<)}$
\item
$\widehat{\Circ(<)}$ has rank $n-2$.
\item
$\widehat{\CoC(<)}$ has rank $2$.
\end{enumerate}
\end{prop}
\begin{proof}
The first point follows from the previous proposition.

It then suffices to show $n-2$ independent members of $\widehat{\Circ(<)}$,
and $2$ independent members of $\widehat{\CoC(<)}$, since
$\rank \widehat{\Circ(<)} + \rank \widehat{\CoC(<)} \leq n$.

The set $\{ \widehat{C_{i,n-1,n}} : 1 \leq i \leq n-2 \}$
is independent, by considering the first $n-2$ coordinates.

The set $\{ \widehat{C^*_1}, \widehat{C^*_2}\}$ is independent
by considering the first $2$ coordinates.
\end{proof}

For a partial order $\sjle$, and for any $C \in \Circ(\sjle)$, 
$\hat{C}$ is independent of the total order $<$ extending $\sjle$ used,
so that $\hat{C}$ and $\widehat{\Circ(\sjle)}$ are well-defined. 
In addition, $\widehat{\Circ(\sjle)} \subset \widehat{\Circ(<)}$, so that:
\begin{prop}
\label{p:partialLessThan}
For a partial order,
$\widehat{\Circ(\sjle)}$
has rank at most $n-2$.
\end{prop}

We tie this back in with the previous section:
\begin{prop}
If $T=(\vec{G}, G, \sjle, A)$ is a twisted graph, then $\Sigma(T) = \hat{A}$.
\end{prop}

\section{Twisted Graphs with Simplexes}
\label{s:twistedsimplex}
The key interesting property of the twisted graph that we studied in section~\ref{s:mainproof},
is that the corresponding matrix $\Sigma$ contains a simplex.
We will now study such twisted graphs, generalising the treatment in section~\ref{s:mainproof},
and providing a condition sufficient for a twisted graph to have such a simplex.
This provides an alternative proof of the main theorem.

\begin{defn}
A twisted graph $T= (\vec{G}, G, \sjle, A)$ is {\em simplicial} on $F \subset E$, 
if the submatrix of the  system $\Sigma = \Sigma(T)$ formed from the columns indexed by $F$
is a simplex, for every $\boldsymbol{\theta}$ respecting $\sjle$.
\end{defn}

We note that in this case $|F| = |A| -1$.

As an example, consider the twisted graph $T$ corresponding to $M_8$, equation  (\ref{e:m8}): 
lemma~\ref{l:simplex8} shows that $T$ is simplicial on $\{2,3,4,6,8,9,10\}$.

If there is a linear dependency between the rows of $\Sigma$ then there is only one (up to multiplication)
and it is the unique positive dependency (restricted to $F$),
provided by the simplex. Thus, rank $\Sigma$ is either $|A|$ or $|A|-1$.
In the example, we saw that which of these held, depends on the choice of $\boldsymbol{\theta}$.
Now, $\Sigma$ is also a set of $|A|$ vectors from $\widehat{\Circ(\sjle)}$, which has rank 
at most $|E| -2$,
hence:

\begin{prop}
For a simplicial twisted graph $T= (\vec{G}, G, \sjle, A)$, $|A| \leq |E| -2 $.
\end{prop}

Now, consider the set $F^c=E \setminus F$. This has at least three members.
We want to look at the restriction of $\widehat{\Circ(\sjle)}$ to $F^c$,
i.e. those which are zero on all members of $F$.
In the conditions on the main theorem, we saw that the partial order
being used totally orders $F^c$. This motivates:
\begin{defn}
A twisted graph $T= (\vec{G}, G, \sjle, A)$ is {\em strictly simplicial}
if it is simplicial on $F$ and $E \setminus F$ is totally ordered by $\sjle$.
\end{defn}

In such a case, taking $<$ as a total order extending $\sjle$:
\begin{align}
\widehat{\Circ(\sjle)}|_{F^c} &=\widehat{\Circ(<)}|_{F^c}\\
&= \widehat{\Circ(<)} \cap \Real ^ {F^c}\\
&= \{ c \in \widehat{\Circ(<)} : c(f)=0 \text{ for all } f \in F \}\\
&= \{ c \in \Real ^ {F^c} : c \perp \widehat{\CoC(<)} \}
\end{align}
the last of these, follows from proposition~\ref{p:lessthan}, and we conclude:
\begin{prop}
$\widehat{\Circ(\sjle)}|_{F^c}$ has rank $|F^c|-2$.
\end{prop}

So for a strictly simplicial twisted graph,
we can take a basis $\{ b_1, b_2, \ldots b_k \}$ for $\widehat{\Circ(\sjle)}|_{F^c}$,
with $k = |F^c|-2 = |E|-|F|-2$. Thus $|A|+k = |E| -1$.
Moreover, $\Sigma \cup \{ b_1, b_2, \ldots b_k \} \subset \widehat{\Circ(\sjle)}$,
and $|\Sigma \cup \{ b_1, b_2, \ldots b_k \}| = |E| -1 > |E| -2 \geq \rank \widehat{\Circ(\sjle)}$.
Thus, there is a linear dependency amongst  $\Sigma \cup \{ b_1, b_2, \ldots b_k \}$.

Also, we can choose the basis such that, for each $i$, $b_i = \widehat{C_i}$, for some circuit
$C_i$ in $\Circ(\sjle)$, with $\underline{C_i} \subset F_c$.
By taking $B$ as an appropriate subset (not necessarily proper) of $\{ C_1, C_2, \ldots C_k \}$,
we have the following:

\begin{lem}
For a twisted graph $T= (\vec{G}, G, \sjle, A)$,
strictly simplicial on $F$,
there is some subset $B \subset \Circ(\sjle) \setminus F$, such that there is
a homogeneously unique linear dependency $\Lambda$ in $\widehat{A \cup B}$ 
\end{lem}
\begin{proof}
If $\widehat{A}$ has a linear dependency, take $B = \emptyset$,
any linear dependency of $\widehat{A}$, is also a linear dependency
of the rows of the simplex, and hence homogeneously unique.

Otherwise choose a basis $B_0$ for $\widehat{\Circ(\sjle)}|_{F^c}$.
By considering the rank, we have that $A \cup B_0$ is a linearly dependent
set, hence, there is some $B \subset B_0$, such that $\widehat{A \cup B}$  is a minimal
independent set, and has a homogeneously unique linear dependency.
\end{proof}

\begin{lem}
\label{l:positiveLambda}
In the previous lemma, $B$ can be chosen such that $\Lambda > 0$.
\end{lem}
\begin{proof}
Since $\widehat{B}$ is a linearly independent set, there is some
some $a \in A$, with $\lambda_a \neq 0$. By multiplying by $-1$ if necessary, $\lambda_a > 0$.
Since  $\widehat{B} \subset \Real ^ {F^c}$, $\Lambda$ gives a linear 
dependency for the simplex indexed by $F$, thus $\lambda_{a'} > 0$ for all $a' \in A$.
For each $b \in B$, if $\lambda_b < 0$ we can replace $b$ with $-b$ to get the desired
result.
\end{proof}

Combining the previous lemma with Carver's theorem, we get:
\begin{thm}
For a strictly simplicial twisted graph $T= (\vec{G}, G, \sjle, A)$, for every $\boldsymbol{\theta}$
respecting $\sjle$,
there is
some set $B$ of circuits from $\Circ(\sjle)$ such that $\widehat{ A \cup B } . \mathbf{r} > 0$
is a minimal insoluble system.
\end{thm}
\begin{proof}
The previous lemma furnishes a positive linear dependency, showing from theorem~\ref{t:carver},
that the system is insoluble. If there were some proper insoluble subsystem, then
the same theorem would furnish a different positive linear dependency, contradicting
the uniqueness of the previous lemma.
\end{proof}

As in section~\ref{s:mainproof}, this system can be analysed in terms of determinants of
submatrices of $\widehat{ A \cup B } $ to give a set of equations and inequalities that are sufficient
for a system to be insoluable. By allowing the $\boldsymbol{\theta}$ to vary, we can find conditions sufficient
for the system $\widehat{ A \cup B } . \mathbf{r} > 0$ to be soluble, i.e. for every subset of $B$,
the conditions for insolubility do not hold. We conjecture that for any set $B$ chosen as in lemma~\ref{l:positiveLambda},
with some specific $\boldsymbol{\theta}$ respecting $\sjle$, such that the system $\widehat{ A \cup B } . \mathbf{r} > 0$ is not
soluble, there is some other $\boldsymbol{\theta'}$ which respects $\sjle$ for
which the system is soluble. This would allow us to represent any simplicial twisted graph as a line arrangement, like
figure~\ref{f:straight}, with some constraint on the angles in that arrangement.
Further investigation of this conjecture may depend on converting the various constraints, which
all amount to constraints on the signs of differences of pairs of sums of products of sines, into normal form,
using section~\ref{s:normal}, and showing that there is a simultaneous solution.

\subsection{Maximal Simplicial Twisted Graphs}

We investigate one case further, the easiest, when $|A| = |E| - 2$, so that $k=0$ or $k=1$.
The analogue of lemma~\ref{l:simplex8}, requires us to consider the signs of the determinants
of three matrices. These signs are all covariant, or contravariant. Precisely which depends on 
the detail of the relative ordering of the members of $F_c$ within $E$, since odd permutations
cause sign changes of determinants.

\begin{lem}
For a twisted graph $T= (\vec{G}, G, \sjle, A)$, strictly simplicial on $F$, with $|F| = |E| - 3$,
then there is a signed set $X$ with $\underline{X} = F_c$, such that
for each $\boldsymbol{\theta}$ respecting $\sjle$, 
there is a sign $\sigma(\boldsymbol{\theta} ) \in \{ -1, 0, +1 \}$,
and for
each $x \in E \setminus F$, the matrix formed from the $F \cup \{ x \}$ columns of $\Sigma$ has determinant
with sign $X(x) \sigma(\boldsymbol{\theta} )$. 
\end{lem}
\begin{proof}
Given a specific  $\boldsymbol{\theta}$ either $\rank \Sigma = |A| -1$ or $\rank \Sigma = |A|$.
In the first case, each of the three matrices is singular, and   $\sigma(\boldsymbol{\theta} )=0$.

In the second case, we write $F_c = \{ x_1, x_2, x_3 \}$ with $x_1 \sjle x_2 \sjle x_3$;
we write $C = ( \{ x_1, x_3 \}, \{ x_2 \} )$.
We find the circuit $C'$ and $B = \{ C' \}$ satisfying lemma~\ref{l:positiveLambda}.
Either $C' = C$ or $C'= -C$.

We set:
\begin{equation}
\sigma(\boldsymbol{\theta} ) = \begin{cases}
1 & C' = C \\
-1 & C' = -C
\end{cases}
\end{equation}

For each $x \in F_c$, the positive dependency $\Lambda$ from that lemma is
also a positive dependency of the $|A|$ by $|A|+1$ submatrix $\Sigma'$ 
of $\widehat{A \cup \{ C' \}}$,
formed from the columns $F \cup \{ x \}$. 
Moreover, it is unique, because the last row, $\widehat{C'}$, of the submatrix,
is zero on $F$, and has only the one non-zero entry at $x$, and we know that the
other rows and columns have a unique positive dependency, being a simplex.

Thus,  $\Sigma'$ is a simplex, and the signs of the subdeterminants alternate.
The signs of the first $|A|$ subdeterminants are given by the sign
of $C'(x)$, the position of $x$ with respect to $F$ (and hence the sign of the permutation
to move $x$ to the final column of the submatrix), and the  sign of the corresponding 
subdeterminant in the simplex from $\hat{A}$ on $F$.
We choose each $X(x)$ to satisfy the lemma for some particular $\boldsymbol{\theta}$.
For a different $\boldsymbol{\theta}$, where we also find $C' = C$ 
none of the signs
in the derivation change, so that the same $X$ satisfies. On the other hand, if we find $C' = -C$,
the sign of $C'(x)$ has changed, and so has the sign of $\sigma(\boldsymbol{\theta} )$, so that
once again $X$ satisfies the lemma.
\end{proof}

We conjecture that in all such cases, $\sigma$ is surjective, which allows us to draw a line arrangement
corresponding to both variations of $A$ and $C$ or $-C$, and a third arrangement in which each $a \in A$ 
corresponds to a point in the drawing.

\begin{thm}
\label{t:simplicial}
For a twisted graph $T= (\vec{G}, G, \sjle, A)$, simplicial on $F$,
with $E = F \cupdot \{ x_1, x_2, x_3 \}$, with $x_1 \sjle x_2 \sjle x_3$,
and $C$ being either $( \{ x_1, x_3 \}, \{ x_2 \} )$ or its opposite,
then there is a sign $\sigma_C \in \{ 1, -1\}$, such that
$\widehat{A \cup \{  C \} }.\mathbf{r} > 0$
is soluble for $\boldsymbol{\theta}$ respecting $\sjle$ 
if and only if the sign of the determinant of
the square submatrix of $\hat{A}$ on columns $F \cup \{ x_1 \}$
is $\sigma_C$. Moreover, $\sigma_{-C} = - \sigma_C$.
\end{thm}
\begin{proof}
We suppose a specific $\boldsymbol{\theta}$, and
taking $\sigma(\boldsymbol{\theta}) $ as in the previous
lemma,
if the specified determinant is zero, then $\sigma(\boldsymbol{\theta}) = 0$,
and all the determinants on columns $F \cup \{ x_i \}$ are also zero,
and hence not equal to $\sigma_C$, moreover,
the system is insoluble, so this theorem is satisfied.

Otherwise, whether $\widehat{A \cup \{C\}}$ has a simplex or not
on  $F \cup \{ x_1 \}$
depends only on the sign of $C(x_1)$, and the sign of the determinant
of the square submatrix specified, since $\widehat{A}$ has a simplex
on $F$, and $C$ is zero on $F$.
Thus we choose $\sigma_C$ to be the sign such that this is not a simplex.

Since $\sigma_C$ depends on $C(x_1)$,  which is not zero,
we have that $\sigma_{-C} = - \sigma_C$.
\end{proof}

\subsection{Positive Sequences}

So, given that simplicial twisted graphs have this interesting property, how can we
easily tell if a twisted graph is simplicial, other than the laborious computation
of determinants that we did in section~\ref{s:mainproof}.
If $G$ is cubic, then a very quick glance at the matrix, is sufficient to see that
any spanning tree of $G$ gives a simplex, for example, we consider 
an orientation of $K_4$, to get:
\begin{equation}
\begin{pmatrix}
+&-&+&&&\\
-&&&+&-&\\
&+&&-&&+\\
&&-&&+&-
\end{pmatrix}
\end{equation}
This functions both as the incidence matrix of the directed graph, and the matrix
corresponding to the twisted graph, if we interpret the signs as representing signed values $\SN{i}{j}$.
If we consider the spanning tree formed from the first three edges, we see that in any linear dependency
$\Lambda$
amongst the rows of the 3 by 4 submatrix, the second, third and fourth rows must each have the same
sign as the first, by considering the first, second and third columns respectively. For each,
the ratio between the values of $\lambda_i$ is the inverse of the ratio of the entries
in the matrix. Since the submatrix has
more rows than columns, there must be at least one linear dependency. Since the signs of each row are
the same as the first row, this is a positive linear dependency. Since the ratios are fixed, it is homogenously
unique.

We can generalise this argument as follows.

\begin{defn}
\label{d:positivesequence}
Given a twisted graph $T= (\vec{G}, G, \sjle, A)$, then a {\em positive sequence}
for $T$ on $F$
is an ordering $A = \{ a_1, a_2, \ldots a_m \}$ and an ordered subset of $F \subset E$,
$F = \{  f_2, \dots f_{m} \}$ with $m = |A|$, such that for each $j$, $2 \leq j \leq m$:
\begin{itemize}
\item
$a_{j}(f_j) \neq 0$.
\item
For all $i > j$, $a_i(f_j) = 0$.
\item
For all $i < j$, $a_{i}(f_j) \in \{ 0, -a_{j}(f_j) \}$.
\end{itemize}
\end{defn}
Note, the absence of $f_1$.

The ordering $f_2, f_3 \ldots$ is not usually related to the order $\sjle$.

\begin{prop}
If $G$ is cubic, and has a spanning tree $F$, then
for any twisted graph $T= (\vec{G}, G, <, A)$
has a positive sequence on $F$.
\end{prop}
\begin{proof}
Consider any sequence such that the initial
subsequences form a connected subgraph of $F$.
\end{proof}

\begin{prop}
For any positive sequence as in definition~\ref{d:positivesequence},
for every $2 \leq j \leq m$, there is an $i < j$, such that
$a_{i}(f_j) =-a_{j}(f_j) \}$.
\end{prop}
\begin{proof}
From proposition~\ref{p:twisted} there is some $a_i$ such that $a_i(f_j) = -a_j(f_j)$.
From the definition of positive sequence, we then have $i<j$.
\end{proof}

We now show that a twisted graph with a positive sequence is simplicial.

\begin{thm}
A twisted graph $T= (\vec{G}, G, \sjle, A)$, with a positive sequence
is simplicial.
\end{thm}
\begin{proof}
Take $A = \{ a_1, a_2, \ldots a_m \}$ and
$F = \{ f_2, f_3, \dots f_{m} \} \subset E$,
as the positive sequence.

Consider the submatrix of $\hat{A}$ formed from the columns
corresponding to $F$. There is at least one linear dependency
between its rows, since it has more rows than columns, call
this $\Lambda = \lambda_i$, where the subscripts correspond to
the positive sequence on $A$.

From the definition, for each $j>1$, we have 
\begin{align}
\sum_{i=1}^m  \lambda_i \hat{a_i}(f_j) &= 0 \\
\label{e:pos-seq}
\sum_{i=1}^j  \lambda_i \hat{a_i}(f_j) &= 0 \\
\label{e:pos-seq2}
\lambda_j &= - \sum_{i=1}^{j-1}  \lambda_i \hat{a_i}(f_{j}) / \hat{a_j}(f_{j})
\end{align}

By multiplying by a constant we can assume that the first non-zero
value in $\Lambda$ is 1.

Suppose that the first non-zero value is $\lambda_j$ with $j>1$.
Then equation (\ref{e:pos-seq2}) gives a contradiction.
Thus $\lambda_1 = 1$.

Now, suppose the first non-positive value is  $\lambda_j$.
Then equation (\ref{e:pos-seq2}) gives a contradiction.
Moreover, the same equation shows that the positive dependency is unique.
Thus the matrix is a simplex, and we are done.
\end{proof}

We can now give a shorter proof for the main theorem.
While the ground work for this proof was extensive, the style of the proof
can be reused, like final polynomials, but the positive sequence is easier to verify.
If using this within a non-stretchability proof, the steps corresponding
to lemma~\ref{l:rin9} are also required.

\begin{proof}[Second proof of theorem~\ref{t:main}]
\label{p:second}
Consider the twisted graph $T$ corresponding to figure~\ref{f:digraph},
and $M_8$ of~(\ref{e:m8}).
We can construct a positive sequence from the eight rows in order,
and columns $\{ 8,2,6,3,4,9,10 \}$. Hence, $T$ is simplicial.
Since $\{ 1, 5, 7 \}$ is totally ordered by the specified partial order,
 $T$ is strictly simplicial.
We apply theorem
\ref{t:simplicial}, looking at figure~\ref{f:straight}. This shows that
the system corresponding to~(\ref{e:linear-again}) is soluble,
when $\boldsymbol{\theta} = ( 42, 20, 62, 80, 120, 98, 158, 125, 149, 170 )$. 
The angles
are approximate, measured from the figure.

The determinant of the 8-by-8 submatrix on columns
$\{ 2,3,4,5,6,8,9,10 \}$, is given by $d_5$, equation~(\ref{e:d5}).
In figure~\ref{f:straight} this evaluates to approximately $-0.17$, which is
negative, hence the value $\sigma$ in theorem
\ref{t:simplicial} is -1. Factoring out $\SN{1}{7}\SN{1}{7}\SN{1}{5}$ completes the proof.
\end{proof}

A variant of this proof more suited to automation would consider one of the other
determinants to provide $\sigma$, rather than a drawing.

\section{Pseudoline Stretching and Further Directions}
\label{s:future}

We have studied the realizability of a particular oriented
matroid by considering only some of the values from its 
chirotope. This, of course, is also the case with the method
of final polynomials.

This suggests that we may fruitfully study {\em partial} oriented
matroids. An approach to axiomatising them for the rank 3 case,
is to fix a line at infinity, and then provide
an abstraction of a Euclidean line arrangement, in terms
of the ordering of the angles of the lines, and
the orientation (positive, cocurrent or negative)
of some of the sets of three lines. We sketch such an approach.

\subsection{Euclidean Arrangements}
We define a combinatoric object, a Euclidean arrangement,
$\Arr$
over a set $X$, preordered by $\jle$, in terms of a set of
ordered triples $\PA \subset X^3$. 
\begin{defn}
A set of triples $A$ is alternating when $(x,y,z) \in A$ if and only if $(y,z,x) \in A$.
\end{defn}
This corresponds to the alternating nature of a chirotope.
We record negative values from the chirotope with an odd permutation:
\begin{defn}
If $A$ is a set of triples, then $-A$ is given by:
\begin{equation}
-A = \left\{ (y,x,z) : (x,y,z) \in A \right\}
\end{equation}
\end{defn}
We also write:
\begin{align}
\ZA &= \PA \cap \NA \\
\PPA &= \PA \setminus \ZA \\
\NNA &= \NA \setminus \ZA 
\end{align}
The preorder $\jle$ is not strict, so it is helpful to define:
\begin{defn}
$x, y \in X$ are {\em parallel\footnote {
This usage corresponds to Euclid's usage, and
not to the usage in the oriented matroid literature.}}, 
if $x \jle y$ and
$y \jle x$. 
\end{defn}
We write $x \parallel y$ in this case, and
$x \sjle y$ to mean: $x \jle y$ but not $y \jle x$.
We also use $\parallel$ and $\sjle$ to mean the corresponding subsets of $X^2$.

\begin{defn}
\label{d:arrangement}
Given a set $X$, a transitive and reflexive relation
$\jle$ over $X$, and an alternating set of triples $\PA$ from $X$, 
we can form the Euclidean arrangement
$\Arr = \left( X, \jle, \PA \right)$,
when:
\begin{description}
\item[\axCompare]
For all $(x,y,z) \in \PA$, either
$x \jle y$ or $y \jle x$.
\item[\axZero]
For all $x,y \in X$ with $x \jle y$
then
$(x,x,y), (x,y,y) \in \PA$.
\item[\axDegenerate]
If $(x,y,z) \in \PA$ and $x \parallel y \parallel z$
then $(y,x,z) \in \PA$.
\item[\axMonotone]
Each increasing sequence in $X^4$ is monotonic,
(see definition~\ref{d:monotonic-seq}).
\end{description}
\end{defn}

The notion of monotonic used in the above definition
is a variation of that from~\cite{felsner:sweeps}.
It corresponds to the observation that given four 
pseudolines, $\{ a, b, c, d \}$, if both
$\{ a,b,c \}$ and $\{ a,b, d \}$ are coincident,  then
so are all four lines, and more generally to 
axiom B2, from~\cite{bjorner:oriented} page 126, 
abstracting the Grassmann-Pl{\"u}cker relations.
We use the following subscript conventions, with $i$ from 1 to 4:
$x_i$ is a $\jle$-ordered sequence of elements of $X$; $t_i$
is the three element subsequence of $( x_1, x_2, x_3, x_4 )$
excluding $x_i$; $s_i$ is the three element subsequence
of $( t_1, t_2, t_3, t_4 )$
excluding $t_i$. In full:
\begin{equation}
\label{e:st}
\begin{aligned}
t_1 &= (x_2, x_3, x_4 ) &\qquad \qquad \qquad &t_2 &= (x_1, x_3, x_4 ) \\
t_3 &= (x_1, x_2, x_4 ) & &t_4 &= (x_1, x_2, x_3 ) \\
s_1 &= (t_2, t_3, t_4 ) & &s_2 &= (t_1, t_3, t_4 ) \\
s_3 &= (t_1, t_2, t_4 ) & &s_4 &= (t_1, t_2, t_3 )
\end{aligned}
\end{equation}

\begin{defn}
For a preorder $\jle$
a sequence $(x_1,x_2,\ldots,x_k) \in X^k$ is {\em $\jle$-increasing}
if $x_{i} \jle x_{i+1}$, for each $i$ from 1 to $k-1$.

In an arrangement $\Arr$, a sequence is {\em increasing}
if it is $\jle$-increasing.
\end{defn}
We note that both a sequence and a proper permutation of
the same sequence can be increasing, since two or more elements
may be parallel.
We wish to constrain the triangles formed by an increasing
sequence of four elements.
\begin{defn}
\label{d:monotonic-seq}
In an arrangement $\Arr$,
an increasing sequence $(x_1,x_2,x_3,x_4)$
is {\em monotonic}, when, with
$t_j$ and $s_k$ 
as in~(\ref{e:st}), for every$1 \leq i,j \leq 4$:
 if $s_i \in ( \PA \times \NA \times \PA) \cup  ( \NA \times \PA \times \NA )$
 then either $t_j \in \PA \cap \NA$ or $x_i \parallel x_j$.
\end{defn}

When $\PA \cup \NA = X^3$ this is a cryptomorphic form of the normal axioms
for an acyclic rank 3 oriented matroid over $X \cup \{ \omega \}$, where
$\omega$ corresponds to the line at infinity, given by $\jle$.
This motivates:
\begin{defn}
An arrangement is complete if $\PA \cup \NA = X^3$.
\end{defn}
The advantage of this axiomatization is that it does not require complete
information about the oriented matroid, but only some of the orientations.
Thus, it can be used to represent the partial arrangements such as
that in theorem~\ref{t:main}. Informally, it amounts to permitting
a chirotope $\chi$ to take values in $\left\{ +1, 0, -1, \ast \right\}$,
where $\ast$ means `unknown', and distinguishing a simple member
$\omega$
of the ground set $E$, such that if $\chi(a,b,c) \neq \ast$
then $\chi(\omega,a,b) \neq \ast$.

This allows us to relate the problem of oriented matroid
realizability to Ringel's conjecture.

We need
an appropriate notion of subarrangement,
which we define in terms of homomorphisms:
\begin{defn}
A homomorphism $\phi$ from an arrangement $\Arr = (X,\jle_A,A)$ to
an arrangement $\Brr = (Y,\jle_B,B)$, is a function $\phi:X \rightarrow Y$ such that
$\phi(\sjle_A) \subset \sjle_B$,
$\phi(\parallel_A) \subset \parallel_B$ and $\phi(A^+) \subset B^+$,
$\phi(A^0) \subset B^0$.
\end{defn}

Leading, perhaps, to:
\begin{defn}
An arrangements $\Arr = (X,\jle_A,A)$ is a subarrangement of $\Brr = (Y,\jle_B,B)$,
if there is a function $\phi:X \rightarrow Y$, such that,
for every minimal complete arrangement $\Crr = (Y,\jle_C,C) $ 
extending $\Brr$, i.e. $\jle_B \subset \jle_C$
and $B \subset C$, $\phi$ extends to a homomorphism $\phi: \Arr \rightarrow \Crr$.
\end{defn} 
In this case,
we write $\phi : \Arr \hookrightarrow \Brr$, and $\Arr \lesssim \Brr$.

We extend definition~\ref{d:respect}
\begin{defn}
\label{d:respect2}
A vector of real numbers $\boldsymbol{\theta} \in \Real^X$ {\em respects} $\jle$, when:
\begin{itemize}
\item 
For all $e \in X$, $0 < \theta_e < 180$
\item
For all $e, f \in X$, with $e \sjle f$, $\theta_e < \theta_f$.
\item
For all $e, f \in X$, with $e \parallel f$, $\theta_e = \theta_f$.
\end{itemize}
\end{defn}

A realization of an arrangement $\Arr = (X,\jle_A,A)$
is defined in terms of polar coordinates.
\begin{defn}
A {\em realization} of an arrangement $\Arr = (X,\jle_A,A)$
is a pair of vectors $\boldsymbol{r}, \boldsymbol{\theta} \in \Real^X$, such that:
\begin{itemize}
\item  $\boldsymbol{\theta}$ respects $\jle$.
\item For all $t = \{ x, y, z \} \in \PA \cup \NA$, with $x \jle y \jle z$:
\begin{itemize}
\item
if $t \in \ZA$
equation~(\ref{e:coincident}) holds.
\item
if $t \in \PPA$
inequality~(\ref{e:ptriangle}) holds.
\item
if $t \in \NNA$
inequality~(\ref{e:ntriangle}) holds.
\end{itemize}
\end{itemize}
\end{defn}
For complete arrangements, this is the same notion as
oriented matroid realizability; with a coordinate transform
from homogeneous coordinates into polar coordinates.

We then consider only the angular realization space:

\begin{defn}
The realization space $\mathbf{\Theta}(\Arr) \subset (0,180)^X$  of an arrangement $\Arr$,
is given by:
\begin{equation}
\mathbf{\Theta}(\Arr) = 
\left\{  \boldsymbol{\theta} :  \text{there is } \boldsymbol{r} \text{ such that } 
\boldsymbol{r}, \boldsymbol{\theta}  
           \text{ is a realization of } \Arr \right\}
\end{equation}
\end{defn}

\begin{defn}
Given two arrangements $\Arr = (X,\jle_A,A)$ and $\Brr = (Y,\jle_B,B)$,
with $\phi: \Brr \hookrightarrow \Arr$,
the realization space of $\Brr$
in $\Arr$ via $\phi$
$\mathbf{\Theta}(\Brr ; \Arr ; \phi) \subset (0,180)^X$ is given by:
\begin{equation}
\mathbf{\Theta}(\Brr ; \Arr ; \phi) = \left\{ 
\boldsymbol{\theta} \in (0,180)^X : 
\begin{aligned}
&\boldsymbol{\theta} \text{ respects } \jle_A, \\
&\exists \boldsymbol{\theta'}   \in \mathbf{\Theta}(\Brr) \\
& \qquad \text{ s.t. for every } y \in Y,
\theta'_y  = \theta_{\phi(y)}
\end{aligned}
\right\}
\end{equation}
\end{defn}
that is, the realization space of $\Brr$ in $\Arr$ corresponds to those realizations
of $\Brr$ that respect $\jle_A$.

We define the notion of a minimal counterexample
to Ringel's conjecture by:
\begin{defn}
An arrangement $\Arr$ is minimally angle constraining if 
there is $\boldsymbol{\theta} \notin 
\mathbf{\Theta}( \Arr )$,
for which, for all $\Brr \lesssim \Arr$,
and $\Brr \neq \Arr$
$\boldsymbol{\theta} \in \mathbf{\Theta}( \Brr; \Arr )$
\end{defn}

The following conjecture depends on the (unknown)
well-foundedness of $\lesssim$.

\begin{conj}
\label{c:intersection}
The realization space of an arrangement $\Arr$
is the intersection of the realization spaces in $\Arr$ of all its 
minimally angle constraining subarrangements.
\end{conj}

In these terms, we have seen that many twisted graphs
give rise to angle constraining arrangements. Generalizing:
\begin{conj}
\label{c:mac}
An arrangement is minimally angle constraining if, and only if,
it corresponds to a $<$-minimal twisted graph, for some $<$.
\end{conj}

\subsection{Pseudoline Stretching}

Given conjecture~\ref{c:mac}, a combinatoric analysis of a rank 3 oriented matroid,
could start by chosing an arbitrary pseudoline as the line at infinity,
and then find all such twisted graphs, in the resulting projection. 
These could be analysed as in this paper,
to derive systems of equalities and inequalities between sums of products
of sines. These could then be analysed using the techniques of 
section~\ref{s:normal}.

In this way, it is hoped that, an algorithm can be developed that either:
\begin{itemize}
\item
Finds a contradiction between the constraints
placed on the angles, hencing proving nonrealizability
\item
Or finds specific angles satisfying all the constraints,
hence proving realizability (since the resulting linear
program in $\boldsymbol{r}$ is necessarily soluble).
\end{itemize}

We further conjecture that:
\begin{conj}
A rank 3 oriented matroid over $n$ has a biquadratic final polynomial
if and only if, there is some choice of a line at infinity,
such that the remaining $n-1$ pseudolines form an arrangment $\Arr$
where there is a  subarrangement $\Brr$, which is minimally
angle constraining and for which 
$\mathbf{\Theta}(\Brr ; \Arr )$ is empty,
i.e. only one item in the intersection in conjecture~\ref{c:intersection}
need be considered.
\end{conj}
We have seen in this paper that this is the case for $\Rin$.
For the cases for which no biquadratic final polynomial exists,
then more than one item would be relevant.

\subsection{Isotopy Problems}
The above discussion, also raises the hope that
these techniques may help explain 
oriented matroids with non-isotopic realizations.
This suggests a study of $\Omega_{14}^+$ and $\Omega_{14}^-$ from
 \cite{richter:two} 
to consider both the isotopy question, and the question of
non-realizable oriented matroids
with no biquadratic final polynomial.

\subsection{Drawing With Straight Edge Alone}
\cite{bjorner:oriented}, page 364, relate the isotopy issue with {\em projective
constructive sequences}.
In terms of the Euclidean plane, these correspond to the
drawings that can be done with straight edge alone: i.e.
a simpler version of the classic straight edge and compass 
problem.
When expressed in polar coordinates, this, once again,
has an elegant form, and may be more amenable to analysis
than the projective version.

\subsection{Rank $> 3$}
It is tempting to want to generalize all of the above
to higher rank. Methodologically, I think it is better
to concentrate on addressing rank 3 problems, where
our geometric intuition is more helpful.
The realizability
problem for higher rank oriented matroids,
can be reduced to ETR, which is equivalent
t
\section{Final Conclusion}
Twisted graphs explain all the examples of
counter-examples to Ringel's slope conjecture,  given at the start of this paper.
Figs~\ref{f:circsaw3} and~\ref{f:circsaw4ringel}
are both twistings of $K_4$; circular saws~\cite{carroll:saws} such as 
fig.~\ref{f:circsaw4} are twistings of wheel graphs $W_n$; fig.~\ref{f:main} has been
discussed in depth. 
Ceva's theorem, fig.~\ref{f:ceva}, is a limiting case, explicable from
the uniform variant, fig.~\ref{f:circsaw3}. This points to the possibility
of a modified statement of Ringel's conjecture, which gives a complete
account of the slope constraints in a line arrangement, in terms of twisted graphs.

The key step in this paper was at the beginning, in section~\ref{s:polar},
with the choice to use polar coordinates for lines.
With this choice, it is clear that for realizable rank 3 oriented matroids,
and an appropriate choice of angles, that a realization can be found
by linear programming, (over the reals, which is harder than over the rationals).

Assuming some of the conjectures stated, it may be possible
to also use the same technique for automatically finding non-realizability proofs,
for all non-realizable cases. A specific example of this, the non-Pappus oriented 
matroid, has been explored in depth. 
If this does indeed generalise, then this would also furnish appropriate angles
in the realizable cases.
For oriented matroids where the techniques of this paper are applicable, 
the specific proof for that oriented matroid is short,
like the second proof of the main theorem, found on page~\pageref{p:second}.

$\ldots$ and Pappus was right.

o rank 3 realizability, \cite{mnev:universality}. I hope that, with
an appropriate approach, this can be seen clearly from
within oriented matroid theory, rather than stepping
out to ETR.
In particular, the approach taken to proving the
non-realizability of $\Rin$, starts by fixing a line
at infinity, and a realization of the rank 2 quotient
oriented matroid along that line.
Maybe, an approach to rank $n$ realizability that 
builds on a realization of a rank $n-1$ quotient
can be made to work in general.

\subsection{A new axiomatization of the Euclidean plane}
Line arrangements are pseudoline arrangements with an
additional constraint. This paper points to the possibility
of being able to formalize those constraints in terms
of twisted graphs. With considerable effort, it may be possible
to provide a cryptomorphic axiomatization of the Euclidean
plane in terms of infinite pseudoline arrangements satisfying
those constraints.

The value in this, would be that aspects of the plane that have
been neglected, because of an historical emphasis on distance
at the expense of angle, may become more apparent.

\subsection{Trigonometric Identities}
We gave a complete account of a fairly large class
of trigonometric identities.
With the additional fact that $\sin(90)=1$, these,
at first glance, appear to generate many
of the conventional identities that we learnt at school.

Thus:
\begin{conj}
All finite trigonometric identities
involving products, sums, differences and quotients
of
the sines and cosines
of rational combinations $\sum_{i=1}^k q_i \theta_i + q_0$
of
a set of unknown angles can be derived from
equation~(\ref{e:normalize}) and $\sin(90)=1$.
i.e. the kernel of the corresponding system is generated
by these two equations.
\end{conj}

Noting that the identities appeared as determinants
in our derivations, we may also ask which
of these identities can be derived from the analysis of
twisted graphs.

\section{Acknowledgements}
Professionally, I particularly thank
the anonymous referee, who gave a damning review of~\cite{carroll:drawing},
pointing out that I had not read the literature, specifically~\cite{shor:stretchability}. I hope
that this document addresses your concerns. I also thank Peter Cameron's 
Combinatorial Study Group and Hewlett-Packard, who funded the initially research
into Venn triangles.

At a personal level, there are many too many friends and family to thank for their
support over the many years of this research. I single out Chiara Menchini as
having been particularly supportive over the whole period.
I also 
thank specifically Janet Sherry and Dr Alex Tsilegkeridis.

\bibliographystyle{apalike}
\bibliography{pappus}

\end{document}